\documentclass[reqno,10pt]{amsart}
\usepackage{amsmath}
\usepackage{amssymb}
\usepackage{amsfonts,dsfont}
\usepackage{comment}
\usepackage{graphicx}
\usepackage[abs]{overpic}
\usepackage{caption}
\usepackage{float}
\usepackage{physics}
\usepackage{esint} 
\usepackage{braket} 
\usepackage{tikz}

\usepackage{xcolor}
\definecolor{blue_links}{RGB}{13,0,180} 

\usepackage{hyperref}
\hypersetup{
	colorlinks=true, 
	linktoc=all,     
	linkcolor=blue_links,  
	citecolor=blue_links,
	urlcolor=blue_links,
}

\newcommand{\der}[1]{{\rm{d}}#1}
\newcommand{\derx}{\der{x}}
\newcommand{\ders}{\der{s}}

\newcommand{\hn}{_{h_n}}
\newcommand{\hen}[1]{U^{\mathcal{T}}_{n}(\{#1\neq 0\})}

\newcommand{\en}{_{\varepsilon_n}}

\newcommand{\intnormal}[1]{\int_\Omega|\nabla #1|^2\,\derx}

\newcommand{\bruchdiscn}{\Gamma^{m}_\varepsilon}

\newcommand{\bruchdisc}{\Gamma_\varepsilon}
\newcommand{\bruchdische}[1]{U^{\mathcal{T}}_\varepsilon(\Gamma_\varepsilon(#1))}
\newcommand{\bruchdiscepsilonn}{\Gamma_{n}}
\newcommand{\bruchdiscepsilonnhe}[1]{U^{\mathcal{T}}_n(\Gamma_{n}(#1))}
\newcommand{\bruchmengedische}[1]{\frac{1}{2\varepsilon}\L^d\big(\bruchdische{#1}\big)}
\newcommand{\bruchmengedischen}[1]{\frac{1}{2\varepsilon_n}\L^d\big(\bruchdiscepsilonnhe{#1}\big)}
\newcommand{\truncation}[1]{\mathds{1}_{(Q_n(#1))^c}}

\newcommand{\energymain}[1]{\E_\varepsilon(u_\varepsilon(#1), \gamma_\varepsilon(#1), \Gamma_\varepsilon(#1))}

\newcommand{\energieeps}[1]{\intnormal{u_\varepsilon(#1)-\gamma_\varepsilon(#1)}+\bruchmengedische{#1}}

\renewcommand{\L}{\mathcal{L}}
\renewcommand{\H}{\mathcal{H}}

\newcommand{\T}{\mathcal{T}}

\newcommand{\E}{\mathcal{E}}
\newcommand{\Uhnen}{U^{\mathcal{T}}_n}

\usepackage{enumitem}		

\usepackage[backend=biber, sorting = nyt]{biblatex}		

\renewbibmacro{in:}{%
	\ifentrytype{article}{}{\printtext{\bibstring{in}\intitlepunct}}}
\bibliography{references}

\AtEveryBibitem{%
	\clearfield{issn} 
	\clearfield{doi} 
    \clearfield{isbn}
    \clearfield{month}
	
	\ifentrytype{online}{}{
		\clearfield{url}
	}
}

\usepackage[capitalise,noabbrev]{cleveref}	

 \usepackage[final]{showlabels}							

\usepackage{tikz}
\usetikzlibrary{decorations.pathreplacing,calc}
\def\centerarc[#1](#2)(#3:#4:#5)
{ \draw[#1] ($(#2)+({#5*cos(#3)},{#5*sin(#3)})$) arc (#3:#4:#5); }

\usepackage{mathtools}

\newtheorem{theorem}{Theorem}[section]
\newtheorem{lemma}[theorem]{Lemma}

\newtheorem{corollary}[theorem]{Corollary}
\newtheorem{definition}[theorem]{Definition}

\newtheorem*{theorem*}{Theorem}

\newcommand{\N}{\mathbb{N}}
\newcommand{\Z}{\mathbb{Z}}
\newcommand{\R}{\mathbb{R}}

\newcommand{\EEE}{\color{black}} 
 
\def\eps{\varepsilon}

\def\n{\mathbf{n}}

\def\XXint#1#2#3{{\setbox0=\hbox{$#1{#2#3}{\int}$}
\vcenter{\hbox{$#2#3$}}\kern-.5\wd0}}

\usepackage{color}

\numberwithin{equation}{section}

\DeclarePairedDelimiterX{\inner}[2]{\langle}{\rangle}{#1 \, , \, #2}
\DeclarePairedDelimiterX{\seminorm}[1]{[}{]}{#1}

\begin{document} 

\title[Eigenfracture approximation of quasi-static crack growth]{Eigenfracture approximation \\ of quasi-static crack growth in brittle materials}

\author[B.~Duong]{Ba Duc Duong}
\address[Ba Duc Duong]{
  Department of Mathematics \\
  Friedrich-Alexander Universit\"at Erlangen-N\"urnberg \\
  Cauerstr.~11, D-91058 Erlangen, Germany
}
\email{ba.duc.duong@fau.de}

\author[M. Friedrich]{Manuel Friedrich} 
\address[Manuel Friedrich]{Department of Mathematics, Johannes Kepler Universit\"at Linz. Altenbergerstrasse 69, 4040 Linz,
    Austria}
\email{manuel.friedrich@jku.at}

\subjclass[2020]{49J45, 70G75,   74B10,  74G65, 74R10.   }
\keywords{ Brittle materials, variational fracture, eigenfracture, free discontinuity problem, quasi-static fracture evolution. }
 
  \begin{abstract}   
We study an approximation scheme for a variational theory of quasi-static crack growth based on an eigendeformation approach. Following \cite{SchmidtFraternaliOrtiz2009}, we consider a family of energy functionals depending on a small parameter $\varepsilon$ and on two fields, the displacement field and an eigendeformation field that approximates the crack in the material. By imposing a suitable irreversibility condition and adopting an incremental minimization scheme, we define a notion of quasi-static evolution for this model. We then show that, as $\varepsilon \to 0$, these evolutions converge to a quasi-static crack evolution for the Griffith energy of brittle fracture \cite{FrancfortLarsen2003}, characterized by irreversibility, global stability, and an energy balance.
\end{abstract}

\maketitle

 

\section{Introduction}

    The last two decades have seen tremendous progress in the understanding of free-discontinuity problems and their application to fracture mechanics. In their  seminal work \cite{FrancfortMarigo1998}, \EEE {\sc Francfort} and {\sc Marigo}   proposed an evolutionary model based on the global minimization of so-called  {Griffith energies}, which feature a competition of bulk elastic contributions and an energy needed to increase the area of the cracked surface. Their model, formulated in the framework of rate-independent processes, is based on three fundamental principles: a no-healing irreversibility constraint on the crack, static equilibrium at every time, and an energy balance that ensures that the process is non-dissipative. The derivation of existence results for this model was initiated in \cite{dMasoToader2002_approximation_results} for a $2d$-antiplane model with restrictive assumptions on the topology of the crack set. The theory has then been further developed by removing such topological restrictions and adopting a weak formulation in the functional framework of $SBV$-functions \cite{AmbrosioFuscoPallara2000} or generalizations thereof. Among the vast body of literature, we only mention results in the antiplane setting \cite{FrancfortLarsen2003}, nonlinear elasticity \cite{dMasoFrancfortToader2005, dMasoGiacominiPonsiglione2009, dMasoLazzaroni2010}, and $2d$-linearized elasticity \cite{FriedrichSolombrino2018}, and we refer the reader to \cite{BourdinFrancfortMarigo2008, Francfort2022}  for an overview.

The goal of the present paper is to derive an approximation result of the above-mentioned crack  evolution  based on an eigendeformation approach to fracture \cite{SchmidtFraternaliOrtiz2009}. Several methods of approximating free-discontinuity problems with numerically more tractable models have been proposed  over the years, see \cite{Braides1998}, and their relation to Griffith energies has been analyzed by means of $\Gamma$-convergence \cite{dMaso1993}. One of the most popular computational methods are phase-field approximations \cite{AmbrosioTortorelli1990, Focardi2001, ChambolleCrismale2019}, where the sharp discontinuity is smoothed into a diffuse crack in terms of an auxiliary phase-field variable.   Without being exhaustive, we further mention some rigorous approximation results  on  \EEE finite-difference or finite-element schemes \cite{BachBraidesZeppieri2020, BelletiniCoscia1994, CrismaleScillaSolombrino2020}, nonlocal approximations \cite{MarzianiSolombrino2024, Negri2006, ScillaSolombrino2021, LussardiNegri2007, BraidesdMaso1997, CortesaniToader1999}, 
problems. \EEE or discrete finite-element approximations that make use of adaptive mesh refinements \cite{ChambolledMaso1999, BourdinChambolle2000, Negri2003, BabadjianBonhomme2023}.

\EEE

Another kind of approximation was proposed by {\sc Schmidt, Fraternali, and Ortiz}  \cite{SchmidtFraternaliOrtiz2009}, namely  an eigendeformation approach to fracture which, similar to phase-field methods, is a two-field approximation scheme. Given a small parameter $\eps > 0$, they consider an energy functional of the form
\begin{align}\label{E1}
E_\eps(u,\gamma) = \int_\Omega |\nabla u - \gamma|^2 \, {\rm d}x + \frac{\kappa}{2\varepsilon}\L^d(U_\varepsilon(\{\EEE \gamma\neq 0\}\EEE)),
\end{align}
where $\Omega \subset \R^d$ is the reference domain, $u \colon \Omega \to \R$ represents the displacement field, and $\gamma \colon \Omega \to \R^d$ an eigendeformation field approximating the crack. The eigendeformation allows the displacement field  $u$ \EEE to develop jumps at no cost in the \EEE  elastic energy, which comes at the expense of a certain amount of fracture energy measured in terms of the $\varepsilon$-neighborhood $U_\varepsilon(\cdot)$ of the support of $\gamma$. In \cite{SchmidtFraternaliOrtiz2009} it is shown that the $\Gamma$-limit as $\eps \to 0$ is given by the Griffith energy in the antiplane shear setting, namely
\begin{align}\label{E2}
E(u) = \int_\Omega |\nabla u|^2 \, {\rm d}x + \kappa\mathcal{H}^{d-1}(J_u)
\end{align}
for $u \in SBV^2(\Omega)$, where $J_u$ denotes the jump set of the function. (Strictly speaking, \cite{SchmidtFraternaliOrtiz2009} treats the vectorial case of linearized elasticity; we present a simpler case by restricting to scalar-valued displacements.)

The model has also been implemented for a time-evolution scheme in \cite{SchmidtFraternaliOrtiz2009}, and was later further considered and developed in,  e.g., \cite{VolkmannBeckSchmidt2022, PandolfiOrtiz2012, PandolfiLiOrtiz2014, StochinoQinamiKaliske2017, QinamiPandolfiKaliske2020, PandolfiWeinbergOrtiz2021}. Yet, we highlight that so far no rigorous connection between the approximation \eqref{E1} and the Griffith model \eqref{E2} for the quasi-static evolutionary framework proposed in \cite{FrancfortMarigo1998} is available. 



 \EEE

Indeed, the literature on static approximation results for the Griffith functional has largely overshadowed the analysis of evolutionary counterparts, and rigorous results in this direction are relatively scarce. We mention the phase-field approximation of quasi-static fracture evolution by {\sc Giacomini} \cite{Giacomini2005}, discontinuous finite-element approximations \cite{GiacominiPonsiglione2003, GiacominiPonsiglione2006}, as well as a discrete-to-continuous passage for an adaptive finite-element model \cite{CrismaleFriedrichSeutter2025}. We also recall related works where crack evolutions have been identified as effective variational limits of sequences of problems, such as atomistic models \cite{FriedrichSeutter2025}, homogenization \cite{GiacominiPonsiglione2006_Gamma_convergence_approach}, linearization \cite{FriedrichSteinkeStinson2025}, or a cohesive-to-brittle passage \cite{Giacomini2005_Size_effects}. The goal of the present paper is to show that such a rigorous approximation result    can also be obtained for the  eigenfracture approach.

Let us now describe our results in more detail. We follow the usual approach of   a \EEE time-discretized incremental minimization scheme with time-step size $\delta>0$. An important aspect is the implementation of a suitable irreversibility condition for the energy \eqref{E1}, \EEE which gives rise to a notion of an increasing crack set on the time-discrete level. In the present setting, this is achieved by requiring that the support of the eigendeformation field increases in time. Passing to the time-continuous limit $\delta \to 0$, we then show that there exists a quasi-static eigenfracture evolution for the energy \eqref{E1} for fixed $\eps > 0$, see Theorem~\ref{theorem:approx}. (Strictly speaking, due to compactness issues for $\gamma$ in \eqref{E1}, we additionally consider a finite-element discretization, see also Theorem \ref{theorem:griffithminNondiscrete} for a variant without such regularization.) Afterwards, we pass to the limit $\eps \to 0$ and show that the quasi-static eigenfracture evolutions converge to an irreversible quasi-static crack evolution in the sense of \cite{FrancfortMarigo1998}, see Theorem \ref{theorem:griffithmin}. More precisely, we obtain a pair $(u(t), \Gamma(t))$ of displacement fields and crack sets for $t \in [0,1]$ satisfying (a) an irreversibility condition, (b) global stability at all times (sometimes referred to as unilateral minimality), and (c) an energy balance law, see Definition \ref{main def-lim} for details. Moreover, we show convergence of the energies from \eqref{E1} to \eqref{E2} along the evolution at all times.

As a byproduct, closer to the numerical implementation in \cite{SchmidtFraternaliOrtiz2009} we also show that the limiting evolution can be achieved from the time-discretized  scheme of \eqref{E1} by  a simultaneous limit $\delta\to 0$, $\eps \to 0$, see Theorem \ref{th: simlutaneous}. Let us highlight that in \cite{SchmidtFraternaliOrtiz2009} the model was analyzed in the vectorial setting, whereas we reduce to the antiplane framework for the sake of simplicity. The vectorial version, corresponding to the approximation of the evolution derived in \cite{FriedrichSolombrino2018}, is subject of a future work.

Our results are in spirit close to the important work by {\sc Giacomini} \cite{Giacomini2005} on the Ambrosio–Tortorelli approximation of quasi-static crack evolution in the antiplane setting. The purpose of this paper is not to discuss the relevance of the eigenfracture approach compared to phase-field methods but rather to show that the variational approach to fracture is flexible enough to deal with various kinds of approximation schemes. The novelty of our work lies in developing techniques that allow us to adapt the by-now classical strategy \cite{FrancfortLarsen2003} for proving existence of crack evolutions to the setting of eigenfracture.

The most delicate and original part of the work consists in proving that, in the limiting passage $\eps \to 0$, the static equilibrium property is preserved at all times. This calls for a suitable adaptation of the jump-transfer construction of \cite{FrancfortLarsen2003}, which consists in `transferring' the jump of any competitor to that of a minimizing sequence. A fundamental role in this context is played by the $BV$ coarea formula, which allows  to fill potential small holes in the jump set to obtain, roughly speaking, a jump set locally separating the domain into two parts. (We will refer to this construction as a \emph{separating extension} of the crack set.) In \cite{Giacomini2005} this construction is adapted to the phase-field setting by cleverly using the interplay of both variables, the displacement and the phase-field variable, to obtain a separating extension. This technique, however, is finely tailored to the Ambrosio–Tortorelli functional and it seems that it cannot be transferred neither to the \EEE two-field approximation \eqref{E1} nor to phase-field models in linear elasticity.

We follow a strategy which is closer to the original proof in \cite{FrancfortLarsen2003} by  using directly  the coarea formula to construct  separating extensions. The fundamental additional difficulty stems from the fact that the crack set in \eqref{E1} is not given by a surface, but rather by a \emph{neighborhood} of a surface, described in terms of $U_\varepsilon(\{\gamma\neq 0\})$. In other words, even if the additional separating surface has small $\mathcal{H}^{d-1}$-measure, it is generally not guaranteed that the corresponding neighborhood contributes only a small amount. In this sense, our construction is related to recent adaptations of the jump transfer in spatially discrete settings, in particular concerning the definition of discrete interpolations.  However, in contrast to the 2d-results in \cite{FriedrichSeutter2025, CrismaleFriedrichSeutter2025}, we address the problem in arbitrary space dimensions, which requires significantly more delicate arguments to control the size of the neighborhoods appearing in \eqref{E1}. In particular, one cannot use an arbitrary separating extension with small additional surface, as provided by the coarea formula. Instead, a minimal one is required, in connection with a  lower density bound for this extension, see Lemma \ref{lem:minimalseparator}, and suitable covering arguments.  We believe that the techniques developed in this paper may also be useful for addressing other relevant approximation schemes in the future, in particular for extending the   phase-field evolution \cite{Giacomini2005} to linear elasticity.

Let us highlight two further relevant differences compared to  \cite{Giacomini2005}. First, in  \cite[Theorem~3.2]{Giacomini2005} separate  convergence of elastic and surface energies is established only for \emph{almost all} times. This is due to an  a priori choice of a countable subset of times $I_\infty$ from which the evolution is extended to all times.  However, such procedures generally fail to capture the behavior at times at which the surface energy exhibits discontinuities.    We instead implement a variant of this approach that allows us to prove convergence of both the elastic and the surface energy \EEE \emph{at all times}. In this sense, our evolution is related to those obtained by $\sigma^p$-convergence   \cite{dMasoFrancfortToader2005} or $\sigma$-convergence \cite{GiacominiPonsiglione2006_Gamma_convergence_approach}.  Secondly, in  \cite[Theorem 3.2]{Giacomini2005} only convergence of strains, but not of the displacements  is proved, which is related to a compactness issue. Indeed, on components that are completely disconnected by the crack set the behavior of displacements cannot be controlled. We adapt methods developed in \cite{FriedrichSteinkeStinson2025} to obtain convergence of displacements on those components of the domain which are connected to the Dirichlet boundary.

Our paper is organized as follows.  In Section \ref{sec: main results} we introduce the model and present our results. In addition to our two main results (Theorems \ref{theorem:approx} and \ref{theorem:griffithmin}), we also discuss some possible variants in Subsection \ref{sec:further}.   Section \ref{section:ExistenceForFixedEpsilon} is devoted to the existence of  a quasi-static evolution for the eigenfracture energy \eqref{E1} for fixed $\eps$. In Section \ref{chapter:Limitepsilontozero} we then perform the limit $\eps \to 0$. The most delicate part of the proof, namely the stability of unilateral minimizers, is deferred to Section \ref{section:stabilityresult}. Finally, in Section \ref{section:AlternativeSettings} we sketch the proofs for the variants of our main results. 
Since many steps of the proof follow a by-now classical strategy, not all arguments are presented in full detail. For the reader’s convenience, additional details on selected proof steps are collected in Appendix~\ref{sec:appendix}.




   \EEE

\section{The setting and main results}\label{sec: main results}

This section is devoted to the presentation of the model and the main results.

\subsection{The eigenfracture model}

 We let $ \Omega  \subset \mathbb{ R }^{ d } $ be a bounded Lipschitz domain.  We assume that $(\mathcal{T}_h)_{h > 0}$ is a family of triangulations of $\R^d$ such that $T$ is closed for all $T \in \mathcal{T}_h$ and \EEE
\begin{align}
\label{eq:bound_triangulation}
\sup \lbrace {\rm diam} \, T \colon \, T \in \mathcal{T}_h \rbrace \le h.
\end{align}
By $V_h(\Omega)$ and $W_h(\Omega)$ we  denote the corresponding finite element spaces consisting of continuous piecewise affine and of piecewise constant functions, respectively, defined on all simplices \EEE of $\mathcal{T}_h$ intersecting $\Omega$. Let $\eps>0$ be a small parameter representing the size of the nonlocality of the eigenstrain field. Given \EEE sets  $A\subset\Omega$, we define $U_\eps(A)$ as the closed \EEE $\eps$-neighborhood of $A$ and $U_{\eps,h}^{\mathcal{T}}(A)$ \EEE as the union of \EEE all simplices  intersecting $U_\eps(A)$. We consider the  energy \EEE
\begin{align}    \label{eq:energydefinitionepsilon}
E_{\varepsilon,h}(u, \gamma):= \int_\Omega Q\big(\nabla u-\gamma\big) \,\derx +\frac{\kappa}{2\varepsilon}\L^d(U_{\eps,h}^{\mathcal{T}} \EEE ( \{\gamma\neq 0\})),
\end{align}
where $u \in W^{1,1}(\Omega)$ denotes the elastic displacement and $\gamma \in L^1(\Omega;\R^d)$ denotes the eigendeformation field. The first part of the energy corresponds to the elastic part of the energy, with $Q \colon \R^d \to \R$ being a quadratic form $Q(F) = F^T  \mathbb{C} F$ for all $F \in \R^d$ and some $\mathbb{C} \in \R^{d \times d}$ such that $Q(F)   > \EEE 0$ for all $F \neq 0$. The second part represents the fracture energy, where $\kappa>0$ denotes the fracture toughness and $\lbrace \gamma \neq 0 \rbrace $  is the complement of the zero-set of the precise representative of $\gamma$. The functional \eqref{eq:energydefinitionepsilon} is a simplification of the model analyzed in  \cite{SchmidtFraternaliOrtiz2009}, by restricting to scalar-valued displacements. \EEE

In \cite[Theorem 5.1]{SchmidtFraternaliOrtiz2009} it is shown that, under the assumption $h=h(\eps)$ and 
\begin{align}\label{heps}
h(\eps)/\eps \to 0 \quad \text{ as } \eps \to 0,
\end{align}
the $\Gamma$-limit of $E_{\eps,h}$ as $\eps \to 0$ (with respect to the $L^1$-topology for $u$) \EEE  is given by the Griffith energy in the antiplane setting, namely 
\begin{align}\label{limitgama}
E(u) = \int_\Omega Q(\nabla u) \, {\rm d}x + \kappa \mathcal{H}^{d-1}(J_u) 
\end{align}
for $u \in SBV^2(\Omega)$, where $J_u$ \EEE is the jump set of $u$. As a byproduct, limits of the eigendeformation field $\gamma$ can be identified with the singular part $D^s u$   of the distributional derivative. 
In the sequel, we always assume that $h$ depends on $\eps$ and that  \eqref{heps} holds. We often drop the explicit dependence on $h$ in the notation, e.g., we write $E_\eps$ in place of $E_{\eps,h}$ or  $U^{\mathcal{T}}_\varepsilon$ in place of $U^{\mathcal{T}}_{\varepsilon,h}$. \EEE

%

We also mention that the $\Gamma$-convergence result can be complemented with boundary conditions. As is typically done for fracture problems, we impose displacement boundary conditions by considering a larger set $\Omega '$ containing $\Omega$ such that $\Omega'\setminus \overline{\Omega}$ is   also \EEE
a Lipschitz  set, and $\partial_D \Omega := \partial \Omega \cap \Omega'$ denotes the Dirichlet boundary.   Given $g\in W^{2,\infty}( \R^d)$, we denote by $g_h$ the continuous, piecewise affine interpolation of $g$ on  the triangulation $\mathcal{T}_h$.  Then, requiring $u = g_h,  \gamma = 0$ on $\Omega' \setminus \overline{\Omega}$ for admissible functions $u \in  V_h (\Omega'), \gamma\in W_h(\Omega')$ in  \eqref{eq:energydefinitionepsilon}, the corresponding $\Gamma$-limit is again of the form \eqref{limitgama}, but only finite if $u \in SBV^2(\Omega')$ satisfies $u = g$ on $\Omega' \setminus \overline{\Omega}$.  Here, we highlight that the elastic energy is still defined on $\Omega$ although the functions are defined on the larger set $\Omega'$. Moreover, for functions $u \in SBV^2(\Omega')$ satisfying $u = g$ on $\Omega' \setminus \overline{\Omega}$, the jump set satisfies $J_u \subset \Omega \cup \partial_D \Omega$, i.e., jump along the Dirichlet boundary is in principle possible and accounted for in the surface part of \eqref{limitgama}.

We mention that in this paper we consider a  finite-element discretization  of the eigendeformation approximation although the theoretical $\Gamma$-convergence results in \cite{SchmidtFraternaliOrtiz2009} primarily deal with a version of the functional without additional spatial discretization. The reason for this is the  missing compactness for the functional \eqref{eq:energydefinitionepsilon}. For more details we refer to Subsection \ref{sec:further} below, where we also suggest a variant of the approach which allows us \EEE to drop the finite-element discretization. \EEE




\subsection{Quasi-static evolution of the eigenfracture approximation}\label{sec: main1}
 
In the spirit of \cite{FrancfortMarigo1998, FrancfortLarsen2003}, we start by introducing \EEE a time-discrete evolution which is driven by time-dependent boundary conditions $g\in W^{1,1}( [0,1]; \EEE   W^{2,\infty}   ( \R^d  ))$.  Let $I_\infty$ be a dense subset of $[0, 1]$ with $0\in I_\infty$ as well as $I_m=\{0=t^m_0<t^m_1<\dots t^m_m\}\subset I_\infty$ be a finite subset of $I_\infty$ such that $\lim_{m \to \infty} \max_{i=1}^m (t^m_i-t_{i-1}^m) = 0$. We additionally assume that the sets are nested, i.e.,  $I_{m}\subset I_{m+1}$, and that  $\bigcup_{m}I_m=I_\infty$. For convenience, we introduce the shorthand  notation \EEE
\begin{equation*}
	I_{ m }^{ t }:= \left\{ \tau \in I_{m} \colon \ \tau \leq t \right\}.
\end{equation*} 
 We  set up a quasi-static evolutionary problem for fixed $\eps$, $h$, and $m$, similar to the one addressed in   \cite[Section 5.2]{SchmidtFraternaliOrtiz2009}. 

Recalling that \EEE  $g_h$ denotes the interpolation of $g$ with respect to $\mathcal{T}_h$, we define $g^{m}_{h, k} \EEE := g_h(t^m_k)$  as a shorthand notation\EEE.  We suppose that the \emph{initial pair} $(u_0^m, \gamma_0^m)$ is a solution of the minimization problem
\begin{align}
     \label{eq:mineps0}
    \min\big\{E_{\eps}(u, \gamma) \colon (u, \gamma)\in V_h(\Omega')\times W_h(\Omega'), \,  u=g_{h, 0}^m, \,  \gamma = 0 \text{ on } \Omega' \setminus \overline{\Omega}\big\} ,
   \end{align}
where $E_{\eps}$ denotes the energy given in \eqref{eq:energydefinitionepsilon}. Clearly, this solution can be chosen independently of $m$. \EEE
The subsequent steps are obtained iteratively by taking the `crack set' of the previous time steps into account. To this end, we introduce the energy
\begin{align}
    \label{eq:energydefinitionepsilon-new}
    \mathcal{E}_\varepsilon\big(u, \gamma, \Psi \big):= \int_\Omega Q\big(\nabla u-\gamma\big)\,\derx +\frac{\kappa \EEE }{2\varepsilon}\L^d\big(U^{\mathcal{T}}_\varepsilon(\Psi \cup\{\gamma\neq 0\})\big)
\end{align}
for $u\in V_h(\Omega')$ and $\gamma\in W_h(\Omega')$, where $\Psi\subset  \Omega'  $ represents the already existing support of the eigenstrain field from previous time steps. Then, given the pairs  $(u^m_j, \gamma^m_j)_{0 \le j \le i}$ and defining 
the notation 
\begin{align}\label{gammam}
\Gamma^m_\varepsilon(t):= \bigcup_{k\in \{0, \dots m\}: \,  t^m_k\in I^t_m} \EEE \{\gamma^m_k\neq 0\} \quad \text{ for all } t \in [0,1],
\end{align}
we choose the pair $(u_{i+1}^m, \gamma_{i+1}^m)$ in the next time step as minimizer of the problem 
\begin{align}
        \label{eq:minepsi+1}
    \min \big\{ \E_\varepsilon\big(u, \gamma, \Gamma^m_\varepsilon(t^m_i)\big) \colon \, (u, \gamma)\in V_h(\Omega')\times W_h(\Omega'), \, u=g_{h,i+1}^m, \,  \gamma = 0 \text{ on }  \Omega' \setminus \overline{\Omega}   \big\}.
\end{align}
Proof of  \EEE existence of minimizers for \eqref{eq:mineps0} and \eqref{eq:minepsi+1} is rather straightforward as we briefly elaborate at the beginning of Section \ref{section:ExistenceForFixedEpsilon}.  Denoting the characteristic function of $\Gamma^m_\varepsilon(t)$ by $\mathds{1}_{\Gamma^m_\varepsilon(t)}$, \EEE given the pairs  $(u^m_k, \gamma^m_k)$ for $0 \le k \le i$,  minimality implies \EEE  the identity 
\begin{align}
    \label{eq:identityyni}
    \gamma^m_i=\nabla u^m_i \mathds{1}_{\Gamma^m_\varepsilon(t_i^m)} , 
\end{align}
since by \EEE this choice of $\gamma^m_i$ the elastic energy is clearly minimized. \EEE As $m \to \infty$, the time-discrete solutions $(u^m_k)_{0 \le k \le m}$ converge to a quasi-static evolution for the eigenfracture approximation. To formulate the statement, we define admissible triples: we say that $(\bar{u},\bar{\gamma},\bar{\Gamma}) \in AD_\eps(\bar{g})$ if  $(\bar{u}, \bar{\gamma}) \in V_h(\Omega')\times W_h(\Omega') $,  $\bar{\Gamma}$  is  a  union of simplices of $\mathcal{T}_h$, and 
\begin{equation*}
  \bar{u}=\bar{g} \text{ on } \Omega' \setminus \overline{\Omega}, \quad   \bar{\gamma}=0 \text{ on } \Omega' \setminus \overline{\Omega}, \quad  \text{and} \quad  \lbrace \bar{\gamma} \neq  0  \rbrace \subset \bar{\Gamma}.  
  \end{equation*} 

\EEE

\begin{theorem}[Quasi-static eigenfracture evolution]
    \label{theorem:approx}
Let  $g\in W^{1, 1}([0, 1]; W^{2, \infty}(\R^d))$ and let $g_h\in W^{1, 1}([0, 1]; V_h(\Omega'))$ be its interpolation on $\mathcal{T}_h$. \EEE For all $\varepsilon>0$ there exists an irreversible quasi-static evolution for the eigenfracture approximation, i.e.,   a mapping $      t \mapsto (u_\varepsilon(t), \gamma_\varepsilon(t), \Gamma_\eps(t)) $ for $t \in [0,1]$   satisfying $ (u_\varepsilon(t), \gamma_\varepsilon(t), \Gamma_\eps(t)) \in AD_\eps(g_h(t))$ for all $t\in[0 ,1]$  \EEE such that \EEE   the following properties hold:  
    \begin{enumerate} 
    \item[\rm (a)] \emph{Initial condition:} $(u_\eps(0),\gamma_\eps(0),   \emptyset  )$ minimizes \eqref{eq:energydefinitionepsilon-new} among all   
$  (\bar{u}, \bar{\gamma}  ) \EEE \in V_h(\Omega')\times W_h(\Omega')$ with $\bar{u}=g_h(0)$ and  $\bar{\gamma}=0$ on $\Omega' \setminus \overline{\Omega}$.    
            \item[\rm (b)] \emph{Irreversibility:} We have $\Gamma_\varepsilon(s)\subset \Gamma_\varepsilon(t)$ \EEE for all $0\leq s\leq t\leq 1$. 
       
        \item[\rm (c)] \emph{Global stability:} For all $t\in (0, 1]$ and for all $(\bar{u}, \bar{\gamma})\in V_h(\Omega')\times W_h(\Omega')$ with  $\bar{u}=g_h(t)$ and $\bar{\gamma} = 0$ on $\Omega' \setminus \overline{\Omega}$ we have
        \begin{equation}
            \begin{aligned}
                \label{eq:3.3}
                \E_\varepsilon(u_\varepsilon(t), \gamma_\varepsilon(t), \Gamma_\varepsilon(t))\leq \E_\varepsilon(\bar{u}, \bar{\gamma}, \Gamma_\varepsilon(t)).
            \end{aligned}
        \end{equation}
       \item[\rm (d)] \emph{Energy balance:} The function $t\to \E_\varepsilon(u_\varepsilon(t), \gamma_\varepsilon(t), \Gamma_\varepsilon(t))$ is absolutely continuous and 
        \begin{align}
            \label{eq:3.4}
         \quad \quad   \E_\varepsilon(u_\varepsilon(t), \gamma_\varepsilon(t), \Gamma_\varepsilon(t))=\E_\varepsilon(u_\varepsilon(0), \gamma_\varepsilon(0), \emptyset) + 2 \int_0^t\int_\Omega(\nabla u_\varepsilon(s)-\gamma_\varepsilon(s))  \cdot \mathbb{C} \partial_t\nabla g_h(s)\,\derx\,\ders
        \end{align}
      for all $t\in [0, 1]$, where $\partial_t$ denotes the time derivative. 
    \end{enumerate}
 Moreover, this evolution is given as the limit of the time-discrete solutions $(u^m_k, \gamma^m_k)_{0 \le k \le m}$ for $m \in \N$ defined in \eqref{eq:mineps0} and \eqref{eq:minepsi+1} in the sense that, up to a subsequence (not relabeled),   it holds that, as $m \to \infty$,   
\begin{align}\label{energyconvergenceeps}
&      \E_\varepsilon(u^m_{k(m,t)}, \gamma^m_{k(m,t)}, \Gamma^m_{k(m,t)}) \to    \E_\varepsilon(u_\varepsilon(t), \gamma_\varepsilon(t), \Gamma_\varepsilon(t)),    \notag \\ 
& u^m_{k(m,t)} \to u_\eps(t) \ \text{ in } L^1(G_\eps(t)), \quad \quad   \nabla u^m_{k(m,t)}\mathds{1}_{(\Gamma^m_{k(m,t)})^c}\to \nabla u_\varepsilon(t)-\gamma_\varepsilon(t) \ \text{ in }  L^2(\Omega'),
\end{align}
 for all  $t \in [0,1]$, where $k(m,t)$ is chosen such that $t^m_{k(m,t)} \le  t  <   t^m_{k(m,t)+1}  $, \EEE and  $G_\eps(t) \subset \Omega'$   denotes the largest open set with $\Omega' \setminus \overline{\Omega} \subset G_\eps(t)$ and \EEE  $ U^\T_\varepsilon(\bruchdisc(t)) \EEE \cap G_\varepsilon(t) = \emptyset$.
 

\end{theorem}

We mention that in general the convergence of displacements in \eqref{energyconvergenceeps} cannot be guaranteed outside of $G_\eps(t)$ as $\Omega' \setminus G_\eps(t)$ corresponds to  the parts of the domain which are `broken off' by the crack set, see  \cite[Subsection 2.4]{FriedrichSteinkeStinson2025} for details.  \EEE

\subsection{Passage to a quasi-static crack evolution for a Griffith model in the antiplane setting}\label{sec: main2}

Recall the Griffith-type energy introduced in \eqref{limitgama}. Given a rectifiable set $\Gamma \subset  \Omega \cup  \partial_D \Omega$, we define the energy   
  \begin{align}\label{eq: lim-en}
  \mathcal{E}(u,\Gamma ):=   \int_{\Omega} Q(\nabla u) \,{\rm d}x+  \kappa \mathcal{H}^{d-1}(\Gamma)\,,
  \end{align}
  for each $u \in SBV^2 (\Omega') \EEE $, where the corresponding jump set $J_u$ is   subject to the constraint $J_u \, \tilde{\subset} \,   \Gamma \EEE $. (Here and in the following, $\, \tilde{\subset} \, $ stands for inclusions up to $\mathcal{H}^{d-1}$-negligible sets.)  
By $AD(\bar{g},\bar{\Gamma})$ we denote all functions $\bar{u}\in SBV^2(\Omega') \EEE$ such that   
\begin{equation*}
  \bar{u}=\bar{g} \text{ on } \Omega' \setminus \overline{\Omega}, \quad J_{\bar{u}} \,  \tilde{\subset} \, \bar{\Gamma}.
  \end{equation*}

 \begin{definition}\label{main def-lim}
  We define an \emph{irreversible quasi-static crack evolution} with respect to the boundary condition $g  \in W^{1,1}([0,1];W^{2,\infty}(\R^d)) \EEE $  as any  mapping  $t\to (u(t),\Gamma(t))$ with $\Gamma(t) \subset \Omega \cup \partial_D \Omega$ rectifiable and   $u(t) \in AD(g(t),\Gamma(t))$ for all $t \in [0,1] \EEE $ such that the following  four \EEE conditions hold:
  
\begin{itemize}
   \item[\rm (a)] \emph{Initial condition}: $u(0)$ minimizes $\mathcal{E}(u,J_u)$ given in \eqref{eq: lim-en} among all $v \in  SBV^2\EEE (\Omega')$ with
$v = g(0)$ on $\Omega' \setminus \overline{\Omega}$. 
  \item[\rm (b)] \emph{Irreversibility}:   $\Gamma(t_1) \, \tilde{\subset} \,  \Gamma(t_2)$   for all $0\leq t_1\leq t_2\leq 1\EEE $.
  \item[\rm (c)] \emph{Global stability}:  
  For every $t \in (0,1] \EEE $,   for  every  $\Phi$ \EEE with $\Gamma(t) \, \tilde{\subset} \, \Phi$, and for   every $ v  \in AD(g(t),\Phi)$ it holds that 
      \begin{equation*} 
       \mathcal{E}(u(t),\Gamma(t)) \leq \mathcal{E}( v ,\Phi)\,.
      \end{equation*}
      \item[\rm (d)]   \emph{Energy balance}:    The function $t\mapsto \mathcal{E}(u(t),\Gamma(t))$ is absolutely continuous and it holds that
      \begin{equation}\label{energybalance}
          \frac{\rm d}{{\rm d}t} \, \mathcal{E}(u(t),\Gamma(t)) =  2 \EEE \int_{\Omega}  \nabla u  (t) \cdot \mathbb{C} \EEE \nabla \partial_{t}g(t)\, {\rm d}x  \quad \text{for a.e.\ $t \in [0,1] \EEE $}. \EEE 
      \end{equation}  
\end{itemize}
\end{definition}

The existence of such an evolution has been proven in the groundbreaking work by {\sc Francfort and Larsen} \cite{FrancfortLarsen2003}  by means of time-discrete approximations, analogous to \eqref{eq:minepsi+1}. Our second theorem constitutes the main result of the paper showing that the evolutionary eigenfracture approximation converges to an evolution in the sense of Definition \ref{main def-lim}. 

 As a final preparation for its formulation, we define sets on which convergence of displacement fields can be guaranteed. For a crack set $\Gamma(t)  \subset  \Omega   \cup \EEE \partial_D\Omega $ with $\mathcal{H}^1(\Gamma(t) \EEE)<\infty$, by $B(t) \subset \Omega$ we denote the largest set of finite perimeter (with respect to set inclusion) which satisfies $  \partial^* B(t)  \cap \Omega' \EEE \, \tilde{\subset} \,   \Gamma(t) \EEE  $. \EEE This set represents the `broken off pieces', and by $G(t) := \Omega' \setminus B(t)$ instead we denote the `good set', {which}   in particular satisfies $\Omega' \setminus \overline{\Omega} \subset G(t)$. Note that  convergence of the displacements can only be expected on $G(t)$, see  \cite[Subsection 2.4]{FriedrichSteinkeStinson2025} for details.

\begin{theorem}[Approximation of quasi-static crack growth]
    \label{theorem:griffithmin}
 For all $\varepsilon >0$ with associated $h =h(\eps)>0$, we let $t\to(u_{\varepsilon}(t), \gamma_{\varepsilon}(t), \Gamma_\eps(t)) \EEE $ be the eigenfracture evolutions with boundary data $g_h$ given by Theorem~\ref{theorem:approx}, where  $g_h\in W^{1, 1}([0, 1]; V_h(\Omega'))$ denotes the interpolation on $\mathcal{T}_h$ of some $g\in W^{1, 1}([0, 1]; W^{2, \infty}( \R^d))$. \EEE
     
    Then, there exists a quasi-static evolution $t\to (u(t), \Gamma(t))$ with boundary condition $g$ in the sense of Definition \ref{main def-lim}
    and a sequence $\varepsilon_n\to 0$ with corresponding sequence $h_n := h(\eps_n)\to 0$ such that,
    setting $u_n := u_{\varepsilon_n }$, $\gamma_n:= \gamma_{\varepsilon_n}$, and $\Gamma_n:= \Gamma\en$, the following holds:  
    \begin{enumerate}[label=(\roman*)]
 \item[{\rm (i)}]  \emph{Convergence of displacements:}   For all $t\in[0, 1]$ we have $u_n(t) \mathds{1}_{G(t)}  \to u(t)\mathds{1}_{G(t)} $ in $L^1(\Omega')$.  
        \item[{\rm (ii)}]  \emph{Energy convergence:} For all $t\in[0, 1]$ we have 
        \begin{align*}
            \mathcal{E}_{\varepsilon_n}(u_n(t), \gamma_n(t), \Gamma_n(t))\to\mathcal{E}(u(t), \  \Gamma(t)).
        \end{align*}
        \item[{\rm (iii)}] \emph{Convergence of bulk and surface energy:} For all $t\in [0, 1] $  \EEE  we have 
        \begin{align*}
            & \nabla u_n(t)-\gamma_n(t)\to \nabla u(t) \text{ strongly in } L^2(\Omega; \R^d)
        \end{align*}
        and 
        \begin{align*}
            \lim_{n\to\infty}
            \frac{\kappa}{2\varepsilon_n}\L^d\left( U^{\mathcal{T}}_{\varepsilon_n}(\Gamma_n(t)) \EEE\right) = \kappa\mathcal{H}^{d-1}(\Gamma(t)). 
        \end{align*}
\item[{\rm (iv)}] \emph{Convergence of crack sets:} We have  $\frac{1}{2\varepsilon_n}\mathds{1}_{\Gamma\en(t)} \rightharpoonup^* \EEE \mathcal{H}^{d-1}|_{\Gamma(t)}   $ in the sense of measures for all $t\in[0, 1]$. \EEE
\end{enumerate}
\end{theorem}    

From the result in  \cite{SchmidtFraternaliOrtiz2009}, we get that $ \gamma_{n} \EEE (t) \to D^su(t)$ in the flat norm for all $t \in [0,1]$. Item (iv) can be understood as the evolutionary counterpart taking the union of all cracks up to time $t$ into account. \EEE

%
%
%


Theorem \ref{theorem:approx} and Theorem \ref{theorem:griffithmin} will be  proven \EEE in Sections \ref{section:ExistenceForFixedEpsilon}--\ref{chapter:Limitepsilontozero}.  In the proofs, we will simply assume that $Q(F) = |F|^2$  and $\kappa=1$  for notational convenience, as the adaptations for the general case are straightforward. \EEE
\subsection{Further results}\label{sec:further}

In   this  subsection, we present two \EEE further results which represent variants of the statements given in Subsections \ref{sec: main1}--\ref{sec: main2}.

\subsection*{Simultaneous limit}

In the previous two  subsections, we derived a quasi-static crack evolution for the energy \eqref{limitgama} by subsequently passing to a time-continuous limit ($m \to \infty$) and afterwards to the limit of vanishing eigenstrain-approximation ($\eps \to 0$). We now discuss a simultaneous limit in the time discretization and space approximation. Considering simultaneous space-time discretizations is not only relevant for numerical implementation (see e.g.\ \cite[Section 6]{SchmidtFraternaliOrtiz2009}), but is an approach which has recently been adopted in analytical results for the approximation of crack growth \cite{FriedrichSeutter2025, GiacominiPonsiglione2006_Gamma_convergence_approach, FriedrichSteinkeStinson2025, CrismaleFriedrichSeutter2025}.


We consider a sequence $\varepsilon_n \to 0$ and the corresponding sequence $h_n :=h(\eps_n) \to 0$ as in \eqref{heps}. We consider the time-discretized setting introduced in Subsection \ref{sec: main1}, where the time discretization is now denoted by  $I_n=\{0=t^n_0<t^n_1<\dots t^n_n\}\subset I_\infty$  with $\lim_{n \to \infty} \max_{i=1}^n (t^n_i-t_{i-1}^n) = 0$. We consider the  sequence of   minimization problems \eqref{eq:mineps0} and \eqref{eq:minepsi+1}  (with $t_i^n$ in place of $t_i^m$), and denote by $(u^n_i, \gamma^n_i)$  the solutions for the respective problems. We extend to functions $(\hat{u}_n(t), \hat{\gamma}_n \EEE (t))$ on the entire time interval $[0,1]$ via a piecewise constant interpolation in time, i.e., $\hat{u}_n(t)=u^n_i$  and $\hat{\gamma}_n(t)=\gamma^n_i$ for $t^n_i\leq t<t^n_{i+1}$. We define $\hat{\Gamma}_n(t):=   \bigcup_{\tau \in I^t_n}   \{\hat{\gamma}_n(\tau) \neq 0 \EEE \}$.

\begin{theorem}[Simultaneous limit]\label{th: simlutaneous}
Consider a sequence $\varepsilon_n \to 0$, the corresponding sequence $h_n :=h(\eps_n) \to 0$, and a time discretization $I_n=\{0=t^n_0<t^n_1<\dots t^n_n\}\subset I_\infty$  with $\lim_{n \to \infty} \max_{i=1}^n (t^n_i-t_{i-1}^n) = 0$. Let $t \mapsto (\hat{u}_n(t),\hat{\gamma}_n(t))$ be the time-discrete evolution given by \eqref{eq:mineps0} and \eqref{eq:minepsi+1}. 

Then, there exists a subsequence of $(\eps_n)_n$ (not relabeled) and   a quasi-static evolution $t\to (u(t), \Gamma(t))$ with boundary condition $g$ in the sense of Definition \ref{main def-lim} such that items {\rm (i)}--{\rm (iv)} of Theorem \ref{theorem:griffithmin} hold with $(\hat{u}_n,\hat{\gamma}_n,\hat{\Gamma}_n)$ in place of $({u}_n,{\gamma}_n,{\Gamma}_n)$.
\end{theorem}
\subsection*{Continuous eigenfracture approximation} 
Besides its potential relevance for numerical approximation schemes (see \cite[Section 5]{SchmidtFraternaliOrtiz2009}), we have adopted a 
finite-element discretization  of the eigendeformation approximation also for analytical reasons. Indeed, without spatial discretization, compactness for the variables $(u,\gamma)$ in \eqref{eq:energydefinitionepsilon} is not guaranteed which in turn impedes the derivation of an existence result as in Theorem \ref{theorem:approx} (actually, already in a static setting). Yet, we can also work with the `standard' eigenfracture approximation 
\begin{align*}
    \mathcal{F}_\varepsilon(u, \gamma,\Psi) =  \int_{\Omega'} Q(\nabla u-\gamma)\,\derx + \frac{\kappa}{2\varepsilon}\L^d\big(U_\varepsilon(\Psi\cup \{\gamma\neq 0\})\big)
\end{align*}
for $(u,\gamma) \in W^{1,1}(\Omega') \times L^1(\Omega')$ \EEE without finite-element discretization if we relax the formulation to  \textit{almost  minimizers}.  \EEE

We start with a time discretization $I_m=\{0=t^m_0<t^m_1<\dots t^m_m\}$. Let us also define a sequence $\varepsilon=\varepsilon(m)$ with $\varepsilon\to 0$ for $m\to\infty$. First, the analog of \eqref{eq:mineps0} is
\begin{align*}
      \inf\big\{\mathcal{F}_\varepsilon(u, \gamma, \emptyset)\colon (u, \gamma) \in  W^{1,1}(\Omega')\times L^1(\Omega'), \ u=g(0),  \ \gamma=0 \text{ on } \Omega'\setminus\overline{\Omega}\big\}.
\end{align*}
Since we do not have the necessary compactness for $\gamma$ in this setting, we will set $ (u'_{0,m}, \gamma'_{0,m}) \in  W^{1,1}(\Omega')\times L^1(\Omega') $ to be a \textit{almost  minimizer}, i.e., \EEE a pair that satisfies 
\begin{align*}
    \mathcal{F}_\varepsilon(u'_{0,m}, \gamma'_{ 0,m},\emptyset)\leq \inf\{\mathcal{F}_\varepsilon(u, \gamma,\emptyset)\colon  (u, \gamma) \in  W^{1,1} \EEE (\Omega')\times L^1(\Omega'), \ u=g(0), \ \gamma=0 \text{ on } \Omega'\setminus\overline{\Omega}\} + \frac{\varepsilon}{m}. \EEE
\end{align*}
With this, we can repeat the iteration argument for the problem 
\begin{align*}
    \inf\big\{\mathcal{F}_\varepsilon\big(u, \gamma, \bigcup\nolimits_{0\leq k\leq i}\{\gamma'_{k,m}\neq 0\}\big)\colon (u, \gamma) \in W^{1,1}(\Omega')\times L^1(\Omega'), \EEE \ u=g(t^m_{i+1}),   \  \gamma=0 \text{ on } \Omega'\setminus\overline{\Omega}\big\},
\end{align*}
and choose $(u'_{i+1,m}, \gamma'_{ i+1,m})$ whose energy is $\frac{\eps}{m}$ close to the infimum. Then, we extend to functions on the entire time interval $[0, 1]$ through a piecewise constant interpolation in time as before, such that we get $u'_m(t)$ and $\gamma'_m(t)$. We also define $\Gamma'_m(t):= \bigcup_{\tau\in I^t_m}\{\gamma'_m(\tau)\neq 0\}$.

In this setting, we then have the following result.

\EEE
\begin{theorem}[A version of Theorem \ref{theorem:griffithmin} for the continuous eigenfracture approximation]
    \label{theorem:griffithminNondiscrete}
 For all $\varepsilon >0$, \EEE we let $t\to(u'_{\varepsilon}(t), \gamma'_{\varepsilon}(t), \Gamma'_\eps(t))$ be the evolutions with boundary data $g$ given above, where $g\in W^{1, 1}([0, 1]; W^{2, \infty}( \R^d \EEE ))$. Then there exists a quasi-static evolution $t\mapsto (u(t), \Gamma(t))$ with boundary data $g$ in the sense of Definition \ref{main def-lim} and a (not relabeled) subsequence $\varepsilon_{ m}\to 0$ such that, setting $u'_{ m}:= u'_{ \varepsilon_m}$, $\gamma'_{ m}:= \gamma'_{ \varepsilon_m}$,  and $\Gamma'_{ m}:= \Gamma'_{ \varepsilon_m}$ \EEE,  items {\rm (i)}--{\rm (iv)} of Theorem \ref{theorem:griffithmin} hold with $({u}'_{ m},{\gamma}'_{ m},{\Gamma}'_{ m})$ in place of $({u}_n,{\gamma}_n,{\Gamma}_n)$. \EEE            
\end{theorem}

\EEE

\section{Existence of the quasi-static eigenfracture approximation}
\label{section:ExistenceForFixedEpsilon}

This section is devoted to the proof of Theorem \ref{theorem:approx}. Fix some $\varepsilon>0$ with its associated $h=h(\varepsilon)>0$ and the mesh $\mathcal{T}_h$. Recall that $g_h$ denotes the \EEE continuous, piecewise affine interpolation of $g$ on the triangulation $\mathcal{T}_h$. \EEE Moreover,  we introduce the \EEE notation   $g^m_h(t):= g^{m}_{h, i}= g_h(t^m_i)$ for all $t^m_{i} \le t< t^m_{i+1}$. \EEE With this, we have  $g^m_h(t) \to g_{ h}(t)$ strongly in  $H^1(\Omega')$ \EEE for all $t\in [0, 1]$ and, due to construction of the $I_m$, even $g^m_{h}(t)=g_{ h}(t)$ for $t\in I_\infty$ and $m$ large enough. Further, we recall the time-discrete evolution $(u^m_i,  \gamma^m_i \EEE )_{0 \le i\le m}$ as given in  \eqref{eq:mineps0} and \eqref{eq:minepsi+1}. 

\textbf{Existence of minimizers:} Let us first observe that the existence of minimizers for the problems  \eqref{eq:mineps0} and \eqref{eq:minepsi+1} is standard since the problem is posed on a finite-dimensional space  and for minimizers we can assume the bounds 
\begin{align}\label{apriori bound}
\Vert u^m_i \Vert_{L^\infty(\Omega')} \le \Vert g_h \Vert_{L^\infty(\Omega')}, \quad \quad \quad \Vert \nabla u^m_i \Vert_{L^\infty(\Omega')} \le C_h \EEE \Vert u^m_i \Vert_{L^\infty(\Omega')}.
\end{align} \EEE Indeed, the first estimate follows by a standard truncation argument, and the second follows from the fundamental theorem of calculus, where $C_h$ depends on the minimal size of  simplices \EEE in the triangulation $\mathcal{T}_h$. This yields compactness for the displacement, and compactness for the eigenstrain field immediately follows from \eqref{eq:identityyni}.  
We introduce an \emph{interpolation in time} by defining the piecewise  constant function
\begin{align*}
    u^m_\varepsilon(t)=u^m_i, \quad \gamma^m_\varepsilon(t)=\gamma^n_i\quad \text{for }t^m_i\leq t<t^m_{i+1}. 
\end{align*}

\textbf{Discrete energy estimate:} Recall the definition of $\Gamma^m_\varepsilon(t)$ in \eqref{gammam} and the definition of the neighborhood $U^{\mathcal{T}}_\eps$ introduced before \eqref{eq:energydefinitionepsilon}.
Following the reasoning in \cite[Section 3.2]{FrancfortLarsen2003}, we derive a discrete energy estimate. Testing \eqref{eq:minepsi+1} at fixed time $t^m_{i+1}\in I_m \EEE$  with the admissible \EEE  competitor $(u^m_{i}+g^m_{h, i+1}-g^m_{h, i}, \gamma^m_{i}) \EEE $, we get 
\begin{align*}
        \E_\varepsilon\big(u^m_{i+1}, \gamma^m_{i+1}, \Gamma^m_\varepsilon(t^m_i)\big)\leq \,&\E_\varepsilon\big(u^m_{i}+g^m_{h, i+1}-g^m_{h, i}, \gamma^m_{i}, \EEE \Gamma(t^m_i)\big)\notag\\
        =\,& \intnormal{u^m_i+\nabla g^m_{h, i+1}-\nabla g^m_{h, i}-\gamma^m_i}+\frac{1}{2\varepsilon}\L^d\big(U^{\mathcal{T}}_\varepsilon\big(\Gamma(t^m_i)\cup \{\gamma^m_i\neq 0\}\big)\big)\notag \\
        \leq\,& \E_\varepsilon(u^m_i, \gamma^m_i,  \Gamma(t^m_{i}) \EEE )+2\int_{t^m_i}^{t^m_{i+1}}\int_\Omega (\nabla u^m_{i}-\gamma^m_i)\cdot \partial_t\nabla g_h(s)\,\derx\,\ders \notag \\
        &+ e(m)\int_{t^m_i}^{t^m_{i+1}}\|\partial_t\nabla g_h(s)\|_{L^2(\Omega)}\,\der{s}  ,
    \end{align*}
with $e(m):= \max_{t^m_i \in I_m}\int_{t^m_i}^{t^m_{i+1}}\|\partial_t\nabla g_h(s)\|_{L^2(\Omega)}\,\der{s}$ being monotonously decreasing and vanishing for $m\to\infty$ due to the regularity of $g$. \EEE Note that in the last step we have used the fundamental theorem of calculus and Minkowski's integral inequality. Exploiting    $\E_\varepsilon(u^m_{i+1}, \gamma^m_{i+1}, \Gamma^m_\varepsilon(t^m_i)) =  \E_\varepsilon(u^m_{i+1}, \gamma^m_{i+1}, \Gamma^m_\varepsilon(t^m_{i+1}))$ (see \eqref{gammam}) and iterating this estimate in time leads to the following inequality for arbitrary  times \EEE $0\leq t_1\leq t_2\leq 1$: 
\begin{align}
        \label{eq:ineqtimeintn}
        \E_\varepsilon\big(u^m_\varepsilon(t_2), \gamma^m_\varepsilon(t_2), \Gamma_\eps^m(t_2)\big) \leq\,& \E_\varepsilon\big(u^m_\varepsilon(t_1), \gamma^m_\varepsilon(t_1), \Gamma_\eps^m(t_1)\big) + 2\int_{t_{i_1}^m}^{t_{i_2}^m}\int_\Omega(\nabla u_{\varepsilon}^m(s)-\gamma_\varepsilon^m(s))\cdot \partial_t\nabla g_h\EEE (s)\,\derx\,\ders \notag \\
                &+ e(m) \int_{t_{i_1}^m}^{t_{i_2}^m} \|\partial_t\nabla g_h\EEE (s)\|_{L^2(\Omega)}\,\ders, 
  \end{align}
with $i_1$ and $i_2$ chosen such that $t_{i_1}^m\leq t_1<t_{i_1+1}^m$ and $t_{i_2}^m\leq t_2<t_{i_2+1}^m$.

\textbf{Limiting passage $m\to\infty$:}  
 By means of Helly's theorem, up to passing to a subsequence (not relabeled), we can  ensure that $\lambda^m_\varepsilon(t):= \frac{1}{2\varepsilon}\L^d\left(U^{\mathcal{T}}_\varepsilon(\Gamma^m_\varepsilon(t)\right)$ converges to some increasing limiting function $\lambda_\eps$ for all $t\in [0,1]$. We denote the discontinuity points of $\lambda_\eps$ by $\mathcal{N}_\eps$ and note that $\mathcal{N}_\eps$ is at most countable. (Actually, $\mathcal{N}_\eps$ is even finite since the mappings $\lambda^m_\varepsilon$ satisfy $\lambda^m_\varepsilon([0,1]) \subset D_{\eps,h}$ for some finite set $D_{\eps,h} \subset \R$ independent of $m$. Yet, this will not be needed in the following.) For brevity, we introduce 
\begin{align}\label{iinftyeps} 
I_{\infty,\eps} :=I_\infty  \cup \mathcal{N}_\eps, \quad \quad I_{ \infty, \eps }^{ t } := \left\{ \tau \in I_{ \infty, \eps } \colon \ \tau \leq t \right\}.  
\end{align}
\EEE With the same arguments as for the existence of time-discrete solutions, see \eqref{apriori bound}, \EEE we obtain compactness for the sequences  $(u^m_\varepsilon(t), \gamma^m_\varepsilon(t))_m$ for each $t\in I_{\infty,\eps} \EEE$, up to passing to a further subsequence (not relabeled). Therefore, by a diagonal argument, we obtain a single subsequence (not relabeled) and limits $(u_\eps(t),\gamma_\eps(t)) \in V_h(\Omega')\times W_h(\Omega')$  satisfying $u_\varepsilon(t)=g_h(t)$, $\gamma_\varepsilon(t)=0$ on $\Omega' \setminus \overline{\Omega}$ for $t \in I_{\infty,\eps} \EEE$ such that $u^m_\varepsilon(t)\to u_\varepsilon(t)$  in $H^{1}(\Omega')$  \EEE and $\gamma^m_\varepsilon(t)\to \gamma_\varepsilon(t)$  in $L^1(\Omega')$ \EEE  for all $t\in I_{\infty, \eps}$. Note that, due to the definition of $V_h(\Omega')$ and $W_h(\Omega')$, the convergence also holds pointwise. From now on, if not stated otherwise,  all convergences are intended with respect to this subsequence. \EEE

To extend the notion to the entire time interval $[0, 1]$,
 let $t\notin I_{\infty, \eps} $ and $t_k\nearrow t$ with $t_k\in I_{\infty, \eps}$ for every $k\in\N$. Then, with the same arguments as above, we find a subsequence of $(t_k)_k$ (not relabeled) such that   $(u_\varepsilon(t_k), \gamma_\varepsilon(t_k))$  converges in the $H^1(\Omega') \times L^{1}(\Omega')$-sense  \EEE to a limit, denoted by  $(u_\varepsilon(t), \gamma_\varepsilon(t))$. \EEE Later we will see that this limit coincides with the limit of  $(u^m_\varepsilon(t), \gamma^m_\varepsilon(t))$ as $m \to \infty$  for \emph{all} $t \in [0,1]$. \EEE

Defining \EEE  $ \Gamma_\varepsilon(t):= \bigcup_{ \tau\in I^t_{\infty,\eps}} \EEE \{\gamma_\varepsilon(\tau)\neq 0\}$ for all $t \in [0,1]$, we have the following property on the supports of the eigenstrain fields. \EEE  
\begin{lemma}[Eigenstrain supports]
    \label{lem:stabilityjumpset}
    For every $t\in [0, 1]$, we find $m_0 \in \N$ depending on $t$ such that  $ \bruchdisc(t) \subset \bruchdiscn(t)$ for all  $m \ge m_0$.   
\end{lemma}
\begin{proof}
We first observe that we have a finite amount of simplices in $\Gamma_\varepsilon(\cdot)$ by construction. Fix a time step $t\in [0, 1]$ and let us  consider \EEE a simplex $T$ with $\gamma_\varepsilon(t) \neq \EEE 0$ on $T$. Then there exists a  time \EEE $s\in I^t_{\infty, \varepsilon}$ such that $\gamma_\varepsilon(r) = 0$ on $T$ for $r<s$ and $\gamma_\varepsilon(s)\neq 0$ on $T$. By pointwise convergence  there exists $m_0 = m_0(T) \in \N$ \EEE such that  $\gamma^m_\varepsilon(s)\neq 0$ on $T$ for all $m\geq m_0$ and therefore $T\subset \Gamma^m_\varepsilon(t)$ for all $m\geq m_0$. Repeating this procedure for every $T  \subset \EEE \Gamma_\varepsilon(t)$, we can find a joint $m_0$  depending only on $t$ \EEE such that $\gamma^m_\varepsilon(t)  \neq \EEE 0$ on $\Gamma_\varepsilon(t)$ for $m\geq m_0$.  
\end{proof}

 Note that, at this stage,  equality of $ \bruchdisc(t)$ and $\bruchdiscn(t)$ for large  $m$ cannot guaranteed as possibly $\gamma^m_\varepsilon(t) \to 0$ on some $T \in \mathcal{T}_h$. \EEE Next, we address  \EEE  the convergence of the strains  for all $t\in [0,1]$.  
\begin{lemma}[Convergence of strains\EEE]
    \label{lem:PointwiseConvergence}
For all $t \in [0,1]$ it holds that  $\nabla u^m_\varepsilon(t)\mathds{1}_{(\Gamma^m_\varepsilon(t))^c}\to \nabla u_\varepsilon(t)-\gamma_\varepsilon(t)$  in $L^2(\Omega')$  as $m \to \infty$. \EEE
\end{lemma}

 For $t\in I_{\infty,\eps} $ the proof follows directly by the convergence obtained below \eqref{iinftyeps}  and identity \eqref{eq:identityyni}. In the general case,   the proof follows a standard argument which we present in the appendix. \EEE  With the preceding results, we are now able to  obtain \EEE the minimality for the limit.

\begin{lemma}[Stability]
    \label{lem:stabdisc}
    For every $t\in (0, 1]$ and every $(\bar{u}, \bar{\gamma})\in V_h(\Omega')\times W_h(\Omega')$ such that $\bar{u}=g_h(t)$ and $\bar{\gamma}=0$ on $\Omega' \setminus \overline{\Omega}$, we have 
       \begin{align}
        \label{eq:mindisc}
        \E_\varepsilon(u_\varepsilon(t), \gamma_\varepsilon(t), \Gamma_\varepsilon(t))\leq \E_\varepsilon(\bar{u}, \bar{\gamma}, \Gamma_\varepsilon(t)). 
           \end{align}
    Additionally, for $t=0$,  for every $(\bar{u}, \bar{\gamma})\in V_h(\Omega')\times W_h(\Omega')$ with $\bar{u}=g_h(0)$ and $\bar{\gamma} = 0$ on $\Omega' \setminus \overline{\Omega}$   we have
    \begin{align}
        \label{eq:mindisczero}
        \E_\varepsilon(u_\varepsilon(0), \gamma_\varepsilon(0), \emptyset)\leq \E_\varepsilon(\bar{u}, \bar{\gamma}, \emptyset).
    \end{align}
\end{lemma}

Here, we note that $ \E_\varepsilon(\cdot, \cdot, \emptyset)$ coincides with the energy defined in \eqref{eq:energydefinitionepsilon}.  

\begin{proof}
       We define $\bar{u}_m:= \bar{u}+g^m_h(t)-g_h(t)$. With the regularity properties of $g_h$, we obtain $\bar{u}_m\to \bar{u}$ in  $H^1(\Omega')$ \EEE for $m\to\infty$. Since $(  \bar{u}_m, \EEE \bar{\gamma} \EEE )$ is admissible, we get by the minimality in \eqref{eq:minepsi+1} \EEE  
    \begin{align*}
        \E_\varepsilon(u^m_\varepsilon(t), \gamma^m_\varepsilon(t), \Gamma^m_\varepsilon(t))\leq \E_\varepsilon(\bar{u}_m,    \bar{\gamma}, \EEE \Gamma^m_\varepsilon(t)).
    \end{align*}
    Then, subtracting the surface energy in \EEE \eqref{eq:energydefinitionepsilon-new}  and taking the limit in $m$\EEE, by \eqref{eq:identityyni} and Lemma \ref{lem:PointwiseConvergence}, 
    this shows
    $$
 \int_\Omega |\nabla u_\eps(t)-\gamma_\eps(t)|^2 \, {\rm d}x \le   \int_\Omega |\nabla \bar{u}-\bar{\gamma}|^2 \, {\rm d}x  + \limsup_{m \to \infty }\frac{1}{2\varepsilon}\L^d\left(U^{\mathcal{T}}_\varepsilon(\{\bar{\gamma}\neq 0\} )\setminus U^\T_\varepsilon(\Gamma^m_\varepsilon(t))\right).    
     $$
 Using  Lemma \ref{lem:stabilityjumpset}  we find
\begin{align*}
 \int_\Omega |\nabla u_\eps(t)-\gamma_\eps(t)|^2 \, {\rm d}x  \le   \int_\Omega |\nabla \bar{u}(t)-\bar{\gamma}|^2 \, {\rm d}x  +  \frac{1}{2\varepsilon}\L^d\left(U^{\mathcal{T}}_\varepsilon(\{\bar{\gamma}\neq 0\}) \setminus U^\T_\varepsilon(\Gamma_\varepsilon(t) ) \right).    
\end{align*}
Adding \EEE   $\frac{1}{2\varepsilon}\L^d(U^{\mathcal{T}}_\varepsilon(\Gamma_\varepsilon(t) ))$ on both sides we obtain \EEE  (\ref{eq:mindisc}).  Property \eqref{eq:mindisczero} is easier and follows directly from minimality (see \eqref{eq:mineps0}) and   Lemma~\ref{lem:PointwiseConvergence}. 
\end{proof}
With this minimality,  for all $t \in [0,1]$ we obtain
\begin{align}\label{limitinggamma}
 \gamma_\varepsilon(t)=\nabla u_\varepsilon(t)\mathds{1}_{\bruchdisc(t)}.
\end{align}
Indeed,   from the definition of $\Gamma_\eps$ and the definition of $\gamma_\eps(s)$ for $s \notin I_{\infty,\eps}$ we get  $\lbrace \gamma_\eps(t) \neq 0 \rbrace \subset \Gamma_\eps(t)$. This along with    minimality yields \eqref{limitinggamma} \EEE with the same argument as in (\ref{eq:identityyni}).  Therefore, we can also write 
$\intnormal{u_\varepsilon(t)-\gamma_\varepsilon(t)}=\int_{\Omega\setminus\bruchdisc(t)}|\nabla u_\varepsilon(t)|^2\,\derx$.  Based on Lemmas \ref{lem:stabilityjumpset}--\ref{lem:stabdisc} we can now  derive an  \EEE energy balance.  
\begin{lemma}[Energy balance]
    \label{lem:bothinequalitesdisc}
    The limit $(u_\varepsilon(t), \gamma_\varepsilon(t))$ satisfies the   energy equality
    \begin{align}
        \label{eq:energyequalitydisc}
        \energymain{t}=\energymain{0}+2\int_0^t\int_\Omega (\nabla  u_\varepsilon(s)-\gamma_\varepsilon(s)) \cdot \partial_t\nabla g_h(s)\,\derx\,\ders.
    \end{align}
\end{lemma}

The proof is conceptually very similar to  \cite[Section 3.2]{FrancfortLarsen2003}. We only present one inequality here and defer the proof of the other inequality \EEE to  the appendix. We use (\ref{eq:ineqtimeintn}) with the time steps $t^m_{i_1} = 0$ and $t^m_{i_2}\in I_m$ such that $t^m_{i_2}\leq t<t^m_{i_2+1}$. \EEE  Then, we get
        \begin{align*}
            \E_\varepsilon(u^m_\varepsilon(t), \gamma^m_\varepsilon(t), \Gamma^m_\varepsilon(t))\leq \,&\E_\varepsilon(u^m_\varepsilon(0), \gamma^m_\varepsilon(0), \Gamma^m_\varepsilon(0))+ 2 \int_0^{t_{i_2}^m} \EEE  \int_\Omega(\nabla u^m_\varepsilon(s)-\gamma^m_\varepsilon(s)) \cdot \partial_t\nabla \EEE g_h(s)\,\derx\,\ders\\
        &+\, e\left(m\right) \EEE \int_0^t\| \partial_t \nabla  g_h(s)\|_{L^2(\Omega)}\,\ders.
        \end{align*}
Taking the limit $m\to\infty$ we find
  \begin{align}\label{firstinequ0}   \liminf_{m \to \infty}   \E_\varepsilon(u^m_\varepsilon(t), \gamma^m_\varepsilon(t), \Gamma^m_\varepsilon(t))       \le \energymain{0}+2\int_0^t\int_\Omega (\nabla  u_\varepsilon(s)-\gamma_\varepsilon(s)) \cdot \partial_t\nabla g_h(s)\,\derx\,\ders. 
  \end{align}
Indeed, we can pass to the limit in the integral  on the right-hand side    by    Lemma \ref{lem:PointwiseConvergence},  \eqref{eq:identityyni}, \eqref{apriori bound}, and  dominated convergence. Moreover, we have $\E_\varepsilon(u^m_\varepsilon(0), \gamma^m_\varepsilon(0), \Gamma_\eps^m(0)) = \E_\varepsilon(u_\varepsilon(0), \gamma_\varepsilon(0), \Gamma_\eps(0))$  for all $m \in \N$  as the solution of \eqref{eq:mineps0} is independent of $m$.  For the left-hand side of the previous equation, we use  Lemma \ref{lem:PointwiseConvergence},  \eqref{eq:identityyni}, and Lemma \ref{lem:stabilityjumpset} to get  
 \begin{align}\label{gammeps2}\liminf_{m \to \infty}   \E_\varepsilon(u^m_\varepsilon(t), \gamma^m_\varepsilon(t), \Gamma^m_\varepsilon(t))  & \ge \liminf_{m \to \infty} \Big(  \int_\Omega |\nabla u_\eps^m(t)-\gamma_\eps^m(t)\EEE|^2 \, {\rm d}x  + \frac{1}{2\varepsilon}\L^d\big(U^{\mathcal{T}}_\varepsilon(\Gamma^m_\varepsilon(t) )\big) \Big) \notag \\ & \ge   \int_\Omega |\nabla u_\eps(t)-\gamma_\eps(t)\EEE|^2 \, {\rm d}x  \,\derx+ \frac{1}{2\varepsilon}\L^d\big(U^{\mathcal{T}}_\varepsilon(\Gamma_\varepsilon(t) )\big). 
\end{align}
  In view of \eqref{eq:energydefinitionepsilon-new} and the fact that $\lbrace \gamma_\eps(t) \neq 0 \rbrace \subset \Gamma_\eps(t)$ (see \eqref{limitinggamma}), \EEE  this shows
  \begin{align}\label{firstinequ} 
  \liminf_{m \to \infty}   \E_\varepsilon(u^m_\varepsilon(t), \gamma^m_\varepsilon(t), \Gamma^m_\varepsilon(t))  \ge  \energymain{t}.
  \end{align}
This shows the first  inequality in \eqref{eq:energyequalitydisc}. \EEE The second  inequality in \eqref{eq:energyequalitydisc} will be discussed in the appendix.    \EEE With the preliminary results, we are now ready to  prove Theorem \ref{theorem:approx}. 
\begin{proof}[Proof of Theorem \ref{theorem:approx}]
 First, \EEE  conditions (a) and (c) follow from Lemma  \ref{lem:stabdisc}  and (b) is a direct consequence of the definition of $\Gamma_\eps(t)$. The energy balance (d) is stated in \EEE  Lemma \ref{lem:bothinequalitesdisc}. The fact that $u_\varepsilon(t)=g_h(t)$, $\gamma_\varepsilon(t)=0$ on $\Omega' \setminus \overline{\Omega}$ for all $t\in[0 ,1]$ follows directly from the construction of the limits. Moreover, \eqref{limitinggamma} implies $\lbrace \gamma_\eps(t) \neq 0 \rbrace \subset \Gamma_\eps(t)$, which shows $ (u_\varepsilon(t), \gamma_\varepsilon(t), \Gamma_\eps(t)) \in AD_\eps(g_h(t))$ for all $t\in[0 ,1]$. \EEE Combining \eqref{firstinequ0}--\eqref{firstinequ} and the energy balance \eqref{eq:energyequalitydisc}, we obtain convergence of energies at all times $t \in [0,1]$, i.e., the first item in  \eqref{energyconvergenceeps} holds.  In view of Lemma \ref{lem:PointwiseConvergence}, to conclude \eqref{energyconvergenceeps}, it suffices to show  $u^m_\eps(t) \to u_\eps(t)$  in $L^1(G_\eps(t))$. 


To this end,  we start by observing that both the elastic and the surface energy are lower semicontinuous, see the \EEE argument in  \eqref{gammeps2}.   This indeed shows 
\begin{align}\label{conviiiii}
\lambda_\eps(t) = \lim_{m \to \infty}\L^d(U^{\mathcal{T}}_\varepsilon(\Gamma^m_\varepsilon(t) )) = \L^d(U^{\mathcal{T}}_\varepsilon(\Gamma_\varepsilon(t) )) 
\end{align}
for all $t \in [0,1]$. From this we deduce that $U^\T_\varepsilon(\Gamma^{m}_\varepsilon(t))$ is constant equal to  $U^\T_\varepsilon(\bruchdisc(t))$  for  $m$ large enough depending on $t \in [0,1]$. Indeed,  by Lemma \ref{lem:stabilityjumpset} we have  $U^\T_\varepsilon(\bruchdisc(t)) \subset U^\T_\varepsilon(\bruchdiscn(t))$ for $m$ large enough.   This along with \eqref{conviiiii} and   the discrete nature of the problem induced by  the triangulation $\mathcal{T}_h$ shows the property. \EEE Then, by Lemma \ref{lem:PointwiseConvergence} and \eqref{limitinggamma} we get $\nabla u^m_\varepsilon(t)\to \nabla u_\varepsilon(t)$  in $L^2(\Omega'\setminus  U^\T_\varepsilon(\bruchdisc(t)))$.

Recall that  $G_\eps(t) \subset \Omega'$   denotes the largest open set with $\Omega' \setminus \overline{\Omega} \subset G_\eps(t)$ and    $\Gamma_\varepsilon(t)\cap  U^\T_\varepsilon(\bruchdisc(t)) = \emptyset$.  \EEE Now, by the fundamental theorem of calculus applied successively on the simplices of $\mathcal{T}_h$ along with the fact that $u^m_\eps(t) - u_\varepsilon(t) = g^m_h(t)-g_h(t) \to 0$ on $\Omega' \setminus \overline{\Omega}$ \EEE as  $m \to \infty$, we also get  $  u^m_\varepsilon(t)\to   u_\varepsilon(t)$  in $L^1(G_\eps(t))$. This concludes the proof.   
\end{proof}

\section{Passage to quasi-static crack growth}
\label{chapter:Limitepsilontozero}

%
%
%
%
%
%
%
%
%


This section is devoted to the proof of Theorem \ref{theorem:griffithmin}. We start by stating an energy bound for $\eps>0$. Then, we proceed with compactness and stability properties on a  countable subset of times. \EEE Afterwards, we extend the limiting evolution to all times and give the proof of the main result.

 For every $\varepsilon>0$, we denote the triple \EEE given by Theorem \ref{theorem:approx}  by $t \mapsto (u_\varepsilon(t), \gamma_\varepsilon(t),   \Gamma_\eps(t))$. \EEE By a truncation argument, the minimality properties (i) and (iii) along with $g \in L^{\infty}([0,1] \times \Omega') $ imply that   
\begin{align*}
    \|u_\varepsilon(t)\|_{L^\infty(\Omega')}\leq C
\end{align*}
for some $C>0$ independent of $\varepsilon$ and $t$.  Recalling the notation in \eqref{eq:energydefinitionepsilon-new}, we obtain the following  \EEE uniform bound on the energy.  
\begin{lemma}[Energy bound]
    \label{lem:energybound}
    There exists a constant $C_1\geq 0$ depending only on $g$ such that for all $t\in [0, 1]$ and $\varepsilon >0$ \EEE it holds that
    \begin{align}
        \label{eq:energybound}
        \E_\varepsilon(u_\varepsilon(t), \gamma_\varepsilon(t), \Gamma_\varepsilon(t))+\|u_{\varepsilon}(t)\|_{L^\infty(\Omega')}\leq C_1.
    \end{align}
  \end{lemma}
\begin{proof}
Given the energy balance stated in  Lemma \ref{lem:bothinequalitesdisc}, it suffices to estimate   $\E_\varepsilon(u_\varepsilon(0), \gamma_\varepsilon(0), \Gamma_\varepsilon(0))$ and $2\int_0^t\int_\Omega (\nabla u_\varepsilon(t)-\gamma_ \varepsilon(t)) \cdot \partial_t\nabla g_h(s)\,\derx\,\ders$.      First, due to minimality (see (i) in Theorem \ref{theorem:approx}), we have $\E_\varepsilon(u_\varepsilon(0), \gamma_\varepsilon(0), \Gamma_\varepsilon(0))\leq \intnormal{g_h(0)}\leq C\|g(0)\|^2_{W^{2,\infty}(\Omega')}$.      For the second term, for every $0\leq s\leq 1$, we can estimate 
    \begin{align} \label{5y5}
        \left|\int_{ \Omega}(\nabla u_{\varepsilon}(s)-\gamma_\varepsilon(s))\cdot \partial_t\nabla  g_h(s)\,\derx\right|
        \leq\,&
        \|\nabla u_{\varepsilon}(s)-\gamma_\varepsilon(s)\|_{L^2( \Omega)}
        \|\partial_t\nabla  g_h(s)\|_{L^2(\Omega)}
    \end{align} 
by Hölder's inequality. Using the admissible competitor $(g_h(s),0)$ in the minimality property (iii) of Theorem \ref{theorem:approx}, we obtain  the inequality 
    \begin{align*}
        \energieeps{s}
        \leq \,&\intnormal{g_h(s)}+\bruchmengedische{s},
    \end{align*}
 Subtracting $\bruchmengedische{s}$ then \EEE yields  
    \begin{align*}
        \intnormal{u_\varepsilon(s)-\gamma_\varepsilon(s)}\leq \intnormal{g_h(s)}\leq C\|g(s)\|^2_{W^{2,\infty}(\Omega')}.
    \end{align*}
Inserting this into \eqref{5y5}, the proof is concluded. \EEE
\end{proof}

 Next, we address compactness of the pairs $(u_\eps(t),\gamma_{\eps}(t))$. We will employ multiple times the   compactness result in $SBV$, due to {\sc Ambrosio} \cite[Theorem 4.8]{AmbrosioFuscoPallara2000} which we recall in the appendix, see Theorem \ref{theorem:sbvcompactness}.   In particular, the convergence in \eqref{eq:sbvcompactnessresult} will be called \emph{$SBV^2$-convergence}.  \EEE

To apply Theorem \ref{theorem:sbvcompactness},  we need a bound on $\nabla  u_\varepsilon(t)$. This, however, is not available due to the presence of the eigenstrain $\gamma_\eps(t)$ in the elastic energy. As a workaround, following the ideas in \cite{SchmidtFraternaliOrtiz2009},  we will cut off the function approximately at the places where $\gamma_\eps(t)$ has its support. A first attempt could  be \EEE to simply cut off at either  $\Gamma_\varepsilon(t)$ or at the neighborhood $U^{\mathcal{T}}_\varepsilon(\Gamma_\varepsilon(t))$, i.e., we could consider $u_\eps(t) \mathds{1}_{\Omega' \setminus \Gamma\en(t)}$  or \EEE   $u_\eps(t) \mathds{1}_{\Omega' \setminus U^{\mathcal{T}}_\varepsilon(\Gamma_\varepsilon(t))}$. However, to apply Ambrosio's compactness result, we need to control the jump set, and bounds on $\partial \Gamma\en(t)$ or $\partial U^{\mathcal{T}}_\varepsilon(\Gamma_\varepsilon(t))$ are not immediately available. Therefore, we need to work with a related, but geometrically simpler object.

For that matter, using the grid $\frac{\varepsilon}{\sqrt{d}}\Z^d$  we cover \EEE the reference configuration $\Omega'$  with  \EEE sets of the form $ Q_\varepsilon^y\EEE =\frac{\varepsilon}{\sqrt{d}}(y+(0, 1)^d), y\in \Z^d$. For any $t \in [0,1]$, we define $\mathcal{Q}_\varepsilon(t):= \{Q\in Q_\varepsilon^y\EEE\colon \,  \bruchdisc(t)\cap Q \neq \emptyset \text{ or } Q\not\subset \Omega'\}$ as the collection of cubes that intersect the `broken simplices' or the boundary. Let \EEE   $Q_\varepsilon(t):= \bigcup_{Q\in \mathcal{Q}_\varepsilon(t)}Q$. Note that by construction we have $\bruchdisc(t)\subset Q_\varepsilon(t)\subset  {U}^{\T}_\eps(\Gamma_\eps(t) \cup \partial \Omega') \EEE $ and  $\gamma_\varepsilon(\tau)\equiv0$ in $(Q_\varepsilon(t))^c$ for every $0\leq \tau\leq t$, see \eqref{limitinggamma}. \EEE

Moreover, we get  $\# \mathcal{Q}_\varepsilon(t)\leq \frac{ \mathcal{L}^d \EEE (\bruchdische{t})}{(\varepsilon/\sqrt{d})^d}   +  \frac{C_{\Omega'}}{\varepsilon^{d-1}}   \leq \frac{2 C_1 d^\frac{d}{2}}{G\varepsilon^{d-1}} +  \frac{C_{\Omega'}}{\varepsilon^{d-1}}$ due to the control on $\bruchdische{t}$ given by \eqref{eq:energybound} and the fact that $\Omega'$ is a Lipschitz set. We \EEE can then estimate 
\begin{align}
    \label{eq:jumpsetestimate}
    \mathcal{H}^{d-1}(\partial Q_\varepsilon(t))\leq \frac{4C_1 d \EEE \sqrt{d}}{G} + 2 d \EEE C_{\Omega'},
\end{align} 
which is independent of $h, \varepsilon$ and $t$.  For later purposes, let us remark that the energy bound in Lemma \ref{lem:energybound} and the control on $\# \mathcal{Q}_\varepsilon(t)$ obtained above imply   
\begin{align}
    \label{eq:triangleentozero}
    \L^d\big(U^{\mathcal{T}}_{\varepsilon}(\Gamma_\eps(t))\big)\to 0, \quad \L^d(Q_\eps(t))\to 0 \quad  \text{as $\eps \to 0$ \  for every $t \in [0,1]$}.
\end{align}
Let us now consider an \EEE arbitrary sequence $\varepsilon_n\to 0$ with corresponding sequence $h_n := h(\eps_n)\to 0$. 
From now on, \EEE with abuse of notation, we will write $u_n(t):= u\en(t)$, $\gamma_n(t):= \gamma\en(t)$, $Q_n(t):= Q\en(t)$, and $\Gamma_n(t):= \Gamma\en(t)$. We also use  $U_n(\cdot):=U\en(\cdot)$,  $U^{\mathcal{T}}_n(\cdot)=U^{\mathcal{T}}_{ \eps_n}(\cdot)$,   $\mathcal{E}_n(\cdot) := \mathcal{E}_{\eps_n}(\cdot)$,  and $g_n:=g_{h_n}$. \EEE

 By means of Helly's theorem  and Theorem \ref{theorem:approx}(b), \EEE  up to passing to a subsequence (not relabeled), we can   assume \EEE that
\begin{align}
    \label{eq:lambdaepsilondefinition}
    \lambda_n(t):= \frac{1}{2\varepsilon_n}\L^d\left(U^{\mathcal{T}}_n(\Gamma_n(t))\right)
\end{align}
 converges to some increasing limiting function $\lambda_0$ for all $t\in [0,1]$. We denote the discontinuity points of $\lambda_0$ by $\mathcal{N}_0$ and note that $\mathcal{N}_0$ is at most countable. For brevity, we define $I_{\infty,0} :=  I_\infty  \cup \mathcal{N}_0$. \EEE

 With these preparations, we are now able to obtain a \EEE compactness result for the \EEE \emph{modifications }
\begin{align}\label{zn}
z_n(t):= u_n(t)\mathds{1}_{(Q_n(t))^c}, \quad t \in [0,1],
\end{align}
when restricting to times in $I_{\infty,0}$. \EEE
\begin{lemma}[Compactness on $I_{\infty,0}$]
   \label{lem:truncationIinfinity}
    There exists a subsequence of $(\varepsilon_n)_n$ (not relabeled) such that for all $t\in I_{\infty,0}$ there exists $u(t)\in SBV^2( \Omega')\EEE$, satisfying $u(t)= g(t)$ on $\Omega' \setminus \overline{\Omega}$, with 
    \begin{align}
        \label{eq:truncationIinfinityconvergence}
       u_n(t)  \to u(t) \quad\emph{in }L^1(\Omega'),  \quad \quad \quad      
       z_n(t)   \to u(t) \quad\emph{in }SBV^2(\Omega').
    \end{align}
    In particular, for all $t\in I_{\infty,0}$, we have
    \begin{align}
        \label{eq:truncationIinfinitybound}
        \int_{\Omega}|\nabla u(t)|^2\,\derx+ G \EEE \mathcal{H}^{d-1}(J_{u(t)})+\|u(t)\|_{L^\infty(\Omega')}\leq C_1.
    \end{align}
\end{lemma}
\begin{proof}
    Fix  $t\in I_{\infty,0} $. This sequence  $(z_n(t))_n$ inherits the $L^\infty$-bound of $u_n(t)$ given by \eqref{eq:energybound}. Concerning the gradient, by using  $\bruchdiscepsilonn(t) \subset Q_n(t)$ and   \eqref{eq:identityyni} we can estimate 
    \begin{align*}
        \intnormal{z_n(t)}\leq \int_{\Omega\setminus\bruchdiscepsilonn(t)}|\nabla u_n(t)|^2\,\derx=\intnormal{u_n(t)-\gamma_n(t)}\leq C_1. 
    \end{align*}
Moreover,  since $u_n(t)$ is continuous for every $\varepsilon_n$,  jumps of $z_n(t)$ can only occur along $\partial Q_n(t)$. Thus, in view of \eqref{eq:jumpsetestimate}, we obtain $\mathcal{H}^{d-1}(J_{z_n(t)})\leq C$ for a constant independent of $t$ and $n$. By Theorem \ref{theorem:sbvcompactness} we thus find $u(t) \in SBV^2(\Omega')$ such that \EEE   $z_n(t)\to u(t)$ in $SBV^2(\Omega')$ along a not relabeled subsequence. With a diagonal argument, we are then able to find a joint subsequence for all $t\in I_{\infty,0}$  \EEE such that we have $z_n(t)\to u(t)$ in $SBV^2(\Omega')$ for all $t\in I_{\infty,0}$. The uniform bound on $(u_n(t))_n$ along with \eqref{eq:triangleentozero} also  yields \EEE $u_n(t)\to u(t)$ in $L^1(\Omega')$. This also shows that \EEE the boundary values are still satisfied. Finally, we have 
    \begin{align*}
        &\intnormal{u(t)}+ G \EEE \mathcal{H}^{d-1}(J_{u(t)})+\|u(t)\|_{L^\infty(\Omega')}\\
        \leq\,&\liminf_{n\to\infty} \intnormal{u_n(t)-\gamma_n(t)}+\bruchmengedischen{t}+\|u_n(t)\|_{L^\infty(\Omega')} \leq C_1,
    \end{align*}
    due to the $\Gamma$-liminf-inequality  of \EEE  \cite[Theorem 5.1]{SchmidtFraternaliOrtiz2009}. This concludes the proof. 
\end{proof}

In Lemma \ref{lem:truncationIinfinity} we have found limits for $t\in I_{\infty,0}$ for suitable modifications. Next, we discuss that the cut-off can be also done with the union of cubes $Q_n(t)$ at some later time $t$.

\begin{lemma}
    \label{lem:alternativetruncation}
    Let $s\in I_{\infty,0}$ and $t\in [0, 1]$ with $s\leq t$. Then, it holds that 
    \begin{align*}
        u_n(s)\mathds{1}_{(Q_n(t))^c}\to u(s)\quad\emph{in } SBV^2(\Omega').
    \end{align*}
\end{lemma}
\begin{proof}
    The function $u_n(s)\mathds{1}_{(Q_n(t))^c}$ has the same bounds as in Lemma \ref{lem:truncationIinfinity} which means that 
    \begin{align}
        \label{eq:uen(s)truncation(t)toz}
        u_n(s)\mathds{1}_{(Q_n(t))^c}\to z
    \end{align}
    for some $z\in SBV^2(\Omega')$ and a not relabeled subsequence.  Now, we can estimate 
    \begin{align*}
        \| u_n(s)\mathds{1}_{(Q_n(s))^c} -u_n(s)\truncation{t}\|_{L^2(\Omega')}\leq \|u_n(s)\|_{L^\infty(\Omega')}\mathcal{L}^d(Q_n(t))^\frac{1}{2}\leq C_1 \mathcal{L}^d(Q_n(t))^\frac{1}{2}
    \end{align*}
    and with \eqref{eq:triangleentozero} we get that $z -u_n(s)\truncation{ s} \to 0 $ in $L^2(\Omega')$. 
This along with  \eqref{eq:truncationIinfinityconvergence} and \eqref{eq:uen(s)truncation(t)toz} yields the  result.
\end{proof}

The next result addresses the unilateral stability property of the limiting evolution at times $t \in I_{\infty,0} \EEE$.  For convenience, we introduce the notation  
\begin{align}
    \label{eq:Gammadefinition}
    \Gamma(t):= \bigcup_{\tau\in I^t_{\infty,0}} \EEE J_{u(\tau)},
\end{align}
where   $I_{ \infty, 0 }^{ t } := \left\{ \tau \in I_{ \infty, 0} \colon \, \tau \leq t \right\}$.   \EEE
\begin{theorem}[Stability result]
    \label{theorem:stabilityresulteps}
    Let $t\in I_{\infty,0}$. \EEE Then for every $ \phi \EEE\in SBV^2(\Omega')$ with $ \phi = g(t)$ on $\Omega' \setminus \overline{\Omega}$ it holds that  
    \begin{align}
        \label{eq:stabilityresulteps}
        \int_{\Omega}|\nabla u(t)|^2\,\derx \leq \int_{\Omega}|\nabla \phi \EEE|^2\,\derx+ \mathcal{H}^{d-1}\left(J_{\EEE\phi}\setminus \Gamma(t)\right).
    \end{align}
\end{theorem}
The proof is delicate and the most original part of the paper. It is deferred to Section \ref{section:stabilityresult} below.   
We   now summarize the  results already found and derive a limit for the broken simplices. \EEE We recall the functions $\lambda_n$ defined in \eqref{eq:lambdaepsilondefinition} as well as their pointwise limit  $\lambda_0$ \EEE obtained by Helly's theorem.  \EEE 

\begin{lemma}
    \label{lemma:firstsummary}
It holds that \EEE
    \begin{align*}
        \lambda_0(t)\geq \mathcal{H}^{d-1}\left(\Gamma(t)\right) \quad \text{for all $t\in [0,1]$.}
    \end{align*}
\end{lemma}

\begin{proof}    
We start with  a preliminary observation. Given any open set $V \subset \Omega'$, the $\Gamma$-liminf inequality of  \cite[Theorem 5.1]{SchmidtFraternaliOrtiz2009} along with \eqref{eq:truncationIinfinityconvergence} implies that  
    \begin{align}\label{gammaV}
       \EEE  \mathcal{H}^{d-1}(J_{u(s)} \cap V)\leq \liminf_{n\to\infty} \frac{1}{2\varepsilon_n}\L^d\big(\hen{\gamma_n(s)} \cap V \big) \quad \text{ for all } s \in I_{\infty,0}.
    \end{align}
Therefore, given any number of points $(s_i)_{i}^N \subset I_{\infty,0}^t$ and  $\mu>0$ arbitrary, we can apply \cite[Lemma~A.3]{FriedrichSteinkeStinson2025} to find open, pairwise disjoint sets $(V_i)_{i=1}^N$ in $\Omega'$ such that 
\begin{align*}
   \EEE \mathcal{H}^{d-1} \Big( \bigcup_{i=1}^N J_{u(s_i)}  \Big)  - \mu\EEE & \le     \EEE \sum_{i=1}^N \mathcal{H}^{d-1} \big(J_{u(s_i)}  \cap V_i \big) \le \sum_{i=1}^N \liminf_{n \to \infty} \frac{1}{2\varepsilon_n}\L^d\big(\hen{\gamma_n(s_i)} \cap V_i \big).
\end{align*}
Then, recalling that $\{\gamma_{n}(\tau)\neq 0\} \subset \Gamma_n(t)$ for all $\tau \le t$, see \eqref{limitinggamma}, \EEE   and using  \eqref{eq:lambdaepsilondefinition}, we derive 
\begin{align*}
   \EEE \mathcal{H}^{d-1} \Big( \bigcup_{i=1}^N J_{u(s_i)}  \Big)  - \mu \EEE  \le \liminf_{n\to\infty} \frac{1}{2\varepsilon_n}\L^d\big(U^{\mathcal{T}}_n(\Gamma_n(t))\big) = \lambda_0(t).
\end{align*}
Exhausting  $I_{\infty,0}^t$ by sending $N \to \infty$ and recalling that $\mu>0 $ was arbitrary, the proof is concluded.
\end{proof}

\EEE
%
%
%
%

Next, we concern ourselves with the extension of the displacement fields onto the entire time interval. This is achieved by a limit from below, namely for each $t \notin I_{\infty,0}$ we choose a sequence $(t_n)_n \subset I_{\infty,0}$ with $t_n\nearrow t$, and  we \EEE define
\begin{align*}
    u(t):= \lim_{n\to \infty} u(t_n),
\end{align*}
where the limit has to be understood  in \EEE the $SBV^2$-sense. For this type of extension,  in \EEE \cite[Lemma 3.8]{FrancfortLarsen2003} it has been shown that   the following properties hold.
\begin{theorem}
    \label{theorem:continuationtheorem}
    For every $t\in [0, 1]$, we have $u(t)\in SBV^2(\Omega')$ with $u(t) = g(t) $ on $\Omega' \setminus \overline{\Omega}$ such that $\nabla u\in L^\infty([0, 1]; L^2(\Omega'; \R^d))$ and \EEE $\nabla u$ is left continuous in $[0, 1]\setminus I_{\infty,0}$ with respect to the strong $L^2(\Omega';\R^d)$-topology. Additionally, with $\Gamma$ as defined as in (\ref{eq:Gammadefinition}), the following properties hold: 
    \begin{enumerate}[label=(\roman*)]
        \item[\rm (i)] for all $t\in [0, 1]$,
        \begin{align}
            \label{eq:jumpisingamma}
            J_{u(t)}   \,\Tilde{\subset} \,\EEE \Gamma(t)
        \end{align}
        and for $\lambda_0$ as in Lemma \ref{lemma:firstsummary} we have 
        \begin{align}
            \label{eq:lambdainequalityentiretimeinterval}
           \lambda_0\EEE(t)\geq G \mathcal{H}^{d-1}(\Gamma(t)).
        \end{align}
        \item[\rm (ii)] for all $\bar{u} \in SBV^2(\Omega')$ with $\bar{u}=g(0)$ on $\Omega'\setminus \overline{\Omega}$
        \begin{align} 
            \label{eq:minimalityzeroentiretimeinterval}
            \int_{\Omega}|\nabla u(0)|^2\,\derx + \mathcal{H}^{d-1}(J_{u(0)})\leq \int_{\Omega}|\nabla \bar{u}|^2\,\derx + \mathcal{H}^{d-1}(J_{\bar{u}}). 
        \end{align}
        \item[\rm (iii)] for all $t\in (0, 1]$ and for all $\bar{u}\in SBV^2(\Omega')$ with $z=g(t)$ on $\Omega' \setminus \overline{\Omega}$
        \begin{align}
            \label{eq:minimalityentiretimeinterval}
            \int_{\Omega}|\nabla u(t)|^2\,\derx \leq \int_{\Omega}|\nabla \bar{u}|^2\,\derx+ \mathcal{H}^{d-1}(J_{\bar{u}}\setminus\Gamma(t)).
        \end{align}
    \end{enumerate}
    Finally, 
    \begin{align}
        \label{eq:energyequalityentiretimeinterval}
        \mathcal{E}(t)   \ge  \mathcal{E}(0) + 2\int_0^t\int_{\Omega}\nabla u(s)\cdot \partial_t\nabla  g(s)\,\derx\,\ders,
    \end{align}
    where 
    \begin{align}
        \label{eq:energydefinition}
        \mathcal{E}(t):= \int_{\Omega}|\nabla u(t)|^2\,\derx+ \mathcal{H}^{d-1}(\Gamma(t)). 
    \end{align}
\end{theorem}


The proof is similar to the one in \cite[Propostion 5.9]{Giacomini2005}, fundamentally based on the stability result in Theorem \ref{theorem:stabilityresulteps}. We sketch it in Appendix \ref{sec:appendix}. We come to a final preliminary result before  proving \EEE Theorem~\ref{theorem:griffithmin}. It extends the compactness result of   displacements  given in Lemma \ref{lem:truncationIinfinity} to general times.    
\begin{lemma}[Compactness of displacements]
    \label{lem:convergenceoutsideIinfinity}
For every $t\in  [0, 1] $,  \EEE we have that 
    \begin{align*}
    \nabla z_n(t) \rightharpoonup\nabla u(t)\quad\emph{weakly in }L^2(\Omega;\R^d). 
    \end{align*}
\end{lemma}
 The \EEE proof essentially follows the lines   of \EEE the proof in \cite[Proposition 5.10]{Giacomini2005}, by combining the minimality property in   Theorem \ref{theorem:stabilityresulteps} with \EEE  Lemma \ref{lem:alternativetruncation} and Theorem~\ref{theorem:continuationtheorem}. We defer it to  Appendix \ref{sec:appendix}.  We are now able to conclude the proof of Theorem \ref{theorem:griffithmin}.
\begin{proof}[Proof of Theorem \ref{theorem:griffithmin}]
Let  $\varepsilon_n\to 0$ be the sequence  given by Lemma \ref{lem:truncationIinfinity}, \EEE and define the  corresponding sequence $h_n := h(\eps_n)\to 0$.  
In view of the results given in Lemma \ref{lemma:firstsummary} and Theorem \ref{theorem:continuationtheorem},  we see that $t \mapsto (u(t),\Gamma(t))$ is a quasi-static  crack evolution in the sense of Definition~\ref{main def-lim}, except for one inequality in the energy balance. Thus, to conclude it remains to prove 
\begin{align}\label{tooopro}
       \mathcal{E}(t)  \le  \mathcal{E}(0) + 2\int_0^t\int_{\Omega}\nabla u(s)\cdot \partial_t\nabla  g(s)\,\derx\,\ders,
\end{align}
and \EEE the convergence  properties stated in Theorem \ref{theorem:griffithmin}.

 We defer the proof of (i) to the end and start with  \EEE \eqref{tooopro} and (ii). \EEE As $   \mathcal{E}_{n} \EEE (\cdot,\cdot, \emptyset)$ defined in  \eqref{eq:energydefinitionepsilon-new}  $\Gamma$-converges to $E$ defined in \eqref{limitgama}, see \cite[Theorem 5.1]{SchmidtFraternaliOrtiz2009} and the subsequent remark therein on the incorporation of boundary conditions, the minimality properties given in Theorem \ref{theorem:approx}(i) and Definition \ref{main def-lim}(i) imply that $ \lim_{n\to\infty}  \E_{n} \EEE (u_n(0), \gamma_n(0), \emptyset) = E(u(0))$. Using the definition of $\Gamma_n(0)$, $\Gamma(0)$, and $ \mathcal{E}(0)$ in \eqref{eq:energydefinition}, this yields
   \begin{align}
        \label{eq:energyconvergencetimezero}
        \lim_{n\to\infty}  \E_{n} \EEE (u_n(0), \gamma_n(0), \Gamma_n(0))=\mathcal{E}(0).
    \end{align}
For $s\in [0, 1]$,     since $\gamma_n(s)\equiv0$ on $\Omega\setminus\Gamma_n(s)$ and $\gamma_n(s)=\nabla u_n(s)$ on $\Gamma_n(s)$, see \eqref{limitinggamma}, \EEE we can split the energy into the two parts
\begin{equation}
    \begin{aligned}\label{eineneuenummer}
        &\int_{\Omega}(\nabla u_n(s)-\gamma_n(s)) \cdot \partial_t\nabla  g_n(s)\,\derx \\=&\int_{\Omega \setminus Q_n(s)}(\nabla u_n(s)-\gamma_n(s))\cdot  \partial_t\nabla  g_n(s)\,\derx + \int_{Q_n(s)}(\nabla u_n(s)-\gamma_n(s))\cdot\partial_t\nabla  g_n(s)\,\derx  \\
        =\,&\int_{\Omega \setminus Q_n(s)}\nabla u_n(s)\cdot \partial_t\nabla  g_n(s)\,\derx +\int_{Q_n(s)\setminus \Gamma_n(s)}\nabla u_n(s)\cdot \partial_t\nabla  g_n(s) \,\derx,
    \end{aligned}
\end{equation}
   where \EEE we used that $\Gamma_n(s) \subset Q_n(s)$. \EEE We want to show that the second term vanishes for $n\to\infty$. By Hölder's inequality and \eqref{eq:energybound} it holds that 
    \begin{align*}
        \left|\int_{Q_n(s)\setminus \Gamma_n(s)}\nabla u_n(s) \cdot\partial_t\nabla  g_n(s) \,\derx\right|
        \leq C_1^\frac{1}{2}\left(\int_{Q_n(s)\setminus\Gamma_n(s)}|\partial_t\nabla  g_n(s)|^2\,\derx\right)^\frac{1}{2}.
    \end{align*}
 After \EEE taking the integral in time, we see that 
    \begin{align*}
        \left|\int_0^t\int_{Q_n(s)\setminus \Gamma_n(s)}\nabla u_n(s)\cdot \partial_t\nabla g_n(s)\,\derx\,\ders\right|  \le \EEE  C\int_0^t \left(\int_{Q_n(s)}|\partial_t\nabla  g_n \EEE (s)|^2\,\derx\right)^\frac{1}{2}\,\ders
    \end{align*}
     vanishes for $n \to \infty$ due to the regularity properties of $g$ and $\mathcal{L}^d(Q_n(s))\to 0$,  see \EEE \eqref{eq:triangleentozero}. 
  
     Recall the definition in \eqref{zn}. \EEE As $\nabla z_n(s) \rightharpoonup \nabla u(s) \text{ weakly in }L^2(\Omega; \R^d)$ for $s\in [0, 1]$ by Lemma~\ref{lem:convergenceoutsideIinfinity}, and $\partial_t\nabla  g_n(s) \to \partial_t\nabla  g(s)$ strongly in $L^2(\Omega; \R^d)$ for a.e.\ $s\in [0, 1]$,  we deduce   by  \eqref{eq:triangleentozero} \EEE
    \begin{align*}
        \lim_{n\to\infty}\int_{\Omega\setminus Q_n(s)} \nabla u_n(s) \EEE  \cdot \partial_t\nabla  g_n(s)\,\derx = \int_{\Omega}\nabla u(s)\cdot \partial_t\nabla  g(s)\,\derx
    \end{align*}
 for a.e.\ $s \in [0,1]$. By \eqref{eineneuenummer}, \EEE \eqref{eq:energybound}, and the dominated convergence theorem we thus get, for all $t \in [0,1]$, 
    \begin{align}\label{analog2}
        \lim_{n\to\infty} \int_0^t \int_{\Omega}(\nabla u_n(s)-\gamma_n(s)) \cdot \partial_t\nabla  g_n(s)\,\derx \, {\rm d}s=  \int_0^t \int_{\Omega}\nabla u(s)\cdot \partial_t\nabla  g(s)\,\derx\, {\rm d}s.
    \end{align} 
  This combined with Lemma \ref{lem:bothinequalitesdisc} and \EEE \eqref{eq:energyconvergencetimezero} implies
    \begin{align}\label{lowi}
        \lim_{n\to\infty}  \mathcal{E}_{n}(u_n(t), \gamma_n(t), \Gamma_n(t)) =
        \,& \,\mathcal{E}(0)+2\int_0^t\int_\Omega \nabla u(s)\cdot \partial_t\nabla  g(s)\,\derx\,\ders \quad \text{ for all $t \in [0,1]$.} 
    \end{align}
 Then, using    $\nabla z_n(t) \rightharpoonup \nabla u(t)$ weakly in $L^2(\Omega; \R^d)$  for every $t \in [0, 1]$ (see Lemma \ref{lem:convergenceoutsideIinfinity}) \EEE and Lemma \ref{lemma:firstsummary} we get   
    \begin{align}\label{lsc}
  \mathcal{E}(t)   \le       \liminf_{n\to\infty}  \mathcal{E}_{n}(u_n(t), \gamma_n(t), \Gamma_n(t)) =
        \,& \,\mathcal{E}(0)+2\int_0^t\int_\Omega \nabla u(s)\cdot \partial_t\nabla  g(s)\,\derx\,\ders \quad \text{ for all $t \in  [0, 1]$.} 
    \end{align}    
This shows \eqref{tooopro}. \EEE Now,   \eqref{energybalance} along with \eqref{lowi} shows the energy convergence stated in (ii). \EEE 

Concerning (iii), with \eqref{zn} and  Lemma \ref{lem:convergenceoutsideIinfinity} we get
    \begin{align*}
        \liminf_{n\to\infty}\intnormal{u_n(s)-\gamma_n(s)}\geq \liminf_{n\to\infty}\int_\Omega|\nabla z_n(s)|^2\,\derx\geq \intnormal{u(s)}
    \end{align*}
    for every $s\in   [0, 1] \EEE$ by a lower-semicontinuity argument. \EEE  Additionally, \eqref{eq:lambdainequalityentiretimeinterval}  and the definition  in \eqref{eq:lambdaepsilondefinition} give \EEE    
    \begin{align*}
        \liminf_{n\to\infty}
            \frac{1}{2\varepsilon_n}\L^d\big(\Uhnen(\Gamma_n(s))\big) \geq \mathcal{H}^{d-1}(\Gamma(s)). 
    \end{align*}
 As the sum of the left-hand sides is equal to the sum of the right-hand sides by (ii), both inequalities are actually equalities. Thus, the statement of (iii) holds for all $s\in  [0,1] \EEE$. 
 We continue with proving \EEE (iv). Considering any open set $V \subset \Omega'$ and $s \in [0,1]$, and repeating the reasoning in the proof of Lemma \ref{lemma:firstsummary} (see \eqref{gammaV}) on $V$, we get
\begin{align}\label{measure}
 \liminf_{n\to\infty} \frac{1}{2\varepsilon_n}\L^d\big(U^{\mathcal{T}}_n(\Gamma_n(t)) \cap V\big)  \ge \mathcal{H}^{d-1} ( \Gamma(s) \cap V). 
 \end{align}
Given any $A \subset \Omega'$ with $\mathcal{H}^{d-1}(\partial A \cap \Gamma(s)) = 0$, by applying \eqref{measure} for $V = A$ and $V = \Omega' \setminus \overline{A}$ and using the convergence of the surface energy given in (iii) we conclude 
\begin{align*}
 \frac{1}{2\varepsilon_n}\L^d\big(U^{\mathcal{T}}_n(\Gamma_n(t)) \cap A\big)  = \mathcal{H}^{d-1} ( \Gamma(s) \cap A). 
 \end{align*}
The fact that this holds for all $A$ with $\mathcal{H}^{d-1}(\partial A \cap \Gamma(s)) = 0$ shows (iv).  \EEE 

It remains to prove (i). \EEE For that matter, we first note that by definition $\H^{d-1}(\Gamma(t))$ can only be discontinuous for $t\in I_{\infty, 0}$. \EEE
Now we repeat the proof of Lemma \ref{lem:truncationIinfinity} for  given $t \in [0,1] \setminus I_{\infty,0}$ to find $\hat{u}(t) \in SBV^2(\Omega')$  such that $u_n(t)  \to \hat{u}(t)$ in $L^1(\Omega')$ and  $z_n(t)   \to \hat{u}(t)$  in $SBV^2(\Omega')$ along a not relabeled subsequence depending on $t$. Note that $\hat{u}(t) = g(t)$ on $\Omega' \setminus \overline{\Omega}$. Additionally, for $\tau\in I_{\infty, 0}$, $\tau<t$, repeating the  lower-semicontinuity argument in the proof of Lemma~\ref{lemma:firstsummary}, we get 
\begin{align*}
    \H^{d-1}(J_{\hat{u}(t)}\cup \Gamma(\tau))\leq \liminf_{n\to \infty} \frac{1}{2\varepsilon_n}\L^d\big(U^\T_n(\{\gamma_n(t)\neq 0\}\cup \Gamma_n(\tau))\big) \leq  \frac{1}{2\varepsilon_n}\liminf_{n\to \infty} \L^d\big(U^\T_n(\Gamma_n(t))\big),
\end{align*}
where  we used  $\{\gamma_n(t)\neq 0\}\cup \EEE  \Gamma_n(\tau)\subset \Gamma_n(t)$. \EEE Then, continuity of $\mathcal{H}^{d-1}(\EEE\Gamma(\cdot))$ at $t$ leads to  
\begin{align*}
    \H^{d-1}(J_{\hat{u}(t)}\cup \Gamma(t)) \leq  \frac{1}{2\varepsilon_n}\liminf_{n\to \infty} \L^d\big(U^\T_n(\Gamma_n(t))\big).
\end{align*} \EEE
This together with (iii) yields $J_{\hat{u}(t)}\subset \Gamma(t)$ up to a $\H^{d-1}$-null set. It \EEE  suffices to show that $u(t)= \hat{u}(t)$ on $G(t)$ as then the statement follows from Urysohn's principle. To see this, we observe that by Lemma \ref{lem:convergenceoutsideIinfinity}
we have  $\nabla z_n(t) \rightharpoonup\nabla u(t)$ weakly in $L^2(\Omega;\R^d)$ which \EEE along with $z_n(t)   \to \hat{u}(t)$  in $SBV^2(\Omega')$ shows that $\nabla u(t) = \nabla \hat{u}(t)$ on $\Omega'$. This yields that $u(t) - \hat{u}(t)$ is a piecewise constant function whose jump set is contained in   $\Gamma(t)$. 
As $u(t) - \hat{u}(t) = 0$ on  $\Omega' \setminus \overline{\Omega}$, we get $u(t) - \hat{u}(t) = 0$ on  $G(t)$ by the definition of $G(t)$, see \cite[Proof of Theorem 2.2, Step 2]{FriedrichSteinkeStinson2025} for the precise argument.  \EEE 
\end{proof}

%
%
%
%

\section{The stability result}
\label{section:stabilityresult}

This section is devoted to the proof of Theorem \ref{theorem:stabilityresulteps}. The key method is a \textit{jump set transfer} in the spirit of \EEE \cite{FrancfortLarsen2003}. 
\paragraph{\textbf{A density result}} We start by observing that it suffices to prove the statement for functions $\phi$ \EEE with more regularity, employing a suitable density argument. Let $\mathcal{W}(\Omega')\subset SBV(\Omega')$ be the set of functions $v$ with closed jump set $J_v$, which is included in a finite union of closed and connected pieces of $C^1$-hypersurfaces, and with $v\in W^{2, \infty}(\Omega' \EEE \setminus J_v)$.   For each $v \in SBV^{2}(\Omega')$ with $v = g(t)$ on $\Omega'  \setminus \overline{\Omega}$ we can choose a sequence $(v_k)_k \subset\mathcal{W}(\Omega') $ with   $v_k = g(t)$  on \EEE  $\Omega'  \setminus \overline{\Omega}$   such that 
    \begin{align}\label{densiiit}
        & v_k\to v \text{ strongly in }L^1(\Omega), \notag \\
        & \nabla v_k\to \nabla v \text{ strongly in }L^2(\Omega; \R^{ d}), \notag \\
        & \limsup_{k\to\infty} \mathcal{H}^{d-1}(J_{v_k}\triangle J_v)=0.
    \end{align}
This follows as explained in \cite[Beginning of Section 5]{FriedrichSeutter2025}, see also  \cite[Theorem 3.2]{FriedrichSteinkeStinson2025} for an analogous statement and proof for $GSBD$-functions.

Equipped with this result, it suffices to construct a recovery sequence for $\phi\in \mathcal{W}(\Omega')$ with $\phi=g(t)$ on $\Omega' \setminus \overline{\Omega}$. Furthermore, fixing  $\theta>0$, we only need to find a sequence of pairs $(\phi_n, \gamma^\phi_n)\in V\hn(\Omega')\times W\hn(\Omega')$ with $\phi_n= g_n \EEE (t)$ and $\gamma^\phi_n=0$ on $\Omega' \setminus \overline{\Omega}$ such that   
\begin{equation}\label{only thing to check}
    \begin{aligned}
        &\lim_{n\to\infty}\|(\nabla \phi_n-\gamma^\phi_n)-\nabla\phi\|_{L^2(\Omega)}\leq C\theta, \\
        &\limsup_{n\to\infty}\frac{1}{2\varepsilon_n}\L^d\left(\hen{\gamma^\phi_n}\setminus   U^\T_n\EEE(\bruchdiscepsilonn(t))\right)\leq \mathcal{H}^{d-1}(J_\phi\setminus J_{u(t)})+C\theta.
    \end{aligned}
\end{equation}
Then, the statement in \eqref{eq:stabilityresulteps} follows by a suitable diagonal argument, sending $k \to \infty$ (for a sequence as in \eqref{densiiit}) and $\theta \to 0$, by using also the minimality \eqref{eq:3.3} in the time-discrete problems. More precisely, subtracting the surface term, \eqref{eq:3.3} implies
$$\int_\Omega |\nabla u_n(t)-\gamma_n(t)|^2\,\derx \le \int_\Omega |\nabla \phi_n-\gamma^\phi_n|^2\,\derx  + \frac{1}{2\varepsilon_n}\L^d\left(\hen{\gamma^\phi_n}\setminus   U^\T_n\EEE(\bruchdiscepsilonn(t))\right),$$
and then we also use \EEE lower semicontinuity for the elastic energy based on \eqref{eq:truncationIinfinityconvergence}. \EEE

   From now on,  fix $\phi\in \mathcal{W}(\Omega')$ with $\phi=g(t)$ on $ \Omega' \setminus \overline{\Omega} \EEE $.    For simplicity we only treat the case that $\mathcal{H}^1(J_\phi \cap \partial_D \Omega)=0$ (no jump along the boundary), for the general case follows by minor adaptations of the construction at the boundary (see \cite{FrancfortLarsen2003}) which would merely overburden notation in the sequel.

\EEE  

\paragraph{\textbf{Besicovitch covering}} For \EEE the construction of $(\phi_n)_n$, the idea is  \EEE to transfer the jump set of $\phi$ onto the one of $u_n(t)\mathds{1}_{\bruchdiscepsilonn(t)}$.  The first step is to cover the jump set of $\phi$ with  by following the procedure in \cite[Section 2]{FrancfortLarsen2003}.

We define $E_t$ to be the set of all Lebesgue-density 1 points for $\{x\in \Omega:u(x) > t \EEE \}$ with $u:= u(t)$ and $E^n_t$ to be the set of all Lebesgue-density 1 points for $\{x\in \Omega:y_n(x) > t \EEE \}$,  where \EEE  $y_n:= u_n(t)\mathds{1}_{ (\bruchdiscepsilonn(t))^c}$. We can choose a subset $G_j\subset J_u$ like in \cite[Equation (2.2)]{FrancfortLarsen2003},  consisting of the points with jump height larger than $1/j$, \EEE such that 
\begin{align*}
    \H^{d-1}(J_u\setminus G_j)\leq \theta.
\end{align*}
For every $x\in G_j\cap J_\phi$, we consider closed cubes $Q_r(x)$ with sidelength $2r$ and  two faces normal to the unit normal vector $\nu_\phi(x)$ at $J_\phi$.  For $\H^{d-1}$-a.e.\ $x\in J_\phi\setminus J_u$, we also consider such cubes $Q_r(x)$, \EEE for $r$ sufficiently small such that
\begin{align}
    \label{eq:globalestimateJucapQrinGphijwithoutJu}
    \H^{d-1}(J_u\cap Q_r(x))\leq \theta r^{d-1}.
\end{align}
Here we use that  $J_u$ has   $\mathcal{H}^{d-1}$-density  zero  $\H^{d-1}$-a.e.\ \EEE in $J_\phi\setminus J_u$. Since $J_\phi$ is contained in a finite union of closed $C^1$-hypersurfaces, for  $\H^{d-1}$-a.e.\ \EEE $x\in J_\phi$ we can choose the   above cubes \EEE such that, possibly upon  passing \EEE to a smaller $r$,  we have \EEE
\begin{align}
    \label{eq:propertiesforcubesforu}
    r^{d-1}&\leq  2 \EEE \H^{d-1} (J_\phi\cap Q_r(x)) \\
    \label{eq:continuityofnormal}
    J_\phi\cap Q_{ r}(x)&\subset \{y \in  \R^d \colon \, \EEE |(y-x)\cdot \nu_\phi(x)|\leq\theta r\}.
\end{align}
With this construction, we obtain a fine cover of $\Xi:= (G_j\cap J_\phi)\cup (J_\phi\setminus J_u)$, to which we can apply the Besicovitch covering theorem with respect to the Radon measure $\L^d+\H^{d-1}|_{ J_u}$. For $\theta>0$, we can therefore find a finite and disjoint subcollection $\mathcal{B}:= (Q_{r_i}(x_i))_i$, or short $(Q_i)_i$, such that the cubes satisfy the above conditions, as well as 
\begin{align}
    \label{eq:propertiesforcubes}
    \L^d\left({\bigcup}_{\mathcal{B}}Q_i\right)\leq  \theta^2, \EEE \quad \H^{d-1}\big(\Xi\setminus {\bigcup}_\mathcal{B}Q_i\big)\leq \theta.
\end{align}
Here and in the following, we use $\bigcup_\mathcal{B}Q_i$ as a shorthand notation for $\bigcup_{Q_{r_i}(x_i)\in \mathcal{B}}Q_{r_i}(x_i)$. 
By H\"older's inequality and \eqref{eq:globalestimateJucapQrinGphijwithoutJu} \EEE this implies
\begin{align}
        \label{eq:propertiesbesicovitchcovering}
    \int_{{\bigcup}_\mathcal{B}Q_i}|\nabla u| \,\derx \leq C\theta, \quad  \quad \H^{d-1}\big(J_\phi\setminus {\bigcup}_\mathcal{B}Q_i\big)\leq 2\theta.
\end{align}
Without further notice, we will \EEE frequently use the fact that the  cubes \EEE are pairwise disjoint. By $\mathcal{B}_{\text{good}}\subset \mathcal{B}$ we denote the collection of cubes \EEE $Q_i$ with $x_i\in J_\phi\setminus J_u$ and similarly we let $\mathcal{B}_\text{bad}\subset \mathcal{B}$ be the collection of all cubes  \EEE $Q_i$ with $x_i\in J_\phi \cap G_j$. We \EEE have $\mathcal{B} = \mathcal{B}_\text{good}\cup \mathcal{B}_\text{bad}$,  and  we also define the set \EEE 
\begin{align*}
 B_\text{bad}={\bigcup}_{\mathcal{B}_\text{bad}}Q_i.
\end{align*}

\EEE

\paragraph{\textbf{Good cubes}} On the good cubes,  by \eqref{eq:globalestimateJucapQrinGphijwithoutJu}--\eqref{eq:propertiesforcubesforu}  we get \EEE 
\begin{align}
    \label{eq:goodcubesestimate}
    {\bigcup}_{ Q_i \EEE  \in \mathcal{Q}_\text{good}}\mathcal{H}^{d-1}(J_u\cap Q_i)\leq \theta \mathcal{H}^{d-1}(J_\phi)\leq C\theta. 
\end{align}
For convenience but with a slight abuse of notation, from now on we will only consider bad cubes whenever taking sums or unions. In particular, we write $\bigcup_i Q_i = B_\text{bad}$. \EEE

\paragraph{\textbf{Bad cubes: jump transfer}}  Recall that $J_{y_n} \subset \partial \Gamma_n(t)$. \EEE Using the methods in \cite[Proof of Theorem~2.1]{FrancfortLarsen2003}, we can choose $t_i\in \R$  (depending on $n$) \EEE such that   
\begin{align}
    \label{eq:EntiwithoutJun}
    \sum\nolimits_i\mathcal{H}^{d-1}\big((\partial^*E^n_{t_i}\cap Q_i)\setminus J_{y_n}\big)\le \theta^{d+1}
\end{align}
for all $n \in \N$, \EEE see \cite[Equation (2.7)]{FrancfortLarsen2003} for details. This can be done in such a way that (up to a sign) $\nu(x_i) := \nu_\phi(x_i)$ is the  outward normal of $\partial^* E_{t_i}$ at $x$. \EEE Setting $H_i(s):= \{y\in Q_i:(y-x_i)\cdot \nu(x_i)=s\}$, one can choose $s_{n, i}^{\pm}$ such that $H_{n, i}^\pm:= H_{i}(s_{n, i}^\pm)$ satisfies
\begin{align}
    \begin{cases}
        \label{eq:estimatesnegativerectangle}
        \mathcal{H}^{d-1}(H^-_{n,i}\setminus E^n_{t_i}) + \mathcal{H}^{d-1}(H^-_{n,i}\setminus E_{t_i})\le 8 \theta^{d+1} (2r_i)^{d-1} \\
               \mathcal{H}^{d-1}(H^+_{n,i}\cap E^n_{t_i}) +   \mathcal{H}^{d-1}(H^+_{n,i}\cap E_{t_i}) \le 8 \theta^{d+1} (2r_i)^{d-1} \\
             \text{dist}(H^\pm_{n,i}, H_i)\in [\frac{\theta}{2}r_i, \theta r_i], 
    \end{cases}
\end{align}
see \cite[Equation (2.10)--(2.11)]{FrancfortLarsen2003} for details. ($\pm$ stands as a placeholder for $+$ and $-$.)  We denote the rectangle between $H_{n, i}^-$ and $H_{n, i}^+$ inside of $Q_i$ by $R_{n, i}$.  

Note that  by \eqref{eq:continuityofnormal} we get
\begin{align}
    \label{eq:G_iwithoutR_i}
    \H^{d-1}(G_j\setminus R_{n, i} \EEE )=0.
\end{align}
The basic idea in \cite{FrancfortLarsen2003} consists in transferring the jump of $\phi$ onto $\partial^*E^n_{t_i}$, which by \eqref{eq:EntiwithoutJun} lies almost entirely in $J_{y_n}$. It also cuts off $\partial^*E^n_{t_i}$ at $R_{n, i}$, such that the only jump within $Q_i$ occurs in $R_{n, i}$. We will   alter some arguments in their produce to suitable adapt to the eigenfracture setting.   

The first aspect in this direction is that, although  $\mathcal{H}^{d-1}(\bigcup_i \partial^* E^n_{t_i}\cap Q_i\setminus J_{y_n})$ \EEE is small thanks to \eqref{eq:EntiwithoutJun}, its neighborhood, being a relevant quantity in the energy \eqref{eq:energydefinitionepsilon-new}, could be large. To solve \EEE this problem, we introduce a suitable notion of separation. For \EEE a fixed cube $Q_i$, we define the \emph{precrack} $\Psi^n_{i, \text{pre}}:= J_{y_n}\cap Q_i$. Let $D^+_{n, i}$ and $D^-_{n, i}$ be the two connected components of $Q_i \setminus R_{n, i}$.  We say that a closed set $\Phi \subset Q_i$ with $\mathcal{H}^{d-1}(\Phi) < + \infty $ is \emph{separating} if $D^+_{n, i}$ and $D^-_{n, i}$ are contained in two different connected components of $Q_i \setminus   \Phi  $.  

Let $\Phi^n_i \subset Q_i$ be such that $\Phi^n_i \cup \Psi^n_{i, \text{pre}}$ is separating. We suppose that $\Phi^n_i$ is chosen minimal in the sense that 
\begin{align}\label{Psi}
\mathcal{H}^{d-1}(\Phi^n_i) \le \mathcal{H}^{d-1}(\Phi') \quad \text{for all $\Phi' \subset Q_i \EEE $ such that $\Phi' \cup  \Psi^n_{i, {\rm pre}} \EEE $ is separating}.  
\end{align}
See  Figure~\ref{figur1}  for an illustration in 2d and Figure \ref{figure:badnose}  for an explanation why we consider a separation in this sense. \EEE Note that $\partial^* ((E^n_{t_i}\cup D^-_{n, i})\setminus D^+_{n, i})$  
\EEE is separating with 
$$\mathcal{H}^{d-1}\big( \partial^* ((E^n_{t_i}\cup D^-_{n, i})\setminus D^+_{n, i}) \EEE \setminus J_{y_n}\big)\leq \mathcal{H}^{d-1}((\partial^*E^n_{t_i}\cap Q_i)\setminus J_{y_n})+\mathcal{H}^{d-1}(H_i^-\setminus E^n_{t_i})+\mathcal{H}^{d-1}(H_i^+\cap \EEE E^n_{t_i}).$$ 
Thus, in view of \eqref{Psi}, by summing  over all $i$ (meaning all bad cubes) we get  
\begin{equation}
    \begin{aligned}
           \sum_i\mathcal{H}^{d-1}( \Phi^n_i )\leq \,& \sum_i\mathcal{H}^{d-1}\big((\partial^* ((E^n_{t_i}\cup D^+_{n, i})\setminus D^-_{n, i}))\setminus J_{y_n}\big)  \leq C \theta^{d+1}, \EEE  
    \end{aligned}
    \label{psiestimate}
\end{equation}
where we have used \eqref{eq:propertiesforcubesforu}, \eqref{eq:EntiwithoutJun}, and the first and second \EEE inequality in \eqref{eq:estimatesnegativerectangle}. Here, the constant $C$ depends also on $\mathcal{H}^{d-1}(J_u)$.

%

\paragraph{\textbf{Definition of the competitors.}} \EEE
For every cube $Q_i\in \mathcal{Q}_\text{bad}$, we  now define   $ \Psi^n_i:= \Psi^n_{i, \text{pre}}\cup \Phi^n_i$, with $\Phi^n_i$ chosen minimally as above.   As $\Psi^n_i$ splits $Q_i$ into two parts, let us denote the part that contains  $D^\pm_{n, i}$ by $L^\pm_{n, i}$. \EEE Then, we can define the function $\phi'_{n,i}$ \EEE which inside of $Q_i$ exhibits discontinuities only on $\Psi^n_i$ \color{black}. This is achieved by employing the reflection argument in \cite{FrancfortLarsen2003}, see particularly \cite[below Equation (2.14)]{FrancfortLarsen2003}. \EEE We first define $\phi^-_{n, i}:= \phi \chi_{Q^-_{i}\setminus R_{n, i}}$, extended to $R_{n, i}$ by reflection,  where $Q^-_{i} = \{y\in Q_i:(y-x_i)\cdot \nu(x_i) \le 0\}$. \EEE  We proceed similarly for $\phi^+_{n, i}$. 
Now, we define 
\begin{align}
    \phi'_{n, i}:= 
    \begin{cases}
        \phi^-_{n, i} &\text{on } L^-_{n, i}, \\
        \phi^+_{n, i} &\text{on } L^+_{n, i}. \\
    \end{cases}
\end{align}
By this construction we particularly get, for all $Q_i \in \mathcal{Q}_{\rm bad}$, \EEE
\begin{align}\label{eq.bbb}
\Vert \nabla  \phi'_{n,i}\Vert_{L^\infty(Q_i)} & \le C\Vert \nabla \phi \Vert_{L^\infty(\Omega')},\\
J_{\phi'_{n,i}} & \subset \Psi^n_i\cup (\partial Q_i\cap R_{n, i}),\label{eq.bbbq2} \EEE
\end{align}
 where for \eqref{eq.bbbq2} we used \eqref{eq:G_iwithoutR_i}.  We are now ready to define the sequence $\phi_n$ appearing \EEE in \eqref{only thing to check}. First, we define an auxiliary function $\phi'_n$ which is defined as $\phi'_n = \phi'_{n,i}$ in each bad cube  $Q_i\in \mathcal{Q}_{\text{bad}}$ and  $\phi'_n=\phi$ everywhere else. Then, the function   $\phi_n \in V_{h_n}(\Omega')$ denotes the continuous, piecewise affine interpolation of $\phi'_n$ with respect to $\mathcal{T}\hn$. By construction,   we \EEE have $J_{\phi'_n}= J_\phi$ outside of $B_{\rm bad}$,  \EEE whereas inside   $B_{\rm bad}$ \EEE  we have $J_{\phi'_n}\subset \bigcup_i \Psi^n_i\cup (\partial Q_i\cap R_{n, i})$.  

    \newcommand{\triangleobenlinks}[2]{\fill[blue, opacity = 0.4] (#1, #2) -- (#1, #2+0.5) -- (#1+0.5, #2+0.5) -- (#1, #2);}
    \newcommand{\triangleuntenrechts}[2]{\fill[blue, opacity = 0.4] (#1, #2) -- (#1+0.5, #2) -- (#1+0.5, #2+0.5) -- (#1, #2);}
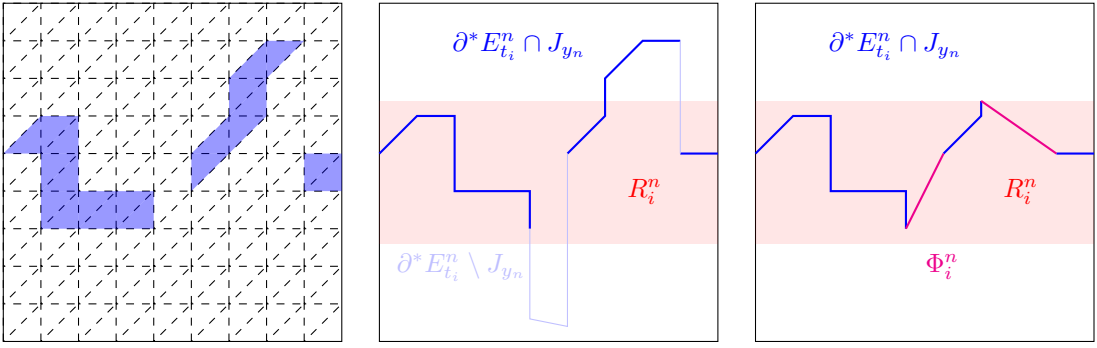
\begin{figure}
\begin{tikzpicture}
    \draw (0, 0) rectangle (4.5, 4.5);
    \foreach \n in {0, ..., 8}
        {\foreach \m in {0, ..., 8}
            {
            \draw[dashed] (0.5*\n, 0.5*\m) -- (0.5*\n, 0.5*\m+0.5) -- (0.5*\n+0.5, 0.5*\m+0.5) -- (0.5*\n, 0.5*\m); 
            }    
        }

    \triangleobenlinks{2.5}{2.5-0.5}
    \triangleobenlinks{3}{3.5-0.5}
    \triangleobenlinks{3}{3-0.5}
    \triangleobenlinks{3.5}{4-0.5}

    \triangleuntenrechts{2.5}{3-0.5}
    \triangleuntenrechts{3}{3.5-0.5}
    \triangleuntenrechts{3}{4-0.5}


    \triangleobenlinks{4}{2.5-0.5}

    \triangleuntenrechts{4}{2.5-0.5}

    \triangleobenlinks{1.5}{2-0.5}
    \triangleobenlinks{1}{2-0.5}
    \triangleobenlinks{0.5}{2-0.5}
    \triangleobenlinks{0.5}{2.5-0.5}
    \triangleobenlinks{0.5}{3-0.5}
    
    \triangleuntenrechts{1.5}{2-0.5}
    \triangleuntenrechts{1}{2-0.5}
    \triangleuntenrechts{0.5}{2-0.5}
    \triangleuntenrechts{0.5}{2.5-0.5}
    \triangleuntenrechts{0.5}{3-0.5}
    \triangleuntenrechts{0}{3-0.5}

    
    \draw (5, 0) rectangle (9.5, 4.5);
    \fill[very nearly transparent, red] (5, 1.3) rectangle (9.5, 3.2);

    \draw[thick, blue] (6-1, 3-0.5) -- (6.5-1, 3.5-0.5) -- (7-1, 3.5-0.5) -- (7-1, 2.5-0.5) -- (8-1, 2.5-0.5) -- (8-1, 2-0.5);
    \draw[thick, blue] (8.5-1, 3-0.5) -- (9-1, 3.5-0.5) -- (9-1, 4-0.5) -- (9.5-1, 4.5-0.5) -- (10-1, 4.5-0.5);
    \draw[thick, blue] (10-1, 3-0.5) -- (10-0.5, 3-0.5) ;

    \draw[blue!25] (8-1, 2-0.5) -- (8-1, 0.3) -- (8.5-1, 0.7-0.5) -- (8.5-1, 3-0.5);
    \draw[blue!25] (10-1, 4.5-0.5) -- (10-1, 3-0.5);


    \draw (10, 0) rectangle (14.5, 4.5);
    \fill[very nearly transparent, red] (10, 1.3) rectangle (14.5, 3.2);

    \draw[thick, blue] (10, 3-0.5) -- (10.5, 3.5-0.5) -- (11, 3.5-0.5) -- (11, 2.5-0.5) -- (12, 2.5-0.5) -- (12, 2-0.5);
    \draw[thick, blue] (12.5, 3-0.5) -- (13, 3.5-0.5) -- (13, 3.2);
    \draw[thick, blue] (14, 3-0.5) -- (14.5, 3-0.5) ;

    \draw[thick, magenta] (12, 2-0.5) -- (12.5, 3-0.5);
    \draw[thick, magenta] (13, 3.2) -- (14, 3-0.5);
    
    \node[circle, label = {0: \color{red} $R^n_i$}] at (8, 2){};
    \node[circle, label = {135: \color{blue} $\partial^* E^n_{t_i}\cap J_{y_n}$}] at (8, 3.5){};
    \node[circle, label = {180: \color{blue!25} $\partial^* E^n_{t_i}\setminus J_{y_n}$}] at (7.25, 1){};

    \node[circle, label = {0: \color{red} $R^n_i$}] at (13, 2){};
    \node[circle, label = {135: \color{blue} $\partial^* E^n_{t_i}\cap J_{y_n}$}] at (13, 3.5){};
    \node[circle, label = {180: \color{magenta} $\Phi^n_i$}] at (13, 1){};




\end{tikzpicture}
\caption{First Picture: $Q_i$ with the blue triangles belonging to $\Gamma_n(t)$. Second picture: a possible $\partial^*E^n_{t_i}$. Third picture: the resulting $\Phi^n_i$.} \label{figur1}
\end{figure}

Our next task is to define a suitable $\gamma^\phi_n$ with the following properties: \EEE The support of $\gamma^\phi_n$  should cover the simplices with high strain induced by the interpolation of   $\phi_n$   but,  at the same time, \EEE   the volume of these simplices should asymptotically still have  the same size   as $\Gamma_n(t)$ \EEE in the limit $n \to \infty$. \EEE

Let us collect all such simplices.   In the bad cubes $Q_i$, all simplices intersecting $\Psi^n_i$ \EEE are collected in $\triangle^{n, {\rm bad}}_i$, and further we collect all simplices intersecting  $\partial Q_i\cap R_{n, i}$  (which are two $(d-1)$-dimensional cuboids in \EEE $\partial Q_i$) by $\triangle^{n, \text{lat}}_i$. We also define $\triangle^{n, {\rm bad}} = \bigcup_i \triangle^{n, {\rm bad}}_i$ and  $\triangle^{n, {\rm lat}} = \bigcup_i \triangle^{n, {\rm lat}}_i$. \EEE  All simplices  intersecting $J_\phi$  which are not contained in $\triangle^{n, {\rm bad}} \cup \triangle^{n, {\rm lat}}$ are collected in the set  $\triangle^{n, \text{good}}$. Therefore, we have 
\begin{align}
    \label{eq:trianglegoodcapQi}
    \triangle^{n, \text{good}}\cap B_{\rm bad} \EEE =\emptyset. 
\end{align}
\EEE 
The challenge is that we have no  immediate control on the simplices $\triangle^{n, {\rm bad}}$, but some delicate covering arguments are necessary to bound their number. \EEE Since our construction also introduces jumps along a part of the lateral boundary $\partial Q_i\cap R_{n, i}$, \EEE   we  will also need a control  on $ \# \triangle^{n, \text{lat}}_i$ to estimate their contribution.  
 
We further define $\triangle^n_i:= \triangle^{n, \text{bad}}_i\cup \triangle^{n, \text{lat}}_i$ \EEE  as well as $\triangle^n:= \triangle^{n, \text{bad}}\cup \triangle^{n, \text{lat}}\cup \triangle^{n, \text{good}}$   and set \EEE 
\begin{align}\label{finallygamma}
\gamma^\phi_n:= \nabla \phi_n\mathds{1}_{\triangle^n}.
\end{align}
Taking  \eqref{eq.bbbq2} into account, \EEE this construction ensures 
\begin{align}\label{finallygamma0}
J_{\phi_n'} \subset \triangle^n.
\end{align} 
\EEE We now separately control the three collections of simplices $\triangle^{n, \text{good}}$, $\triangle^{n, \text{lat}}$, and $\triangle^{n, \text{bad}}$. 

\paragraph{\textbf{Lateral simplices.}} \EEE Note that, by \eqref{eq:propertiesforcubesforu} and \eqref{eq:estimatesnegativerectangle}, it holds that 
\begin{align}
    \label{eq:estimatelateral}
    \L^d\left({\bigcup}_i U_{2\varepsilon_n}(R_{n, i}\cap \partial Q_i)\right)\leq C\theta\varepsilon_n.
\end{align}
\paragraph{\textbf{Good simplices.}} \EEE Recall that  we assumed that $J_\phi$ consists of finitely many $C^1$-hypersurfaces.  Thus,  recalling \eqref{heps}, \EEE we get   
\begin{align}
    \label{eq:estimategoodconvergence}
    \frac{1}{2 \varepsilon_n \EEE}\L^d\left( U_{\varepsilon_n+ h_n}\left(J_{\phi}\setminus  B_{\rm bad} \EEE \right)\EEE \right) \to  \mathcal{H}^{d-1} \EEE \left(J_\phi\setminus B_{\rm bad} \EEE \right),
\end{align}
\EEE where we exploit that the $(d-1)$-dimensional Minkowski content of regular sets coincides with their $\mathcal{H}^{d-1}$-measure.
\paragraph{\textbf{Bad simplices.}} \EEE
Observe \EEE that all simplices that are included in $\Gamma_n(t)$ do not increase the surface energy of the competitor at all, cf.\ \eqref{eq:3.3}. Below in Corollary \ref{thecorollary}  we will see that one can actually replace $\Gamma_n(t)$ by $U_{l_n}(\Gamma_n(t))$,  $l_n:= \theta\varepsilon_n+ h_n$, \EEE  without significantly increasing the energy. Recalling that  in every bad  cube $Q_i$ we have the curve $\Psi^n_i=\Psi^n_{i, \text{pre}}\cup \Phi^n_i$ with $\Psi^n_{i, \text{pre}} = J_{y_n} \cap Q_i \subset \Gamma_n(t) \cap Q_i \EEE$, 
it is therefore meaningful to split the curve $\Psi^n_i$ into  the parts \EEE  $\Phi^{n, \text{new}}_i :=\Phi^n_i \setminus U_{l_n}(\Gamma_n(t))\EEE $  and   $  \Psi^{n, \text{cur}}_i:= \Psi^n_i\cap U_{l_n}(\Gamma_n(t))$. 

We also let $\Phi^{n, \text{new}} =  \bigcup_i\Phi^{n, \text{new}}_i$ and $\Phi^{n, \text{cur}} =  \bigcup_i\Phi^{n, \text{cur}}_i$.   Additionally, we split the set $\triangle^{n, \text{bad}}_i$ into the sets $\triangle^{n, \text{cur}}_i$ and $\triangle^{n, \text{new}}_i$, where 
$\triangle^{n, \text{cur}}_i$ denotes all simplices $\triangle^{n, \text{cur}}_i \subset \triangle^{n, \text{bad}}_i$  intersecting \EEE  $U_{l_n}(\Gamma_n(t))$ and $\triangle^{n, \text{new}}_i =:  \triangle^{n, \text{bad}}_i \setminus \triangle^{n, \text{cur}}_i$.  As $\triangle^{n, \text{bad}}_i$ was defined as all the triangles intersecting $\Psi^n_i$, we know by construction that for every simplex $T  \subset \EEE \triangle^{n, \text{new}}_i\neq \emptyset$, we  have $T\cap \Phi^{n, \text{new}}_i\neq \emptyset$. 
We also set \EEE $\triangle^{n, \text{new}}:= \bigcup_i \triangle^{n, \text{new}}_i$ and $\triangle^{n, \text{cur}}:= \bigcup_i \triangle^{n, \text{cur}}_i$.  In Figure~\ref{figure:U_ln(Gamma)}, we illustrate $U_{l_n}(\Gamma_n(t))$  as well as    $\triangle^{n, \text{cur}}$ and  $\triangle^{n, \text{new}}$. \EEE

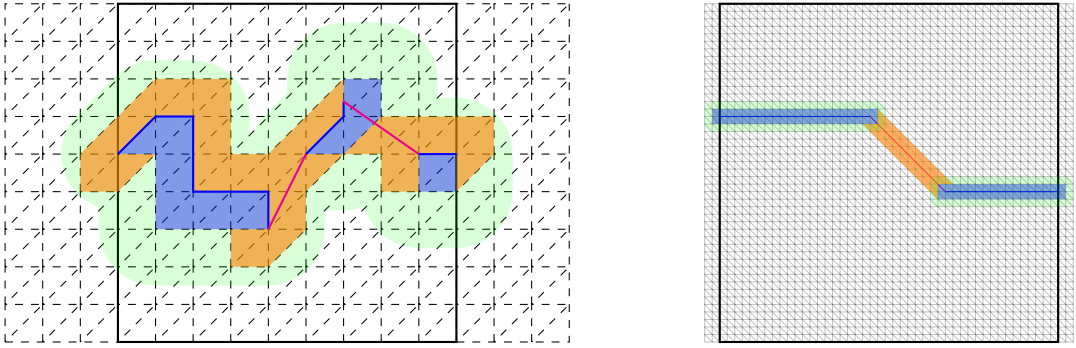
\begin{figure}
\begin{tikzpicture}
    
        \draw[thick] (4, 2) rectangle (8.5, 6.5);
    \foreach \n in {0, ..., 14}
        {\foreach \m in {0, ..., 8}
            {
            \draw[dashed] (0.5*\n+2.5, 0.5*\m+2) -- (0.5*\n+2.5, 0.5*\m+2.5) -- (0.5*\n+3, 0.5*\m+2.5) -- (0.5*\n+2.5, 0.5*\m+2); 
            }    
        }
    \draw[dashed] (2.5, 2) -- (4, 2);
    \draw[dashed] (8.5, 2) -- (10, 2) -- (10, 6.5);

    \def\givencurve{
    (4, 4.5) -- (4.5, 5) -- (5, 5) -- (5, 4) -- (6, 4) -- (6, 3.5) -- (4.5, 3.5) -- (4.5, 4.5) -- (4, 4.5)
    (8, 4.5) -- (8.5, 4.5) -- (8.5, 4) -- (8, 4) -- (8, 4.5) 
    (6.5, 4.5) -- (7, 5) -- (7, 5.5) -- (7.5, 5.5) -- (7.5, 5) -- (7, 4.5) -- (6.5, 4.5) 
    }


    \draw[line width=1.5cm, line cap=round, line join=round, draw=green, opacity = 0.15] \givencurve;

    \def\blmagShifted{
    (4,4.5) -- (4.5,5.0) -- (5,5.0) -- (5,4.0) -- (6,4.0) -- (6,3.5)
    -- (6.5,4.5) -- (7,5.0) -- (7,5.2) -- (8,4.5) -- (8.5,4.5)
  };

    \def\tobefilled{
    (3.5, 4) -- (3.5, 4.5) -- (4.5, 5.5) -- (5.5, 5.5) -- (5.5, 4.5) -- (6, 4.5) -- (7, 5.5) -- (7.5, 5.5) -- (7.5, 5) -- (9, 5) -- (9, 4) -- (7.5, 4) -- (7.5, 4.5) -- (7, 4.5) -- (6.5, 4) -- (6.5, 3.5) -- (6, 3) -- (5.5, 3) -- (5.5, 3.5) -- (4.5, 3.5) -- (4.5, 4) -- (3.5, 4) 
    };

    \def\newfilled{
    (3.5, 4) -- (3.5, 4.5) -- (4.5, 5.5) -- (5.5, 5.5) -- (5.5, 4.5) -- (6, 4.5) -- (7, 5.5) -- (7, 5) -- (6.5, 4.5) -- (7, 4.5) -- (6.5, 4) -- (6.5, 3.5) -- (6, 3) -- (5.5, 3) -- (5.5, 3.5) -- (6, 3.5) -- (6, 4) -- (5, 4) -- (5, 5) -- (4.5, 5) -- (4, 4.5) -- (4.5, 4.5) -- (4, 4) -- (3.5, 4) 
    (7, 4.5) -- (7.5, 5) -- (9, 5) -- (9, 4.5) -- (8.5, 4) -- (8.5, 4.5) -- (8, 4.5) -- (8, 4) -- (7.5, 4) -- (7.5, 4.5) -- (7, 4.5)
    };


    \fill[opacity = 0.6, orange] \newfilled;

        \begin{scope}[shift = {(4, 2)}]
        \triangleobenlinks{3}{3.5-0.5}
        \triangleobenlinks{3}{3-0.5}

        \triangleuntenrechts{2.5}{3-0.5}
        \triangleuntenrechts{3}{3.5-0.5}

        \triangleobenlinks{4}{2.5-0.5}
        \triangleuntenrechts{4}{2}

        \triangleobenlinks{1.5}{2-0.5}
        \triangleobenlinks{1}{2-0.5}
        \triangleobenlinks{0.5}{2-0.5}
        \triangleobenlinks{0.5}{2.5-0.5}
        \triangleobenlinks{0.5}{3-0.5}
    
        \triangleuntenrechts{1.5}{2-0.5}
        \triangleuntenrechts{1}{2-0.5}
        \triangleuntenrechts{0.5}{2-0.5}
        \triangleuntenrechts{0.5}{2.5-0.5}
        \triangleuntenrechts{0.5}{3-0.5}
        \triangleuntenrechts{0}{3-0.5}

    \end{scope}
    
  

    \draw[thick, blue] (4, 4.5) -- (4.5, 5) -- (5, 5) -- (5, 4) -- (6, 4) -- (6, 3.5);
    \draw[thick, blue] (6.5, 4.5) -- (7, 5) -- (7, 5.2);
    \draw[thick, blue] (8, 4.5) -- (8.5, 4.5);

    \draw[thick, magenta] (6, 3.5) -- (6.5, 4.5);
    \draw[thick, magenta] (7, 5.2) -- (8, 4.5);

    \begin{scope}[shift={(12, 2)}]
            \draw (0, 0) rectangle (4.5, 4.5);
    \foreach \n in {-2, ..., 46}
        {\foreach \m in {0, ..., 44}
            {
            \draw[line width = 0.002 mm, nearly transparent] (0.1*\n, 0.1*\m) -- (0.1*\n+0.1, 0.1*\m) -- (0.1*\n, 0.1*\m+0.1) -- (0.1*\n, 0.1*\m); 
            }    
        }
    \draw[line width = 0.002 mm, nearly transparent] (-0.2, 4.5) -- (0, 4.5) (4.5, 4.5) -- (4.7, 4.5) -- (4.7, 0);
        
    \draw[thick] (0, 0) rectangle (4.5, 4.5);
    \draw[blue] (0, 3) -- (2, 3) (3, 2) -- (4.5, 2);
    \draw[magenta] (2, 3) -- (3, 2);
    \def\blueoutline {(-0.1, 2.9) -- (-0.1, 3.1) -- (2.1, 3.1) -- (2.1, 2.9) -- (-0.1, 2.9) (2.9, 2.1) -- (2.9, 1.9) -- (4.6, 1.9) -- (4.6, 2.1) -- (2.9, 2.1) };
    \fill[orange, opacity = 0.6] (2.1, 3.1) -- (2.1, 2.9) -- (1.9, 2.9) -- (2.9, 1.9) -- (2.9, 2.1) -- (3.1, 2.1) -- (2.1, 3.1);
    \fill[blue, opacity = 0.4] \blueoutline;
       \draw[line width=0.2cm, line cap=round, line join=round, draw=green, opacity = 0.15] \blueoutline;
    \end{scope}

    
\end{tikzpicture}
\caption{For both figures: blue curve $\Psi^n_{i,{\rm pre}}$, magenta curve $\Phi^n_i$, blue triangles $\triangle^n_i \cap \Gamma_n(t)$, orange triangles $\triangle^n_i   \setminus \EEE \Gamma_n(t)$,  and \EEE green zones $U_{l_n}(\Gamma_n(t))$. First figure: continuation of Figure \ref{figur1},  where \EEE $U_{l_n}(\Gamma_n(t))$ completely covers $\triangle^n_i$,  i.e., $\triangle^n_i = \triangle^{n, {\rm cur}}_i$. \EEE Second figure: there are orange triangles that are not covered by $U_{l_n}(\Gamma_n(t))$,  i.e., $\triangle^{n, {\rm new}}_i \neq \emptyset$. 
}
\label{figure:U_ln(Gamma)}
\end{figure}



\color{black}
We begin with a lemma that controls the passage to larger neighborhoods. This \EEE will allow us to control $U_{\varepsilon_n+2\theta\varepsilon_n}(\Gamma_n(t))$. In the sequel, by $B_r(x)$  we  denote the \emph{closed} ball with center $x$ and radius $r$.  

        


\begin{lemma}[Increasing the neighborhood]
    \label{lem:incneighbourhoodlemma}
    Given some set $V \subset \R^d$ and $r>0$, we suppose that the $r$-neighborhood $U_r(V)$ of $V$ satisfies $\L^d(U_r(V))<\infty$. Then, we find a constant $C>0$ only depending on the dimension $d$ such that for arbitrary $r'>0$  the $r+r'$-neighborhood satisfies
    \begin{align}
        \label{eq:biggerneighbourhood}
        \L^d(U_{r+r'}(V))\leq C\L^d(U_r(V))\left(\frac{ r+ r'}{ r}\right)^d. 
    \end{align}
    We also have 
    \begin{align}
        \label{eq:biggerneighbourhood2}
        \L^d\big(U_{r+r'}(V)\setminus U_r(V)\big)\leq C\L^d(U_r(V))\frac{( r+ r')^d- r^d}{ r^d}. 
    \end{align}
\end{lemma}
\begin{proof}   
     Define $\mathcal{B}\EEE:= \{B_ r(x)\colon \, x\in V\}$. It holds that $\bigcup_{B_ r(x)\in \mathcal{B}\EEE}B_ r(x)=U_r(V).$
    With the Besicovitch covering theorem, we can find $D = D(d) \in \N$ countable and disjoint subsets $\mathcal{C}_i\subset \mathcal{B} \EEE $ such that $\bigcup_{i=1}^D\bigcup_{B_ r(x)\in \mathcal{C}_i}B_ r(x)=U_r(V).$ We also define the sets of centers $M_i:= \{x\colon \, B_r(x)\in \mathcal{C}_i\}$, as well as $S:=  (\frac{r+r'}{r} )^d$ and $S^\text{add}:= \frac{( r+ r')^d- r^d}{ r^d}$.
    
 We fix $i = 1,\ldots, D$.  First, we note that $\# \mathcal{C}_i\leq \frac{\L^d(U_r(V))}{\L^d(B_ r(0))}$.  By scaling we have the relation  $\L^d(B_ r(x))= S^{-1} \EEE \L^d(B_{r+r'}(x))$  for all $x \in \R^d$. Since the subcollection $\mathcal{C}_i$ is disjoint, we can estimate
    \begin{align*}
 \L^d(B'_{i}) =   \L^d\left({\bigcup}_{x\in M_i}B_{ r+ r'}(x)\right)
    \leq \sum_{x\in M_i}\L^d(B_{ r+ r'}(x))
    =S\L^d\left({\bigcup}_{x\in M_i}B_r(x)\right) =  S\L^d(B_{i}),
    \end{align*}
where we  define \EEE $B_{i}:= {\bigcup}_{x\in M_i}B_ r(x)$ and $B'_{i}:= \bigcup_{x\in M_i}B_{ r+ r'}(x)$. Letting also  $B^\text{add}_i:= B_i'\setminus B_i$, we can estimate
    \begin{align*}
        \L^d(B^\text{add}_i)\,
        \leq\,&\#\mathcal{C}_i\big(\L^d(B_{ r+ r'}(0))-\L^d(B_ r(0))\big) 
        \leq S^\text{add} \L^d( B_i \EEE) ,
    \end{align*}
where we used that $B_{i}\subset B_{i}'$. Then, summing over $i = 1,\ldots, D$ we derive 
    \begin{align*}
        \L^d\left(\bigcup_{i=1}^D B'_{i}\right)\leq  D   S \L^d(U_r(V)),
    \end{align*}
    as well as 
    \begin{align*}
        \L^d\left(\bigcup_{i=1}^D B'_{i}\setminus U_r(V)\right)\leq\L^d\left(\bigcup_{i=1}^D B'_{i}\setminus B_{i}\right)
        =\L^d\left(\bigcup_{i=1}^D B^\text{add}_i\right)
        \leq  D   S^\text{add} \L^d(U_r(V)).
    \end{align*}
    Now,   to conclude, it suffices to show   that $  \bigcup_{i=1}^D B'_{i}\supset U_{r+r'}(V)$. For this matter, fix an arbitrary point $x\in U_{r+r'}(V)$. By definition, there has to be a point $y\in V$ such that $|y-x|\leq  r+ r'$. We define $l$ to be the line connecting $x$ and $y$ and define $x':= l\cap \partial U_r(V)$ as the crossing point of $l$ and $\partial U_r(V)$.   As ${\rm dist}(x',V) = r$, we note that $|x-x'| \le r'$. \EEE Since $x'\in  {U_r(V)}$, \EEE 
    there exists \EEE $z\in V$ and a collection $\mathcal{C}_i$ such that $B_r(z)\in \mathcal{C}_i$ and $x'\in {B_ r(z)} \EEE $. Then, with the triangle inequality we can deduce that $|x-z|\leq |x-x'|+|x'-z|\leq  r+ r'$, which means that $x\in B_{ r+ r'}(z)$ and therefore $x\in \bigcup_{i=1}^D B'_{i} \EEE$. This concludes  the proof.  
\end{proof}
Recall the neighborhood $U_{n}^{\mathcal{T}}(\cdot) = U_{\eps_n,h_n}^{\mathcal{T}}(\cdot)$ introduced before \eqref{eq:energydefinitionepsilon}. \EEE
\begin{corollary}\label{thecorollary}
There exists a constant $C >0$ \EEE only depending on $g$ and $d$ \EEE such that, for all $t \in [0,1]$ and $n \in \N$, it holds that
    \begin{align}
        \label{biggerneighbourhoodfortheta}
        \L^d\left(U_{\varepsilon_n+2\theta\varepsilon_n}(\bruchdiscepsilonn(t))\setminus \bruchdiscepsilonnhe{t} \right) \leq C\theta\varepsilon_n.  
    \end{align}
    Additionally, for $n$ large enough, we have 
    \begin{align}
        \label{eq:UlGamman}
        U^{\mathcal{T}}_n(U_{l_n +  h_n}(\Gamma_n(t)))\subset U_{\varepsilon_n+2\theta\varepsilon_n}(\Gamma_n(t))
    \end{align}
with $l_n= \theta\varepsilon_n+ h_n$.
\end{corollary}
\begin{proof}
    The fact that the neighborhood $U_{\varepsilon_n+2\theta\varepsilon_n}(\bruchdiscepsilonn(t))$ is bigger than $U^{\mathcal{T}}_n( U_{l_n + h_n \EEE }\EEE (\Gamma_n(t)))$  stems from the fact that, for $n$ large enough, $\theta\varepsilon_n$ is much \EEE bigger than the diameter of every simplex in $\mathcal{T}\hn$,  which   is bounded by $h_n$, see \eqref{eq:bound_triangulation} and \eqref{heps}. \EEE 
    
    Concerning   estimate \eqref{biggerneighbourhoodfortheta}, we use  \eqref{eq:biggerneighbourhood2} with $V=\Gamma_n(t)$, $r=\varepsilon_n$, and $r'=\varepsilon_n+2\theta\varepsilon_n$. The statement then follows from the fact that $U_{n}(\Gamma_n(t))\subset U^{\mathcal{T}}_n(\Gamma_n(t))$  and $U^{\mathcal{T}}_n(\Gamma_n(t))\leq C  \eps_n  $ due to \eqref{eq:energybound}.  
\end{proof}
 
Using the bigger neighborhood will therefore only increase the energy functional by an infinitesimal amount which will vanish for $\theta \to 0$. This motivates the choice  of  $\Psi^{n, \text{cur}}_i$  and $\triangle^{n, \text{cur}}_i$. As by definition we have $\Psi^{n, \text{cur}}_i \subset U_{l_n}(\Gamma_n(t))$   and $\triangle^{n, \text{cur}}_i \subset U_{ h_n}(\Psi^{n, \text{cur}}_i)\subset U_{l_n+h_n}(\Gamma_n(t))$, we thus get $U^{\mathcal{T}}_n(\triangle^{n, \text{cur}}_i\EEE)\subset U_{\varepsilon_n+2\theta\varepsilon_n}(\Gamma_n(t))$ by \eqref{eq:UlGamman}.  \EEE


Now we   deal with $\Phi^{n, \text{new}}_i = \EEE \Phi^n_i\setminus U_{l_n}\EEE (\bruchdiscepsilonn(t))$, where $\Phi^n_i$ still denotes the minimal separating set from \eqref{Psi}. Defining $\bar{\varepsilon}_n:= \theta \varepsilon_n$ for convenience, we proceed with a lemma concerning the covering of $\Phi^n_i$.   Recall the definition of the precrack $\Psi^n_{i, {\rm pre}}= J_{y_n}\cap Q_i$. 
\begin{figure}
\begin{tikzpicture}[scale=0.678, transform shape]
    \draw (0, 0) rectangle (8, 8);
    \draw[thick] (8, 8) -- (10, 10) -- (2, 10) -- (0, 8);
    \draw[thick] (8, 8) -- (10, 10) -- (10, 2) -- (8, 0);
    \draw[thick, dashed] (0, 0) -- (2, 2) -- (10, 2);
    \draw[thick, dashed] (2, 2) -- (2, 10);
        \def\badparts{
    (4.95, 4.95) -- (5.05, 4.95) -- (5.1, 5) -- (5.1, 9) -- (5, 9) -- (4.95, 8.95) -- (4.95, 4.95)
    };
     \shade[line width=0.4cm, line cap=round, line join=round, draw=orange, opacity = 0.2] \badparts;
    \def\goodareapart{
    (0, 4) -- (8, 4) -- (10, 6) -- (2, 6) -- (0, 4) (5, 5) -- (5.1, 5) -- (5.05, 4.95) -- (4.95, 4.95) -- (5, 5)
    };
    \def\badpartone{
    (4.95, 4.95) -- (4.95, 8.95) -- (5, 9) -- (5, 5) -- (4.95, 4.95) 
    };
    \def\badparttwo{
    (5, 5) -- (5.1, 5) -- (5.1, 9) -- (5, 9) -- (5, 5) 
    };
    \def\badpartthree{
    (4.95, 4.95) -- (5.05, 4.95) -- (5.05, 8.95) -- (4.95, 8.95) -- (4.95, 4.95)
    };
    \def\badpartfour{
    (5.1, 5) -- (5.1, 9) -- (5.05, 8.95) -- (5.05, 4.95) -- (5.1, 5)
    };
    \draw[blue] (0, 4) -- (8, 4) -- (10, 6);
    \draw[blue, dashed] (0, 4) -- (2, 6) -- (10, 6);
    \draw[green] (4.95, 4.95) -- (4.95, 8.95) -- (5.05, 8.95) -- (5.05, 4.95)
    (5.05, 8.95) -- (5.1, 9) -- (5.1, 5)
    (5.1, 9) -- (5, 9) -- (4.95, 8.95)
    ;
    \draw[green, dashed] (5, 8.9) -- (5, 5) 
    ;
    \draw[green, dash pattern=on 0.75pt off 0.75pt] (5, 5) -- (4.95, 4.95) -- (5.05, 4.95) -- (5.1, 5) -- (5, 5)
    ;
    \fill[blue, opacity = 0.3] \goodareapart;
    \fill[green, opacity = 0.1] \badpartone;
    \fill[green, opacity = 0.1] \badparttwo;
    \fill[green, opacity = 0.1] \badpartthree;
    \fill[green, opacity = 0.1] \badpartfour;
    \fill[magenta, opacity = 0.75] (4.95, 4.95) -- (5.05, 4.95) -- (5.1, 5) -- (5, 5) -- (4.95, 4.95);
    \node[circle, label = {0: \color{blue} $\Psi^n_i$}] at (1, 4.5){};
    \node[circle, label = {270: \color{magenta} $\Phi^n_i$}] at (5, 5){};
    \node[circle, label = {0: \color{green} $\partial^*E^n_{t_i}\setminus J_{y_n}$}] at (5.5, 7){};
    \node[circle, label = {180: \color{orange} $U^\T_{n}(\partial^*E^n_{t_i}\setminus J_{y_n})$}] at (5, 7.5){};
\end{tikzpicture}
\caption{An example in $d=3$, where \eqref{toprov} does not hold for $\partial^*E^n_{t_i}\setminus J_{y_n}$ in place of $\Phi^n_i$ since $\partial^*E^n_{t_i}\setminus J_{y_n}$ (in green) has a small surface. Yet, the set has big diameter, and therefore the volume of $U^\T_{n}(\partial^*E^n_{t_i}\setminus J_{y_n})$ (in orange) is of order $\eps_n$. Consequently, in order to guarantee the estimate in Corollary \ref{corollary:minimalseparator}, we need to replace $\partial^*E^n_{t_i}\setminus J_{y_n}$ by $\Phi^n_i$ (in red), as defined in \eqref{Psi}.}
\label{figure:badnose}
\end{figure}
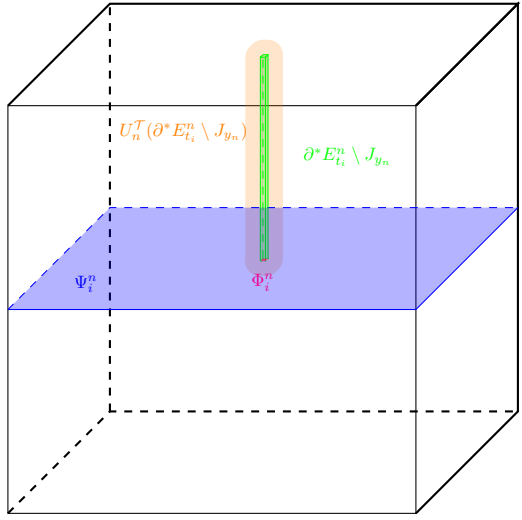
 
\begin{lemma}[Lower density bound for $\Phi^n_i$\EEE]
\label{lem:minimalseparator}
There exists a dimensional \EEE constant $c>0$ such that for  each $x \in \Phi^n_i$ with $B_{\bar{\varepsilon}_n}(x) \subset Q_i$ and   $B_{\bar{\varepsilon}_n}(x) \cap \Psi^n_{i, {\rm pre}} = \emptyset$ it holds that 
\begin{align}
\label{toprov}
\mathcal{H}^{d-1}\big(   B_{\bar{\varepsilon}_n}(x) \cap \Phi^n_i \big) \ge c \bar{\varepsilon}_n^{d-1}.
\end{align}
\end{lemma}

\begin{proof}
Fix $x \in \Phi^n_i$, and for simplicity write $B = B_{\bar{\varepsilon}_n}(x)$ and $B_s = B_{s}(x)$  for $s>0$. Denote the connected component of $Q_i \setminus (\Psi^n_{i, {\rm pre}} \EEE \cup \Phi^n_i)$ which contains $D^+_{n, i}$ by $E$. Note that $\partial E \cap  B \subset \Phi^n_i$ since $B \subset Q_i$ and $\Gamma^n_{i, {\rm pre}} \cap B = \emptyset$. Define $f(s) = \mathcal{H}^{d-2}(  \partial B_s \cap \Phi^n_i)$. By the isoperimetric inequality on the sphere (c.f.\ \cite[Chapter 2]{BuragoZalgaller1988})   we have 
\begin{align}\label{prop1}
 \min\big\{\mathcal{H}^{d-1}(E \cap  \partial B_s ),  \mathcal{H}^{d-1}(  \partial B_s \setminus E )  \big\} \le c_d\big(\mathcal{H}^{d-2}(\partial B_s \cap \partial E)\big)^{\frac{d-1}{d-2}} \le c_d(f(s))^{\frac{d-1}{d-2}}   \quad \text{for $s\le \bar{\eps}_n$} \EEE
 \end{align}
 for some constant $c_d>0$ only depending on the dimension. Let us now assume that  \eqref{toprov} is violated, in the sense that 
 \begin{align}\label{toprov1.5}
\mathcal{H}^{d-1}\big(   B  \cap \Phi^n_i \big) <   \frac{1}{2(c_d^{d-2} (d-1)^{d-1})}  \bar{\varepsilon}_n^{d-1}.
\end{align}
We will show that this leads to a contradiction, which then implies the statement of the lemma.

The proof relies on showing that there exists $\sigma \in (0,\bar{\varepsilon}_n)$
such that $\mathcal{H}^{d-1}(\partial E \cap \partial B_\sigma) = 0 $ and  

\begin{align}\label{toprov2} 
 \int_0^\sigma f(s) \, {\rm d} s >  c_d (f(\sigma))^{\frac{d-1}{d-2}} .
\end{align} 
We defer the verification of \eqref{toprov2} to the end and first proceed with the proof.

Without restriction we assume that $\mathcal{H}^{d-1}(E \cap  \partial B_\sigma ) \le  \mathcal{H}^{d-1}(  \partial B_\sigma \setminus E ) $ as otherwise we simply replace $E$ by the connected component of $Q_i \setminus ( \Psi^n_{i, {\rm pre}} \EEE \cup \Phi^n_i)$ which contains $D^-_{n, i}$. Then, \eqref{prop1} gives  
 $$ \mathcal{H}^{d-1}(E \cap  \partial B_\sigma ) \le  c_d(f(\sigma))^{\frac{d-1}{d-2}}.   $$
 We replace $\Phi^n_i$ by the set $\Tilde{\Phi}^n_i := \Phi^n_i \cap (Q_i \setminus {B_\sigma}) \EEE \cup ( \overline{E} \EEE \cap  \partial B_\sigma ).$ By Fubini's theorem, \eqref{toprov2}, and $\mathcal{H}^{d-1}(\partial E \cap \partial B_\sigma) = 0 $ \EEE  we estimate
\begin{align*}
\mathcal{H}^{d-1}(\Tilde{\Phi}^n_i) & = \mathcal{H}^{d-1}({\Phi^n_i}) - \mathcal{H}^{d-1}({\Phi^n_i} \cap B_\sigma) + \mathcal{H}^{d-1} (E \cap  \partial B_\sigma )  \\&\le \mathcal{H}^{d-1}({\Phi^n_i}) - \int_0^\sigma f(s) \, {\rm d}s + c_d(f(\sigma))^{\frac{d-1}{d-2}}  <  \mathcal{H}^{d-1}({\Phi^n_i}).
\end{align*}
This will give a contradiction to the choice of $\Phi^n_i$ in \eqref{Psi} once we have proven that also the set $\Psi^n_{i, {\rm pre}} \cup \Tilde{\Phi}^n_i$ is separating. 

Assuming by contradiction that $\Psi^n_{i, {\rm pre}} \cup \Tilde{\Phi}^n_i$ is not separating, we find $p^\pm \in D^\pm$ and a curve $\gamma \colon [0,1]  \to Q_i \setminus (\Psi^n_{i, {\rm pre}} \cup \Tilde{\Phi}^n_i)$ with $\gamma(0) = p^+$ and $\gamma(1) = p^-$.  Since $\Phi^n_i\setminus \Tilde{\Phi}^n_i \subset B_\sigma$ and $\Psi^n_{i, {\rm pre}} \cup {\Phi^n_i}$ separates, $\gamma$ necessarily intersects $\Phi^n_i \cap B_\sigma$ and we have $\gamma([0,1]) \cap \Phi^n_i \subset B_\sigma$. Thus, recalling $\gamma(0) \in E$, denoting by $\tau$ the first time such that $\gamma(\tau) \in \partial B_\sigma$, we get  $\gamma(\tau) \in \overline{E} \EEE $. However, this shows $\gamma(\tau) \in \overline{E} \EEE \cap \partial B_\sigma \subset \Tilde{\Phi}^n_i$ which contradicts the fact that the curve $\gamma$ does not meet $\Psi^n_{i, {\rm pre}} \cup \Tilde{\Phi}^n_i$.

It remains to show \eqref{toprov2}. This will again be achieved by contradiction. If the statement was wrong, we would have $ \int_0^t g(s)^{\frac{d-2}{d-1}} \, {\rm d} s \le   c_d g(t)$ for a.e.\ \EEE $t \in (0,\bar{\varepsilon}_n)$,
, where we set $g(t) := f(t)^{\frac{d-1}{d-2}}$.  Here we use that \EEE $\H^{d-1}(\partial E\cap \partial B_t)=0$ holds for a.e.\ $t\in (0,\bar{\varepsilon}_n)$.


It is elementary to check that  the solution of the integral equation 
$$\text{$ \int_0^t h(s)^{\frac{d-2}{d-1}} \, {\rm d} s = c_d h(t)$ for all $t \in (0,\bar{\varepsilon}_n)$}$$
 is given by $h(t) = \frac{1}{(c_d(d-1))^{d-1}} t^{d-1}$. Then, the nonlinear Gr\"onwall inequality implies $g(t) \ge h(t)$ for all $t \in (0,\bar{\varepsilon}_n)$ and thus $f(t) \ge \frac{1}{(c_d(d-1))^{d-2}} t^{d-2}$. Again using Fubini's theorem this yields
 \begin{align*}
  \mathcal{H}^{d-1}(   B \cap \Phi^n_i) = \int_0^{\bar{\varepsilon}_n}  f(s) \, {\rm d}s \ge \int_0^{\bar{\varepsilon}_n} \frac{1}{(c_d(d-1))^{d-2}} s^{d-2} \, {\rm d}s 
 =   \frac{1}{(c_d^{d-2} (d-1)^{d-1})}    \bar{\varepsilon}_n^{d-1}.
  \end{align*}
This contradicts \eqref{toprov1.5} and concludes the proof. 
\end{proof}

\begin{corollary}
    \label{corollary:minimalseparator}
For  $\Phi^{n, \text{\rm new}}_i= \Phi^n_i\setminus U_{l_n}  \EEE (\bruchdiscepsilonn(t))$ it holds that
       \begin{align}
        \label{eq:corollaryminimalseparator}
        \L^d\left({\bigcup}_i U_{2\varepsilon_n}(\Phi^{n, \text{\rm new}}_i)\right)\leq C \theta \EEE \varepsilon_n.
    \end{align}
\end{corollary}

\begin{proof}
Define $\hat{\Phi}^{n, \text{new}}_i = \Phi^{n, \text{new}}_i \setminus U_{\bar{\eps}_n}(\partial Q_i \cap R_{n,i})$. \EEE
    The first step is to cover $\bigcup_i U_{\theta\varepsilon_n}( \hat{\Phi}^{n, \text{new}}_i\EEE )$ \EEE with balls having centers $x\in \bigcup_i\hat{\Phi}^{n, \text{new}}_i$ and radius $\theta\varepsilon_n$. Then Besicovitch's covering theorem allows us to find $D \in \N$  disjoint subcollections $\mathcal{C}_j$ such that 
    \begin{align*}
        \bigcup_{j=1}^D \bigcup_{B_{\theta\varepsilon_n}(x)\in \mathcal{C}_j}B_{\theta\varepsilon_n}(x)\supset {\bigcup}_i U_{\theta\varepsilon_n}(\hat{\Phi}^{n, \text{new}}_i).       
    \end{align*}
Estimate   \eqref{psiestimate} together with \eqref{toprov}   shows that the maximum amount of balls in each $\mathcal{C}_j$ can be at most $\frac{C\theta^2}{\varepsilon_n^{d-1}}$. This means that $\L^d\left(\bigcup_i U_{\theta\varepsilon_n}(\hat{\Phi}^{n, {\rm new}}_i)\right)\leq C\theta^{d+2}\varepsilon_n$. 
    With \eqref{eq:biggerneighbourhood} \EEE applied to $V=\bigcup_i\hat{\Phi}^{n, \text{new}}_i$, $r = \theta\varepsilon_n$, \EEE and $r'=2\varepsilon_n$, we then find \EEE $\L^d(U_{2\varepsilon_n}(\bigcup_i\hat{\Phi}^{n, \text{new}}_i))\leq C\theta^2\varepsilon_n$.  Using \eqref{eq:propertiesforcubesforu} and \eqref{eq:estimatesnegativerectangle} it is elementary to check that $\L^d( U_{2\varepsilon_n + \bar{\eps}_n}( {\bigcup}_i\partial Q_i \cap R_{n,i})) \le C\theta\varepsilon_n$. This concludes the proof. \EEE 
    \EEE
\end{proof}

\paragraph{\textbf{Proof of Theorem \ref{theorem:stabilityresulteps}.}} \EEE We are now in the position to prove   Theorem \ref{theorem:stabilityresulteps}.   Recall \EEE the definition of $\phi_n$ and $\gamma_n^\phi$ below \eqref{eq.bbbq2} and \EEE in \eqref{finallygamma}, as well as the definitions and properties  of the simplices  $\triangle^n= \triangle^{n, \text{bad}}\cup \triangle^{n, \text{lat}}\cup \triangle^{n, \text{good}}$ . \EEE
\begin{proof}[Proof of Theorem \ref{theorem:stabilityresulteps}]
  As observed at the beginning of the section, it suffices to check \eqref{only thing to check}. We start by estimating the surface energy. 
In view of \eqref{finallygamma}, \EEE we obtain
    \begin{align}\label{ainelettegleichung}
         &\frac{1}{2\varepsilon_n}\L^d\left(U^\T_n(\{\gamma^\phi_n\neq 0\})\setminus U^\T_n(\Gamma_n(t))\right)  =\frac{1}{2\varepsilon_n}\L^d\left(U^{\mathcal{T}}_n(\triangle^n)\setminus U^\T_n(\Gamma_n(t))\right) \notag \\
    \leq\,& \frac{1}{2\varepsilon_n} \left(\L^d\left(U^{\mathcal{T}}_n(\triangle^{n, \text{bad}})\setminus U^\T_n(\Gamma_n(t))\right)+\L^d\left(U^{\mathcal{T}}_n(\triangle^{n, \text{lat}})\right)+\L^d(U^{\mathcal{T}}_n(\triangle^{n, \text{good}}))\right).
    \end{align}
    As observed preceding Lemma \ref{lem:minimalseparator}, we have $\triangle^{n, \text{cur}}  \subset U_{l_n+ h_n}(\Gamma_n(t))$. \EEE For this part, \EEE we can apply Corollary~\ref{thecorollary}. Additionally, since $\bigcup_i U_{2\varepsilon_n}(\Phi^{n, \text{new}}_i)\supset U^{\mathcal{T}}_n \EEE (\triangle^{n, \text{new}})$ for $n$ large enough, \EEE we can apply \eqref{eq:corollaryminimalseparator} for $\triangle^{n, \text{new}}$. Since by definition we have the identity $\triangle^{n, \text{bad}}=\triangle^{n, \text{cur}}\cup \triangle^{n, \text{new}}$, this suffices to estimate $\L^d(U^\T_n(\triangle^{n, \text{bad}})\setminus U^\T_n(\Gamma_n(t)))$.
\color{black}
Since $U^{\mathcal{T}}_n(\triangle^{n, \text{good}})\subset U_{\varepsilon_n+ h_n \EEE }(J_{\phi}  \setminus  B_{\rm bad} \EEE ) \EEE $  (recall $\sup\{\text{diam } T: T\in \T_{h_n}\}\leq h_n$ by  \eqref{eq:bound_triangulation} and use \EEE \eqref{eq:trianglegoodcapQi}),  \EEE  we can deduce with \eqref{biggerneighbourhoodfortheta}--\eqref{eq:UlGamman}, \eqref{eq:corollaryminimalseparator}, \eqref{eq:estimatelateral}, and \eqref{eq:estimategoodconvergence} that  
    \begin{align*}
        &\limsup_{n\to\infty}\frac{1}{2\varepsilon_n}\left(\L^d\left(U^{\mathcal{T}}_n(\triangle^{n, \text{bad}})\setminus U^{\mathcal{T}}_n( \Gamma_n(t))\right)+\L^d\left( U^{\mathcal{T}}_n \EEE (\triangle^{n. \text{lat}})\right)+\L^d(U^{\mathcal{T}}_n(\triangle^{n, \text{good}})\right) \\
        \leq\,&\limsup_{n\to\infty}\frac{1}{2\varepsilon_n}\bigg(\L^d(U_{\varepsilon_n+2\theta\varepsilon_n}(\bruchdiscepsilonn(t))\setminus U^{\mathcal{T}}_n(\Gamma_n(t)\EEE))+\L^d\left({\bigcup}_i U_{2\varepsilon_n}(\Phi^{n, {\rm new}}_i)\right) \\
        &+\L^d\left({\bigcup}_i U_{2\varepsilon_n}(  \partial  Q_i \cap R_{n, i} \EEE ) \EEE \right)+\L^d\left(U_{\varepsilon_n+ h_n}\left(J_{\phi}\setminus   B_{\rm bad} \EEE\right) \right)\bigg) \\
        \leq\,& \limsup_{n\to\infty}\frac{1}{2\varepsilon_n}\left(C\theta \varepsilon_n+C  \theta \EEE \varepsilon_n+C\theta \varepsilon_n+\L^d\left(U_{\varepsilon_n+ h_n}\left(J_{\phi}\setminus   B_{\rm bad} \EEE\right) \EEE  \right)\right)         \le   \mathcal{H}^{d-1}\left(J_\phi\setminus   B_{\rm bad} \EEE \right)  + C\theta. \EEE
    \end{align*}
  Now, as $\mathcal{H}^{d-1}( (J_u \cap J_\phi) \EEE \setminus B_\text{bad} \EEE  )\leq C\theta$  by \EEE \eqref{eq:propertiesbesicovitchcovering} and \eqref{eq:goodcubesestimate}, \EEE we can   conclude  by \eqref{ainelettegleichung} \EEE
    \begin{align*}
    \limsup_{n \to \infty} \frac{1}{2\varepsilon_n}\L^d\left(U^\T_n(\{\gamma^\phi_n\neq 0\})\setminus U^\T_n(\Gamma_n(t))\right)   & \le   \mathcal{H}^{d-1}\left(J_\phi\setminus  B_{\rm bad} \EEE \right) + C\theta \EEE \le  \mathcal{H}^{d-1}(J_\phi\setminus J_u)+C\theta.
    \end{align*}
This shows the second part of \eqref{only thing to  check}, where we recall that $u=u(t)$. \EEE

Concerning the elastic part of the energy, we first observe that the estimates above lead to $\L^d(\triangle^n)\le C \varepsilon_n$.   We have $\nabla \EEE \phi_n-\gamma^\phi_n =0$ on $\triangle^n$, see \eqref{finallygamma},  $\phi'_n = \phi'_{n,i}$ in each bad cube  $Q_i\in \mathcal{Q}_{\text{bad}}$, and $\phi'_n = \phi$ everywhere else. \EEE   As outside of the triangles $\triangle^n$ the function  $\phi'_n$ has no jump,  see \eqref{finallygamma0}, and $\phi_n$ is the linear interpolation of $\phi_n'$ \EEE with respect to $\T_{h_n}(\Omega')$, we know that  $\|\nabla \phi_n\|_{L^\infty(\Omega' \setminus \triangle^n)}$ \EEE can be estimated by $C \| \nabla \phi_n' \EEE \|_{L^\infty(\Omega')}$.  In view of \eqref{eq.bbb}, $\|\nabla \phi_n\|_{L^\infty(\Omega' \setminus \triangle^n)}$ can thus be controlled by  $C \|   \nabla \phi  \|_{L^\infty(\Omega')}$. \EEE  Since we have $\phi\in W^{2, \infty}( \Omega' \setminus J_\phi \EEE  )$, it can be further shown that $\|\nabla \phi_n-\nabla \phi\|_{L^\infty(\Omega'\setminus (B_\text{bad} \cup \triangle^n) \EEE  )}\leq C h_n$. Thus, we get  
\begin{align*}
\limsup_{n\to\infty}  \int_\Omega |\nabla \phi_n-\gamma_n^\phi - \nabla \phi|^2\,{\rm d}x & = \limsup_{n\to\infty} \Big( \int\limits_{\triangle^n}|\nabla \phi|^2\,\derx +  \hspace{-0.3cm}\int\limits_{\Omega\setminus (B_\text{bad} \cup \triangle^n ) \EEE }  \hspace{-0.8cm}|\nabla \phi_n-\nabla \phi|^2\,\derx + \int\limits_{ B_\text{bad} \EEE  \setminus \triangle^n} \hspace{-0.4cm}|\nabla \phi_n-\nabla \phi|^2\,\derx  \Big)\\ &
\leq  C\varepsilon_n+ C h_n^2 \EEE + C \theta \EEE \leq C \theta, \EEE
\end{align*} 
where for the second step we used \EEE  $\L^d( B_\text{bad} \EEE ) \le  C\theta \EEE $ (see   \eqref{eq:propertiesforcubes}), the fact that $(\nabla \phi_n)_n$  are bounded in $L^{\infty}(\Omega' \setminus \triangle^n;\R^d)$, $\nabla \phi \in  L^{\infty}(\Omega';\R^d)$, \EEE  as well as  $\L^d(\triangle^n)\le C\eps_n$. \EEE    
This shows the first part of \eqref{only thing to check} and concludes the proof. 
\end{proof}

\section{Proof of further results}
\label{section:AlternativeSettings}

In this section we explain the adaptations which are necessary to obtain the results announced in Subsection \ref{sec:further}.

\begin{proof}[Sketch of the proof of Theorem \ref{th: simlutaneous}]
 With the notation introduced in Subsection \ref{sec:further}, the bound in \eqref{eq:ineqtimeintn} reads as  \EEE 
\begin{equation}\label{energyestima}
    \begin{aligned}
        \mathcal{E}_n( \hat{u}_n  (t), \hat{\gamma}_n(t), \hat{\Gamma}_n(t) \EEE )\leq\mathcal{E}_n(\hat{u}_n(0), \hat{\gamma}_n(0), \emptyset) +2\int_0^t\int_\Omega (\nabla \hat{u}_n(s)-\hat{\gamma}_n(s)) \cdot \partial_t\nabla g_{n} \EEE (s)\,\derx\,\ders+e(n),
    \end{aligned}    
\end{equation}
with $e(n) \to 0$ as $n \to \infty$.   We note that  the calculation   is \EEE  done with the integral from $0$ to $ t^n_i \EEE $ with $ t_i^n \EEE \leq t< t^n_{i+1} \EEE $ for $ t^n_i, t^n_{i+1} \EEE \in I_n$ instead of $0$ to $t$, but due to the regularity of  $g_{n}$ \EEE and the bounds on $\hat{u}_n$ and $\hat{\gamma}_n$, we can absorb this error in the $e(n)$-term. \EEE  With the same argument as in \eqref{eq:energybound} one can also show
\begin{align*}
    \mathcal{E}_n(\hat{u}_n(t), \hat{\gamma}_n(t), \hat{\Gamma}_n(t))+\|\hat{u}_n(t)\|_{L^\infty(\Omega')}\leq C_2
\end{align*}
with $C_2$ being independent of $t$ and $n$. With these estimates, we can  follow closely the lines of the proof in Section \ref{chapter:Limitepsilontozero}.  The Lemmas \ref{lem:truncationIinfinity}--\ref{lem:alternativetruncation}, \ref{lemma:firstsummary}   remain unchanged while in \EEE the proof of the stability result Theorem~\ref{theorem:stabilityresulteps} we employ \eqref{eq:minepsi+1} in place of \eqref{eq:3.3}. (The construction in Section \ref{section:stabilityresult} yields a sequence $(\phi_n)_n$ with boundary conditions $g(t)$ and not with $g(s_n)$, where $s_n \in I_n$ is the largest time with $s_n \le t$. Therefore, strictly speaking,  we need to replace the sequence $(\phi_n)_n $ by $\phi_n -g(t) + g(s_n)$, which causes no problems due to the regularity of $g$.) Also Theorem \ref{theorem:continuationtheorem} and Lemma \ref{lem:convergenceoutsideIinfinity} can be derived in the exactly same way. Finally, for the analog of the proof of Theorem \ref{theorem:griffithmin},   the only difference lies in the derivation of the energy balance and the energy convergence since instead of an energy balance along the sequence (cf.\ \eqref{eq:3.4}) we only have the energy estimate \eqref{energyestima}. For this reason, in place of the equality in  \eqref{lowi}, we only get the inequality 
\color{black}
 \begin{align*} 
        \lim_{n\to\infty}  \mathcal{E}_{n}(\hat{u}_n(t), \hat{\gamma}_n(t), \hat{\Gamma}_n(t))   \le  
        \,& \,\mathcal{E}(0)+2\int_0^t\int_\Omega \nabla u(s)\cdot \partial_t\nabla  g(s)\,\derx\,\ders \quad \text{ for all $t \in [0,1]$,} 
    \end{align*}
where we employed (the analogous versions of) \eqref{eq:energyconvergencetimezero} and \eqref{analog2}.

By lower semicontinuity (see  Lemma~\ref{lemma:firstsummary} and Lemma~\ref{lem:convergenceoutsideIinfinity}) \EEE this implies 
$${ \mathcal{E}(t)   \le       \liminf_{n\to\infty}  \mathcal{E}_{n}(\hat{u}_n(t), \hat{\gamma}_n(t), \hat{\Gamma}_n(t)) \le \EEE
        \mathcal{E}(0)+2\int_0^t\int_\Omega \nabla u(s)\cdot \partial_t\nabla  g(s)\,\derx\,\ders  }
 $$
for all $t \in [0, 1]$, \EEE i.e., the same estimate as in \eqref{lsc}.      From this point on, we can follow the proof of Theorem~\ref{theorem:griffithmin} to conclude.  \EEE
\end{proof} 

\begin{proof}[Sketch of the proof of Theorem 
\ref{theorem:griffithminNondiscrete}]

Concerning the proof of Theorem \ref{theorem:griffithminNondiscrete},  we can almost verbatim repeat \EEE the proof of Theorem \ref{th: simlutaneous}   as \EEE both cases are conceptually the same, just in different function spaces.   In fact, the proof for the stability result (see Theorem  \ref{theorem:stabilityresulteps} for the analog) is even easier as we do not have to deal with the neighborhoods
$U_{\eps_m,h}^{\mathcal{T}}(\cdot)$ but only with neighborhoods $U_{\eps_m}(\cdot)$, i.e.,  the additional simplices added due to the triangulation are not needed.   
\end{proof}
\EEE
\section*{Acknowledgements} 
This research was funded by the Deutsche Forschungsgemeinschaft (DFG, German Research Foundation) - 377472739/GRK 2423/2-2023. The authors are very grateful for this support.
 

\appendix
 
\section{Remaining proofs}\label{sec:appendix}
 
In this appendix, we present the remaining proofs that have been omitted  in the paper. \EEE
\begin{theorem}[$SBV$-compactness]
    \label{theorem:sbvcompactness}
     Let $(u_k)_{k\in \N}$ be a sequence in $SBV^2(\Omega')$ such that there exists $c> 0$ with 
     \begin{align*}
         \int_{\Omega'} |\nabla u_k|^2 \,\derx+\mathcal{H}^{d-1}(J_{u_k})+\|u_k\|_{L^\infty(\Omega')}\leq c
     \end{align*}
     for every $k\in \N$. Then, there exists a subsequence $(u_{k_l})_l$ and a function $u\in SBV^2(\Omega')$ such that, as $l \to \infty$, \EEE 
     \begin{equation}
         \begin{aligned}
             \label{eq:sbvcompactnessresult}
             u_{k_l}&\to u \emph{ strongly in }L^1(\Omega'), \\
             \nabla u_{k_l}&\rightharpoonup \nabla u \emph{ weakly in }L^2(\Omega; \R^d), \\
             \mathcal{H}^{d-1}(J_u)&\leq \liminf_{l\to\infty} \mathcal{H}^{d-1}(J_{u_{k_l}}).
         \end{aligned}
     \end{equation}
\end{theorem}

The convergence in \eqref{eq:sbvcompactnessresult} is called  \emph{$SBV^2$-convergence}. \EEE We begin with the proof of Lemma \ref{lem:bothinequalitesdisc}, and postpone the proof of Lemma \ref{lem:PointwiseConvergence} to the end of the section. \EEE

\begin{proof}[Proof of Lemma \ref{lem:bothinequalitesdisc}]
The first inequality has already been addressed in  \eqref{firstinequ}, so we focus on the other inequality.    For fixed $t\in(0, 1]$, given $\eta >0$, we apply \cite[Lemma 4.12, Remark 4.13]{dMasoFrancfortToader2005} to find  a finite partition   $(s_i)^k_{i=0}$  of $[0,t]$ satisfying  $|s_{i+1} - s_{i}|\le\eta$ for all $i=0,\ldots,k-1$  such that  
\begin{align}\label{Riemann}
{\rm (i)} & \ \ \sum_{i=0}^{k-1} \int_{s_{i}}^{s_{i+1}}  \Vert \big(\nabla u_\eps(s_{i+1}) - \gamma_\eps(s_{i+1}) \big) \cdot \partial_t \nabla  g_h \EEE (s_{i+1})   -  \big(\nabla u_\eps(\tau) - \gamma_\eps(\tau) \big) \cdot \partial_t \nabla  g_h \EEE (\tau) \Vert_{L^2(\Omega)}  \, {\rm d}\tau  \le \eta,\notag \\
{\rm (ii)} & \ \  \sum_{i=0}^{k-1} \int_{s_{i}}^{s_{i+1}}  \Vert   \partial_t \nabla  g_h \EEE (s_{i+1})  -   \partial_t \nabla  g_h \EEE (\tau)\Vert_{L^2(\Omega)}  \, {\rm d}\tau  \le \eta  . 
\end{align}
Using Lemma \ref{lem:stabdisc} with the competitor $(u_\varepsilon(s_{i+1})- g_h \EEE (s_{i+1})+  g_h \EEE (s_{i}), \gamma_\varepsilon(s_{i+1}))$ for $i \in \lbrace 0,\ldots, k-1\rbrace$, we obtain
    \begin{align*}
        &\E_\varepsilon(u_\varepsilon(s_{i}), \gamma_\varepsilon(s_{i}), \Gamma_\varepsilon(s_{i}))   \leq \E_\varepsilon\Big(u_\varepsilon(s_{i+1})-g_h(s_{i+1})+g_h(s_{i}), \gamma_\varepsilon(s_{i+1}), \Gamma_\varepsilon(s_{i})\cup \{\gamma_\varepsilon(s_{i+1})\neq 0\}\Big).
    \end{align*}
 Using 
    $\Gamma_\varepsilon(s_{i})\cup \{\gamma_\varepsilon(s_{i+1})\neq 0\}\subset \Gamma_\varepsilon(s_{i+1})$, expanding the elastic energy, rearranging the terms, and employing Minkowski's integral inequality  we find 
    \begin{align*}
        &\E_\varepsilon(u_\varepsilon(s_{i+1}), \gamma_\varepsilon(s_{i+1}), \Gamma_\varepsilon(s_{i+1}))-\E_\varepsilon(u_\varepsilon(s_{i}), \gamma_\varepsilon(s_{i}), \Gamma_\varepsilon(s_{i})) \\ \geq\, &2\int_\Omega(\nabla u_\varepsilon(s_{i+1})-\gamma_\varepsilon(s_{i+1})) \cdot (\nabla  g_h \EEE(s_{i+1})-\nabla  g_h \EEE(s_{i}))\,\derx -
        \int_\Omega|\nabla  g_h \EEE(s_{i+1})-\nabla  g_h \EEE(s_{i})|^2\,\derx \\
        \geq 
        \,&2\int_\Omega(\nabla u_\varepsilon(s_{i+1})-\gamma_\varepsilon( s_{i})) \EEE \cdot (\nabla  g_h \EEE(s_{i+1})-\nabla  g_h \EEE(s_{i+1}))\,\derx-
        \sigma (\eta)\int_{s_{i}}^{s_{i+1}}\|\partial_t\nabla   g_h \EEE(s)\|_{L^2(\Omega)}\,\der{s}
    \end{align*}
    with $\sigma(r):= \sup_{t_2-t_1=r}\int_{t_1}^{t_2}\|\partial_t\nabla  g_h \EEE(s)\|_{L^2(\Omega)}\,\der{s}$, where in the last step we used  $|s_{i+1} - s_{i}|\le\eta$. \EEE Summing over $i \in \lbrace 0,\ldots, k-1\rbrace$ we deduce   
        \begin{align}\label{pluugy}
            \E_\varepsilon(u_\varepsilon(t), \gamma_\varepsilon(t), \Gamma_\varepsilon(t)) 
          &  \geq\, \E_\varepsilon(u_\varepsilon(0), \gamma_\varepsilon(0), \Gamma_\varepsilon(0))+
            \Pi  - \EEE \sigma\left(\eta\right)\int_0^t\|\nabla \partial_t g_h(s)\|_{L^2(\Omega)}\,\ders.
        \end{align}
        where for shorthand we define
        $$ \Pi := \sum_{i=0 }^{k-1}2\int_\Omega(\nabla u_\varepsilon(s_{i+1})-\gamma_\varepsilon(s_{i+1})) \cdot \int_{s_{i}}^{s_{i+1}}\nabla \partial_t g_h(s)\,\der{s}\,\derx.$$
 For the integral $\Pi$, using  first   \eqref{Riemann}(ii) along with the fact that $\Vert \nabla u_\varepsilon(s_{i+1})-\gamma_\varepsilon(s_{i+1}) \Vert_{L^2(\Omega)} \le C$ for some $C$ depending only  on $g$, $h$, and $\eps$ (by passing to the limit in \eqref{apriori bound} and using \eqref{eq:identityyni}),   and then  \eqref{Riemann}(i) we get 
\begin{align*}
\Pi & \ge  \sum_{i=0 }^{k-1}  2 \EEE \int_\Omega  \int_{s_{i}}^{s_{i+1}} (\nabla u_\varepsilon(s_{i+1})-\gamma_\varepsilon(s_{i+1})) \cdot \nabla \partial_t g_h(s_{i+1})\,\der{s}\,\derx - C\eta \\ 
& \ge   2 \EEE \int_0^t\int_\Omega(\nabla u_\varepsilon(s)-\gamma_\varepsilon(s)) \cdot \nabla\partial_t  g_h \EEE(s)\,\derx\,\ders  - \eta  - C\eta.
\end{align*} 
Plugging this into  \eqref{pluugy} and using that $\sigma\left(\eta\right)\to 0$ for $\eta\to 0 $,    the desired statement follows in the limit $\eta \to 0$. 
\end{proof}

We proceed with the proofs of Theorem \ref{theorem:continuationtheorem} and Lemma \ref{lem:convergenceoutsideIinfinity}.

\begin{proof}[Proof of Theorem \ref{theorem:continuationtheorem}]
First, the fact that $u(t)\in SBV^2(\Omega')$ with $u(t) = g(t) $ on $\Omega' \setminus \overline{\Omega}$  follows from the compactness result in Theorem \ref{theorem:sbvcompactness} along with  the uniform bound \eqref{eq:truncationIinfinitybound} \EEE at each time $s \in I_{\infty,0}$.  Properties \eqref{eq:jumpisingamma} and \eqref{eq:minimalityentiretimeinterval} follow as in the proof of \cite[Lemma 3.8]{FrancfortLarsen2003}, where we particularly employ \cite[Corollary 2.10]{FrancfortLarsen2003} along with  Theorem \ref{theorem:stabilityresulteps}, as well as Theorem \ref{theorem:sbvcompactness}.

 Then, $\nabla u\in L^\infty([0, 1]; L^2(\Omega'; \R^d))$, the fact that $\nabla u$ is left continuous in $[0, 1]\setminus I_{\infty,0}$ with respect to the strong $L^2(\Omega';\R^d)$-topology, and the energy inequality \eqref{eq:energyequalityentiretimeinterval} follow exactly as in the proof of  \cite[Propostion 5.9]{Giacomini2005}, by exploiting  \eqref{eq:minimalityentiretimeinterval}. Indeed, we note that all properties stated so far exclusively rely on the limiting problem, and are completely independent of the $\eps$-problem, whence we can follow the proof in \cite{Giacomini2005}.

Next, \eqref{eq:minimalityzeroentiretimeinterval} is a consequence of the $\Gamma$-convergence result in \cite[Theorem 5.1]{SchmidtFraternaliOrtiz2009} and Theorem \ref{theorem:approx}(i). Finally, \eqref{eq:lambdainequalityentiretimeinterval} follows from Lemma \ref{lemma:firstsummary}. 
\end{proof}

\begin{proof}[Proof of Lemma \ref{lem:convergenceoutsideIinfinity}] 
\label{proof:Lemma4.8}
 As the statement has been proven already for $t\in I_{\infty,0} \EEE $ in Lemma \ref{lem:truncationIinfinity}, we only consider $t\in [0, 1]\setminus  I_{\infty,0}$. Consider  $s\in I_{\infty,0}$ with $s<t$. (Eventually, we will send $s\nearrow t$.) We recall the  definition \EEE $ \Gamma_n \EEE (t):= \bigcup_{ \tau \in I^t_{\infty,\eps_n}} \EEE \{  \gamma_n \EEE (\tau)\neq 0\}$  preceding \EEE Lemma \ref{lem:stabilityjumpset}. \EEE We define the minimum problem
    \begin{align*}
        J:= \inf\left\{\int_{\Omega\setminus \Gamma_n(t)}|\nabla z|^2\,\derx:z\in V\hn(\Omega'), \  z=g_n(s) \text{ on }\Omega' \setminus \overline{\Omega}\right\}
    \end{align*}
    and denote the (unique) minimizer of this problem by $w_n(s, t)$. We observe that $u_n(t)-w_n(s, t)$ is the unique minimizer of  
    \begin{align*}
        L:= \inf\left\{\int_{\Omega\setminus \Gamma_n(t)}|\nabla z|^2\,\derx:z\in V\hn(\Omega'), \  z=g_n(t)-g_n(s) \text{ on }\Omega' \setminus \overline{\Omega}\right\},
    \end{align*}
    due to minimality in the separate problems.   (Use \eqref{eq:3.3} at time $t$ for $\bar{\gamma} =  \gamma_{n}(t)$.) \EEE       As $g_n(t)-g_n(s)$ is admissible for $L$, we have 
    \begin{equation}
    \begin{aligned}
        \label{eq:estimateuminuswcontraghn}
            &\int_{\Omega\setminus \Gamma_n(t)}|\nabla u_n(t)-\nabla w_n(s, t)|^2\,\derx 
            \leq 
            \int_{\Omega\setminus \Gamma_n(t)}|\nabla g_n(t)-\nabla g_n(s)|^2\,\derx.
    \end{aligned}
    \end{equation}
  We find $\int_{\Omega\setminus \Gamma_n(t)} \nabla w_n(s, t) \cdot \nabla (u_n(s)-w_n(s, t)) \, {\rm d}x = 0$ since $u_n(s)-w_n(s, t)$ is an admissible test function for $J$. Thus, \EEE  
    \begin{align}
        \label{eq:equalityuminusw}
        &\int_{\Omega\setminus \Gamma_n(t)}|\nabla u_n(s)-\nabla w_n(s, t)|^2\,\derx
        =\int_{\Omega\setminus \Gamma_n(t)}\left(|\nabla u_n(s)|^2-|\nabla w_n(s, t)|^2\right)\,\derx. 
    \end{align}
    Now, let us note that $\left(w_n(s, t), \nabla w_n(s, t)\mathds{1}_{ \Gamma_n(t)}\right)$ is an admissible competitor in \eqref{eq:3.3} \EEE at time $s\in I_{\infty,0}$.  \EEE Therefore,    recalling \eqref{eq:lambdaepsilondefinition},  
    \begin{align}
        \label{eq:u(s)pluslambda(s)leqwnpluslambda(t)}
        &\int_{\Omega\setminus \Gamma_n(s)}|\nabla u_n(s)|^2\,\derx+\lambda_n(s)
        \leq\int_{\Omega\setminus \Gamma_n(t)}|\nabla w_n(s, t)|^2\,\derx+\lambda_n(t).
    \end{align}
Since $\Gamma_n(s)\subset \Gamma_n(t)$ by Theorem \ref{theorem:approx}(ii),  we clearly have \EEE    
    \begin{align}
        \label{eq:intwithoutGamma(t)smallerthanintwithouGamma(s)}
        &\int_{\Omega\setminus \Gamma_n(t)}|\nabla u_n(s)|^2\,\derx
        \leq\int_{\Omega\setminus \Gamma_n(s)}|\nabla u_n(s)|^2\,\derx.
    \end{align}
    Taking \eqref{eq:u(s)pluslambda(s)leqwnpluslambda(t)} and \eqref{eq:intwithoutGamma(t)smallerthanintwithouGamma(s)}, subtracting $\int_{\Omega\setminus \bruchdiscepsilonn(t)}|\nabla w_n(s, t)|^2\,\derx+\lambda_n(s)$,  and using the identity \eqref{eq:equalityuminusw}, we obtain  
    \begin{equation}
    \begin{aligned}
        \label{eq:estimateuminuswcontralambda}
        \int_{\Omega\setminus \Gamma_n(t)}|\nabla u_n(s)-\nabla w_n(s, t)|^2\,\derx
        \leq\,&\lambda_n(t)-\lambda_n(s).
    \end{aligned}
    \end{equation}
    Putting (\ref{eq:estimateuminuswcontraghn}) and (\ref{eq:estimateuminuswcontralambda}) together, we discover 
    \begin{equation*}
    \begin{aligned}
        &\int_{\Omega\setminus \Gamma_n(t)}\left|\nabla u_n(t)-\nabla u_n(s)\right|^2\,\derx
        \leq
        C\,\|\nabla g_n(t)-\nabla g_n(s)\|_{L^2(\Omega)}^2 + C(\lambda_n(t)-\lambda_n(s)).
    \end{aligned}
    \end{equation*}
 Recalling the definition of $Q_n(t)$ preceding \eqref{eq:jumpsetestimate}, in particular   \EEE  $Q_n(t)\supset  \Gamma_n(t)$, we conclude that 
    \begin{align*}
        \|\nabla u_n(t)\mathds{1}_{(Q_n(t))^c}-\nabla u_n(s)\mathds{1}_{(Q_n(t))^c}\|_{L^2(\Omega)}
        \leq\,&\int_{\Omega\setminus \Gamma_n(t)}\left|\nabla u_n(t)-\nabla u_n(s)\right|^2\,\derx \\
        \leq\,&\,
        C\,\|\nabla g_n(t)-\nabla g_n(s)\|_{L^2(\Omega)}^2 + C(\lambda_n(t)-\lambda_n(s)).
    \end{align*}
  
    Consequently, the right-hand side converges to zero for $s\to t$, since $g_n$ is absolutely continuous and $\lambda_0$, being the pointwise limit of $\lambda_n$, has a continuity point in $t$ by assumption.  
  This means that  
    \begin{align}
    \label{eq:strongconvergencenablauen}
 \lim_{s\to t}  \limsup_{n \to \infty} \Vert \nabla u_n(t)\mathds{1}_{(Q_n(t))^c}  - \nabla u_n(s)\mathds{1}_{(Q_n(t))^c}  \Vert_{L^2(\Omega)}
 = 0.
    \end{align}
We also recall that, by Lemma \ref{lem:alternativetruncation},
\begin{align}\label{eq:alternativetruncation, it is later}
\nabla  u_n(s)\mathds{1}_{(Q_n(t))^c} \rightharpoonup \nabla u(s)  \quad  \text{weakly in }L^2(\Omega;\R^d).
\end{align}
We write 
\begin{align*}
 \nabla u_n (t)\mathds{1}_{(Q_n(t))^c} - \nabla u(t) & =   \big( \nabla u_n (t)\mathds{1}_{(Q_n(t))^c}  - \nabla u_n(s)\mathds{1}_{(Q_n(t))^c} \big)   +  \big(  \nabla u_n(s)\mathds{1}_{(Q_n(t))^c}    - \nabla u(s) \big) \notag  \\ & \ \ \ +  \big( \nabla u(s) - \nabla u(t) \big).
\end{align*}
Combining \eqref{eq:strongconvergencenablauen}, \eqref{eq:alternativetruncation, it is later}, and the left continuity of $\{\tau\mapsto \nabla u(\tau)\}$ at $[0, 1]\setminus I_{\infty,0}$ with respect to the strong topology in $L^2(\Omega; \R^d)$ (see Theorem \ref{theorem:continuationtheorem}) this shows $\nabla u_n (t)\mathds{1}_{(Q_n(t))^c}\rightharpoonup\nabla u(t)$ weakly in $L^2(\Omega;\R^d)$.  This concludes the proof. \EEE  
\end{proof}


\EEE

 We close \EEE with the proof of Lemma \ref{lem:PointwiseConvergence}.
\begin{proof}[Proof of Lemma \ref{lem:PointwiseConvergence}]
    As explained below  Lemma \ref{lem:PointwiseConvergence},  $\nabla u^m_\varepsilon(s)\mathds{1}_{(\Gamma^m_\varepsilon(s))^c}\to \nabla u_\varepsilon(s)-\gamma_\varepsilon(s)$ holds already \EEE for $s\in I_{\infty,\eps} \EEE $. Therefore, we only have to prove the statement for $t\notin I_{\infty,\eps}\EEE$. Consider $t\notin I_{\infty,\eps}$ and $s\in I_{\infty,\eps}$ with $s\nearrow t$.
    We define the minimization problem 
    \begin{align*}
        J:= \inf\left\{\int_{\Omega\setminus \Gamma^m_\varepsilon(t)}|\nabla z|^2\,\derx:z\in V_h(\Omega'), \, z= g^m_\varepsilon(s) \text{ on } \Omega'\setminus\overline{\Omega}\right\}
    \end{align*}
    and denote its unique minimizer by $w_m(s, t)$. We observe that $u^m_\varepsilon (t) \EEE -w_m(s, t)$ is the minimizer of 
    \begin{align*}
        L:= \inf\left\{\int_{\Omega\setminus \Gamma^m_\varepsilon(t)}|\nabla z|^2\,\derx:z\in V_h(\Omega'), \, z=g^m_\varepsilon(t)-g^m_\varepsilon(s) \text{ on }\Omega'\setminus \overline{\Omega}\right\}
    \end{align*}
    due to minimality in the separate problems. Then, we can repeat the reasoning used \EEE in the proof of Lemma \ref{lem:convergenceoutsideIinfinity} to obtain  
    \begin{align*}
        \int_{\Omega\setminus \Gamma^m_\varepsilon(t)}|\nabla u^m_\varepsilon(t) - \nabla u^m_\varepsilon(s)|^2\,\derx \leq C\|\nabla g^m_\varepsilon(t)-\nabla g^m_\varepsilon(s)\|^2_{L^2(\Omega)}+C(\lambda^m_\varepsilon(t)-\lambda^m_\varepsilon(s)),
    \end{align*}
    where $\lambda^m_\varepsilon$ was defined   preceding \EEE \eqref{iinftyeps}. \EEE Recall that $\|\nabla g^m_\varepsilon(t)-\nabla g^m_\varepsilon(s)\|^2_{L^2(\Omega)}\to 0$ for $s\nearrow t$ due to the regularity properties of $g$. Additionally, at each  $t\notin I_{\infty,\eps}$,  the function  $\lambda_\varepsilon = \lim_{m \to \infty}\lambda^m_\varepsilon$ is continuous by construction.  \EEE Therefore, we have 
    \begin{align}\label{2ndadd}
        \lim_{s\to t}\limsup_{m\to\infty} \|\nabla u^m_\varepsilon(t)\mathds{1}_{(\Gamma^m_\varepsilon(t))^c}-\nabla u^m_\eps(s)\mathds{1}_{(\Gamma^m_\varepsilon(t))^c}\|_{L^2(\Omega)}=0.
    \end{align}
 Let $t_k\nearrow t$ be the sequence  in the definition of $u_\varepsilon(t)$ and $\gamma_\varepsilon(t)$. Using the relation $\nabla u^m_\varepsilon(t_k)\mathds{1}_{(\Gamma^m_\varepsilon(t))^c} = (\nabla u^m_\varepsilon(t_k)-\gamma^m_\varepsilon(t_k))$ by  \eqref{eq:identityyni}, we can write 
    \begin{align*}
        \nabla u^m_\varepsilon(t)\mathds{1}_{(\Gamma^m_\varepsilon(t))^c}-(\nabla u_\varepsilon(t)-\gamma_\varepsilon(t))         = \,&\nabla u^m_\varepsilon(t)\mathds{1}_{(\Gamma^m_\varepsilon(t))^c}-\nabla u^m_\varepsilon(t_k)\mathds{1}_{(\Gamma^m_\varepsilon(t))^c}\\
        &+(\nabla u^m_\varepsilon(t_k)-\gamma^m_\varepsilon(t_k))-(\nabla u_\varepsilon(t_k)-\gamma_\varepsilon(t_k))\\
        &+(\nabla u_\varepsilon(t_k)-\gamma_\varepsilon(t_k))-(\nabla u_\varepsilon(t)-\gamma_\varepsilon(t)). 
    \end{align*} 
Letting  first \EEE $m\to \infty$ and afterwards $k \to \infty$, we observe that all three addends on the right-hand side converge to zero in $L^2(\Omega')$ by \eqref{2ndadd}, the convergence of $(u^m_\varepsilon(t_k),\gamma^m_\varepsilon(t_k))$ as $m \to \infty$, and the definition of the pair $(u_\varepsilon(t),\gamma_\varepsilon(t))$. This concludes the proof.
\end{proof}

\printbibliography

\end{document}